\documentclass[12pt]{amsart}
\usepackage[T1]{fontenc}
\usepackage{amsmath}
\usepackage{amsfonts}
\usepackage{amssymb}
\usepackage[all,cmtip]{xy}           
\usepackage{bbm}
\usepackage{bbding}
\usepackage{newtxtext}
\usepackage{newtxmath}
\usepackage[shortlabels]{enumitem}
\usepackage{ifpdf}
\ifpdf
\usepackage[colorlinks,final,backref=page,hyperindex]{hyperref}
\else
\usepackage[colorlinks,final,backref=page,hyperindex,hypertex]{hyperref}
\fi
\usepackage{mathrsfs}
\usepackage{amscd}
\usepackage{tikz-cd}
\usepackage[active]{srcltx}
\usepackage{tikz}
\usetikzlibrary{calc}
\usetikzlibrary{arrows,shapes,chains}

\makeatletter

\newtheorem{thm}{Theorem}[section]
\newtheorem{prop}[thm]{Proposition}
\newtheorem{lem}[thm]{Lemma}
\newtheorem{cor}[thm]{Corollary}
\newtheorem{prop-def}{Proposition-Definition}[section]

\theoremstyle{definition}
\newtheorem{defn}[thm]{Definition}

\newtheorem{remark}[thm]{Remark}
\newtheorem{exam}[thm]{Example}

\newcommand{\nc}{\newcommand}

\nc{\delete}[1]{{}}
\nc{\mmargin}[1]{}
\nc{\Alg}{\mathrm{Alg}}
\nc{\rmH}{\mathrm{H}}
\nc{\DT}{\mathrm{DT}}
\nc{\C}{\mathrm{C}}
\nc{\ac}{\mathrm{\textup{!`}}}

\nc{\mlabel}[1]{\label{#1}}  
\nc{\mcite}[1]{\cite{#1}}  
\nc{\mref}[1]{\ref{#1}}  
\nc{\mbibitem}[1]{\bibitem{#1}} 

\delete{
	\nc{\mlabel}[1]{\label{#1}  
		{\hfill \hspace{1cm}{\bf{{\ }\hfill(#1)}}}}
	\nc{\mcite}[1]{\cite{#1}{{\bf{{\ }(#1)}}}}  
	\nc{\mref}[1]{\ref{#1}{{\bf{{\ }(#1)}}}}  
	\nc{\mbibitem}[1]{\bibitem[\bf #1]{#1}} 
	
}

 \font\cyrs=wncyr7

\nc{\vep}{\varepsilon}
\nc{\bin}[2]{ (_{\stackrel{\scs{#1}}{\scs{#2}}})}  
\nc{\binc}[2]{(\!\! \begin{array}{c} \scs{#1}\\
		\scs{#2} \end{array}\!\!)}  
\nc{\bincc}[2]{  ( {\scs{#1} \atop
		\vspace{-1cm}\scs{#2}} )}  
\nc{\oline}[1]{\overline{#1}}
\nc{\mapm}[1]{\lfloor\!|{#1}|\!\rfloor}
\nc{\bs}{\bar{S}}
\nc{\la}{\longrightarrow}
\nc{\ot}{\otimes}
\nc{\rar}{\rightarrow}
\nc{\lon }{\,\rightarrow\,}
\nc{\dar}{\downarrow}
\nc{\dap}[1]{\downarrow \rlap{$\scriptstyle{#1}$}}
\nc{\defeq}{\stackrel{\rm def}{=}}
\nc{\dis}[1]{\displaystyle{#1}}
\nc{\dotcup}{\ \displaystyle{\bigcup^\bullet}\ }
\nc{\hcm}{\ \hat{,}\ }
\nc{\hts}{\hat{\otimes}}
\nc{\hcirc}{\hat{\circ}}
\nc{\lleft}{[}
\nc{\lright}{]}
\nc{\curlyl}{\left \{ \begin{array}{c} {} \\ {} \end{array}
	\right .  \!\!\!\!\!\!\!}
\nc{\curlyr}{ \!\!\!\!\!\!\!
	\left . \begin{array}{c} {} \\ {} \end{array}
	\right \} }
\nc{\longmid}{\left | \begin{array}{c} {} \\ {} \end{array}
	\right . \!\!\!\!\!\!\!}
\nc{\ora}[1]{\stackrel{#1}{\rar}}
\nc{\ola}[1]{\stackrel{#1}{\la}}
\nc{\scs}[1]{\scriptstyle{#1}} \nc{\mrm}[1]{{\rm #1}}
\nc{\dirlim}{\displaystyle{\lim_{\longrightarrow}}\,}
\nc{\invlim}{\displaystyle{\lim_{\longleftarrow}}\,}
\nc{\dislim}[1]{\displaystyle{\lim_{#1}}} \nc{\colim}{\mrm{colim}}
\nc{\mvp}{\vspace{0.3cm}} \nc{\tk}{^{(k)}} \nc{\tp}{^\prime}
\nc{\ttp}{^{\prime\prime}} \nc{\svp}{\vspace{2cm}}
\nc{\vp}{\vspace{8cm}}
\nc{\modg}[1]{\!<\!\!{#1}\!\!>}
\nc{\intg}[1]{F_C(#1)}
\nc{\lmodg}{\!<\!\!}
\nc{\rmodg}{\!\!>\!}
\nc{\cpi}{\widehat{\Pi}}
\nc{\ssha}{{\mbox{\cyrs X}}} 
\nc{\tsha}{{\mbox{\cyrt X}}}
\nc{\shpr}{\diamond}    
\nc{\labs}{\mid\!}
\nc{\rabs}{\!\mid}

\nc{\RBO}{{\mathrm{RBO}_\lambda}}
\nc{\Sh}{\mathrm{Sh}}
\nc{\RBA}{{\mathrm{RBA}_\lambda}}
\nc{\sgn}{\mathrm{sgn}}
\nc{\rH}{\mathrm{H}}

\nc{\ad}{\mrm{ad}}
\nc{\ann}{\mrm{ann}}
\nc{\Aut}{\mrm{Aut}}
\nc{\bim}{\mbox{-}\mathsf{Bimod}}
\nc{\br}{\mrm{bre}}
\nc{\can}{\mrm{can}}
\nc{\Cont}{\mrm{Cont}}
\nc{\rchar}{\mrm{char}}
\nc{\cok}{\mrm{Coker}}
\nc{\de}{\mrm{dep}}
\nc{\dtf}{{R-{\rm tf}}}
\nc{\dtor}{{R-{\rm tor}}}

\nc{\Div}{{\mrm Div}}
\nc{\Diff}{\mrm{DA}}
\nc{\Diffl}{\mathsf{DA}_\lambda}
\nc{\diffo}{{\mathsf{DO}_\lambda}}
\nc{\alg}{\mathsf{Alg}}
\nc{\End}{\mrm{End}}
\nc{\Ext}{\mrm{Ext}}
\nc{\Fil}{\mrm{Fil}}
\nc{\Fr}{\mrm{Fr}}
\nc{\Frob}{\mrm{Frob}}
\nc{\Gal}{\mrm{Gal}}
\nc{\GL}{\mrm{GL}}
\nc{\Hom}{\mrm{Hom}}
\nc{\Hoch}{\mrm{Hoch}}

\nc{\hsr}{\mrm{H}}
\nc{\hpol}{\mrm{HP}}
\nc{\id}{\mrm{Id}}
\nc{\im}{\mrm{im}}
\nc{\Id}{\mrm{Id}}
\nc{\ID}{\mrm{ID}}
\nc{\Irr}{\mrm{Irr}}
\nc{\incl}{\mrm{incl}}
\nc{\Ker}{\mrm{Ker}}
\nc{\length}{\mrm{length}}
\nc{\NLSW}{\mrm{NLSW}}
\nc{\Lie}{\mrm{Lie}}
\nc{\mchar}{\rm char}
\nc{\mpart}{\mrm{part}}
\nc{\ql}{{\QQ_\ell}}
\nc{\qp}{{\QQ_p}}
\nc{\rank}{\mrm{rank}}
\nc{\rcot}{\mrm{cot}}
\nc{\rdef}{\mrm{def}}
\nc{\rdiv}{{\rm div}}
\nc{\rtf}{{\rm tf}}
\nc{\rtor}{{\rm tor}}
\nc{\res}{\mrm{res}}
\nc{\SL}{\mrm{SL}}
\nc{\Spec}{\mrm{Spec}}
\nc{\tor}{\mrm{tor}}
\nc{\Tr}{\mrm{Tr}}
\nc{\tr}{\mrm{tr}}
\nc{\wt}{\mrm{wt}}
\def\ot{\otimes}
\def\ac{{\rotatebox[origin=c]{180}{!}}}

\nc{\bfk}{{\bf k}}
\nc{\bfone}{{\bf 1}}
\nc{\bfzero}{{\bf 0}}
\nc{\detail}{\marginpar{\bf More detail}
	\noindent{\bf Need more detail!}
	\svp}
\nc{\gap}{\marginpar{\bf Incomplete}\noindent{\bf Incomplete!!}
	\svp}
\nc{\FMod}{\mathbf{FMod}}
\nc{\Int}{\mathbf{Int}}
\nc{\Mon}{\mathbf{Mon}}
 \nc{\sproof}{\noindent{  \textit{Sketch of Proof:} }}
\nc{\remarks}{\noindent{\bf Remarks: }}
\nc{\Rep}{\mathbf{Rep}}
\nc{\Rings}{\mathbf{Rings}}
\nc{\Sets}{\mathbf{Sets}}
\nc{\ob}{\mathsf{Ob}}
\nc\blue{\color{blue}}

\nc{\BA}{{\mathbb A}}   \nc{\CC}{{\mathbb C}}
\nc{\DD}{{\mathbb D}}   \nc{\EE}{{\mathbb E}}
\nc{\FF}{{\mathbb F}}   \nc{\GG}{{\mathbb G}}
    \nc{\LL}{{\mathbb L}}
\nc{\NN}{{\mathbb N}}   \nc{\PP}{{\mathbb P}}
\nc{\QQ}{{\mathbb Q}}   \nc{\RR}{{\mathbb R}}
\nc{\TT}{{\mathbb T}}   \nc{\VV}{{\mathbb V}}
\nc{\ZZ}{{\mathbb Z}}   \nc{\TP}{\widetilde{P}}
\nc{\m}{{\mathbbm m}}

\nc{\cala}{{\mathcal A}}    \nc{\calc}{{\mathcal C}}
\nc{\cald}{\mathcal{D}}     \nc{\cale}{{\mathcal E}}
\nc{\calf}{{\mathcal F}}    \nc{\calg}{{\mathcal G}}
\nc{\calh}{{\mathcal H}}    \nc{\cali}{{\mathcal I}}
\nc{\call}{{\mathcal L}}    \nc{\calm}{{\mathcal M}}
\nc{\caln}{{\mathcal N}}    \nc{\calo}{{\mathcal O}}
\nc{\calp}{{\mathcal P}}    \nc{\calr}{{\mathcal R}}
\nc{\cals}{{\mathcal S}}    \nc{\calt}{{\Omega}}
\nc{\calv}{{\mathcal V}}    \nc{\calw}{{\mathcal W}}
\nc{\calx}{{\mathcal X}}

\nc{\fraka}{{\mathfrak a}}
\nc{\frakb}{\mathfrak{b}}
\nc{\frakC}{\mathfrak{C}}
\nc{\frakg}{{\frak g}}
\nc{\frakl}{{\frak l}}
\nc{\fraks}{{\frak s}}
\nc{\frakB}{{\frak B}}
\nc{\frakm}{{\frak m}}
\nc{\frakM}{{\frak M}}
\nc{\frakp}{{\frak p}}
\nc{\frakW}{{\frak W}}
\nc{\frakX}{{\frak X}}
\nc{\frakS}{{\frak S}}
\nc{\frakt}{{\mathfrak{T}}}
\nc{\frakA}{{\frak A}}
\nc{\frakx}{{\frakx}}
\nc{\red}{\color{red}}
\nc{\RB}{{{}_\lambda\mathfrak{RBA}}}
\nc{\RBinfty}{{{}_\lambda\mathfrak{RBA}_\infty}}

\begin{document}

\title[Homotopy Rota-Baxter   algebras]{The minimal model of  Rota-Baxter operad  with arbitrary  weight}

\dedicatory{Dedicated to Alexander Zimmermann on the occasion of his 60th birthday}

\author{Kai Wang and Guodong Zhou}
\address{School of Mathematical Sciences, Key Laboratory of Mathematics and Engineering Applications (Ministry of Education), Shanghai Key Laboratory of PMMP,  East China Normal University,
 Shanghai 200241,
   P.R.China}

\email{wangkaimath@hotmail.com }

\email{gdzhou@math.ecnu.edu.cn}

\date{\today}

\begin{abstract} This paper investigates    Rota-Baxter   algebras of of arbitrary  weight, that is, associative algebras endowed with   Rota-Baxter operators of arbitrary  weight, from an operadic viewpoint. Denote by $\RB$  the operad of Rota-Baxter  associative algebras   of weight $\lambda$.
A homotopy cooperad is explicitly constructed,      which can be seen as the Koszul dual of $\RB$  as  it is proven  that   the  cobar construction of this homotopy cooperad  is exactly the minimal model of $\RB$. This   enables us to   introduce  the notion  of homotopy  Rota-Baxter   algebras.  The deformation complex of a Rota-Baxter   algebra and the underlying $L_\infty$-algebra structure over it are exhibited as well.

\end{abstract}

\subjclass[2010]{
18M70   
 17B38  
16E40   
}

\keywords{cohomology,   deformation complex,  homotopy cooperad, homotopy Rota-Baxter   algebra,  Koszul dual,  $L_\infty$-algebra,  minimal model, operad, Rota-Baxter   algebra }

\maketitle

 \tableofcontents

\allowdisplaybreaks

\section*{Introduction}

A general philosophy of deformation theory of mathematical structures, as evolved from ideas of Gerstenhaber, Nijenhuis, Richardson, Deligne, Schlessinger, Stasheff, Goldman,   Millson etc,  is that  the deformation theory of any given
mathematical object can be described  by
a certain differential graded (=dg) Lie algebra or more generally an $L_\infty$-algebra associated to the
mathematical object  (whose underlying complex is called the deformation complex).  This philosophy has been made into a theorem in characteristic zero by Lurie \cite{Lur}  and Pridham \cite{Pri10}, expressed in terms of infinity categories.    It is an important problem   to construct explicitly the dg Lie algebra or $L_\infty$-algebra governing  deformation theory of the mathematical object under consideration.

Another important problem   about  algebraic structures is to study their homotopy versions, just like   $A_\infty$-algebras for  usual associative algebras.    From the perspective of operad theory, specifically, the task is to formulate a cofibrant resolution for the operad of an algebraic structure.  The most desirable outcome  would be providing a minimal model of the operad governing the algebraic structure.  When this operad  is Koszul, there exists a  general theory, the so-called Koszul duality for operads \cite{GK94, GJ94}, which defines a   homotopy version of this algebraic structure via the cobar construction of the Koszul dual cooperad, which, in this case,  is  a minimal model. However, when an  operad   is NOT Koszul, essential difficulties arise and   few examples of minimal models   have been  worked out.
For instance, G\'alvez-Carrillo,  Tonks, and   Vallette \cite{GCTV12}  gave a cofibrant resolution of the Batalin-Vilkovisky operad using inhomogeneous Koszul duality theory. However, their cofibrant resolution is not minimal and in another paper of Drummond-Cole and  Vallette \cite{DCV13}, the authors succeeded in finding a minimal model which is a deformation retract of the cofibrant resolution found in the previous paper.  Dotsenko and Khoroshkin \cite{DK13} constructed cofibrant resolutions for
shuffle monomial operads by the inclusion-exclusion principle and for operads presented by a Gr\"{o}bner
basis \cite{DK10} they suggested a method to find cofibrant resolutions  by deforming those of the corresponding monomial operads.

These two problems, say, describing controlling $L_\infty$-algebras  and constructing homotopy versions,  are closed related. In fact, given a cofibrant resolution, in particular  a minimal model,  of the operad in question, one can form the deformation complex of the algebraic structure and construct its $L_\infty$-structure as explained by Kontsevich and Soibelman \cite{KS00} and  van der Laan \cite{VdL02, VdL03}.  This method has been generalised to properads by Markl \cite{Mar10}, Merkulov and Vallette \cite{MV09a, MV09b}, and to colored props by  Fr\'{e}gier, Markl, and Yau \cite{FMY09}.
However, in practice, a minimal model or a small cofibrant resolution is not known a priori.

In this paper, we resolve completely the above two problems   for  Rota-Baxter associative  algebras of arbitrary weight.
 We found the minimal model and the Koszul dual homotopy cooperad of the operad of Rota-Baxter   algebras of arbitrary weight. Using the method given in Kontsevich and Soibelman \cite{KS00} and  van der Laan \cite{VdL02, VdL03}, we    exhibit  the deformation complex as well as  its $L_\infty$-structure.


Rota-Baxter algebras, formerly referred to as Baxter algebras, emerged from Baxter's research in probability theory \cite{Bax60}. Subsequent to Baxter's work, scholars like Rota \cite{Rot69}, and Cartier \cite{Car72}, among others, delved into this area, hence the appellation ``Rota-Baxter algebras''. From 2000, renewed interest in the subject arose. 
  Connes and Kreimer  \cite{CK00} established a significant connection between Rota-Baxter algebras and mathematical physics by utilizing Hopf algebra techniques in the renormalization of quantum field theory, leading to an algebraic framework for the BPHZ renormalization process \cite{CK00}. Guo and Keigher \cite{GK00a, GK00b} realised free commutative Rota-Baxter algebras via mixable shuffles, which, under the name of quasi-shuffles, are closed related to multi-zeta values as shown by Hoffmann \cite{Hof00}. Semenov-Tian-Shansky demonstrated in \cite{Sem83} that a skew-symmetric solution to the classical Yang-Baxter equation within a Lie algebra is essentially a Rota-Baxter operator of weight zero on this Lie algebra. This correspondence extends to associative algebras, as illustrated by Aguiar \cite{Agu00} and Bai \cite{Bai07}. Moreover,  Ebrahimi-Fard discovered  that a Rota-Baxter algebra induces a dendriform algebra in the sense of Loday \cite{Lod01}.   Presently, Rota-Baxter algebras find myriad applications and connections across various mathematical domains such as combinatorics \cite{Rot95}, multiple zeta values in number theory \cite{GZ08}, operad theory \cite{Agu01, BBGN13}, and Hopf algebras \cite{CK00}. For basic theory of Rota-Baxter algebras, readers are directed to the concise introduction \cite{Guo09b} and the comprehensive monograph \cite{Guo12}. In this paper, we will introduce the concepts of cohomology of Rota-Baxter algebras and their natural homotopy versions. Exploring the applications of these concepts in the aforementioned areas related to Rota-Baxter algebras will be an intriguing question for the future.

The deformation theory  and cohomology theory of  Rota-Baxter   algebras had been absent for a long time despite   the   importance of Rota-Baxter   algebras.  Recently there are some breakthroughs in this direction.
  Tang,  Bai,  Guo and  Sheng \cite{TBGS19} developed the deformation theory  and cohomology theory of $\mathcal{O}$-operators (also called relative Rota-Baxter operators) on Lie algebras, with applications
to  Rota-Baxter Lie algebras in mind.  Das \cite{Das20}  developed a similar theory for Rota-Baxter   algebras of weight zero.  Lazarev,  Sheng and   Tang  \cite{LST21} succeeded  in establishing   deformation theory  and cohomology theory of  relative  Rota-Baxter Lie algebras of weight zero   and found  applications to triangular Lie bialgebras.   They determined the $L_\infty$-algebra that controls deformations of a relative Rota-
Baxter Lie algebra and  introduced  the notion
of a homotopy relative Rota-Baxter Lie algebra.  The same group of authors also related homotopy relative Rota-Baxter Lie algebras and triangular $L_\infty$-bialgebras via  a functorial approach to Voronov's higher derived brackets construction \cite{LST20}.  Later Das and Misha also  determined the $L_\infty$-structures controlling deformation of  Rota-Baxter operators of weight zero  on  associative algebras \cite{DM20}.
These work all concern    Rota-Baxter  operators of weight zero.

A recent paper by  Pei,  Sheng,  Tang and   Zhao \cite{PSTZ19} considered  cohomologies of crossed homomorphisms for Lie algebras and they found a dg Lie algebra controlling deformations of crossed homomorphisms. Another  exciting progress in this subject is the introduction of the  notion of Rota-Baxter Lie groups by Guo, Lang and Sheng \cite{GLS21}; as a successor to  this work, Jiang, Sheng and Zhu considered cohomology of Rota-Baxter operators of weight $1$ on Lie groups and Lie algebras and relationship between them \cite{JS21, JSZ21}.   Das \cite{Das21} investigated cohomology of Rota-Baxter operators of arbitrary weight on associative algebras.  It seems that these are the only papers which investigates Rota-Baxter operators of nonzero weight (for a related work on differential algebras of nonzero weight, see \cite{GLSZ20}).   In these papers, the authors   dealt with  the deformations of only the Rota-Baxter operators with the Lie algebra or associative algebra structure unchanged. 

The goal of the present paper is to study simultaneous deformations of Rota-Baxter operators of nonzero weight and of    associative algebra structures.  One of the reasons is that when  one structure remains undeformed, the homotopy version obtained could not be a minimal model  of the operad of Rota-Baxter Lie algebras or Rota-Baxter associative  algebras.

\medskip

Several remarks are in order.

We worked out the minimal model for Rota-Baxter   algebras  in a  way which is somehow converse to the classical approach.
  Grosso modo, our method is as follows: Given an algebraic structure on a space $V$ realised as an algebra over an operad, by considering the  formal deformations of this algebraic structure, we first  construct the deformation complex  and using ad hoc method,  find an $L_\infty$-structure on the underlying graded space of this complex such that   the Maurer-Cartan elements are in bijection with the algebraic structures on $V$;
when $V$ is graded, we define a homotopy version of this algebraic structure as Maurer-Cartan elements in the $L_\infty$-algebra constructed above; finally under suitable conditions,  we could  show that the operad governing the homotopy version is a minimal model of the original operad.
However, this paper, as a completely new version of  \cite{WZ21},  is written   in operadic language which does not reflect completely our original method. For instance, in the  previous version    \cite{WZ21}, our proof of the  constructed $L_\infty$-structure is elementary and ad hoc.
For the reader less prepared in operad theory, we suggest to  her/him the reading of  \cite{WZ21}.

It should be mentioned that there is another way to derive the $L_\infty$ structure in the literature, say, the derived bracket technique  \cite{Kos04, Vor05a, Vor05b}.
The above mentioned papers of Sheng et al.  use this method as a main tool.

It might be appropriate to explain  here the relationship of our result with the   paper of  Dotsenko and  Khoroshkin \cite{DK13}. In that paper, the authors tried to deform the minimal model of the corresponding monomial operads obtained by Gr\"{o}bner basis of  the Rota-Baxter operad and they got the generators of the operad of homotopy Rota-Baxter   algebras. It seems that it is not easy to obtain all the relations.
While our generators of homotopy Rota-Baxter   algebras are the same, we could determine all the relations in an indirect way with the aid  of  the  $L_\infty$-structure on the deformation complex  found initially using our ad hoc method. However, it is fair to say that   our method to verify the minimal model has been  inspired from Dotsenko and  Khoroshkin \cite{DK13}.  Dotsenko kindly pointed out another proof based on the paper \cite{DK13}; see Remark~\ref{rem: dotsenko}.


  There remain several   problems unsolved so far. It is still  open how to define  homotopy morphisms  between homotopy Rota-Baxter   algebras and establish a homotopy version of the descendent property for Rota-Baxter   algebras (see Propositions~\ref{Prop: new RB algebra} and \ref{Prop: new-bimodule}), that is, to found the minimal model of the corresponding coloured operad. We also need to define homotopy Rota-Baxter bimodules over homotopy Rota-Baxter   algebras.  These problems will be attacked in a forthcoming paper.

It would be an interesting problem  to give a general approach for  operated algebras in the sense of Guo \cite{Guo09a} and other algebraic structures.     There are at least two interesting  concrete problems.
  Loday \cite{Lod10} asked to develop a Koszul duality theory for differential algebras with  nonzero weight, whose  operad  is not Koszul. Using our method,  we succeeded in finding the minimal model of the operad  of differential algebras with nonzero weight \cite{CGWZ}, thus   continuing \cite{GLSZ20} and answering  the question of Loday \cite{Lod10}.
  Another problem is  to show that the (coloured) operad of  (relative) homotopy Rota-Baxter Lie algebras
introduced by   Lazarev,  Sheng and   Tang  \cite{LST21} is the minimal model of that of (relative) Rota-Baxter Lie algebras. We are working on this   project \cite{CQWZ}.

\medskip

This paper is organised as follows. In the first section, we collection relevant notions and facts about $L_\infty$-algebras and homotopy (co)operads scattered in the literature in order to fix the notations and we also present an account about classical theory of Rota-Baxter   algebras.
 In the second section
  a homotopy cooperad is introduced and in the third section,  the  cobar construction of this homotopy cooperad   is shown to be  the minimal model of the operad of Rota-Baxter   algebras of arbitrary weight, so this homotopy cooperad can be considered as the Koszul dual homotopy cooperad of the operad of Rota-Baxter   algebras.  The notion of homotopy Rota-Baxter   algebras of arbitrary weight is made explicit in the fourth section.
In the fifth section,  the deformation complex and its $L_\infty$-algebra structure are derived.
 In the sixth section, the   cohomology theory of  Rota-Baxter   algebras introduced in \cite{Das20, Das21, DM20, WZ21} is recovered and an example is also included.

Throughout this paper, let $\bfk$ be a field of characteristic $0$.  All vector spaces are defined over $\bfk$,  all  tensor products and Hom-spaces  are taken over $\bfk$. We assume that the reader is familiar with the theory of operads \cite{MSS02, LV12, BD16}.

\bigskip

\section{Preliminaries}





 \subsection{$L_\infty$-algebras and Maurer-Cartan elements}\ \label{Subsect: DGLAs and Linfinity algebras}

In this subsection, we will recall some preliminaries on differential graded Lie algebras and $L_\infty$-algebras. For more background  on differential graded Lie algebras and $L_\infty$-algebras, we refer the reader to \cite{Sta92, LS93, LM95,Get09}.


 We shall use the homological grading and employ everywhere the Koszul rule to determine  signs.
For a graded space  $V=\{V_n\}_{n\in \ZZ}$,  its suspension $sV$ is defined to be   $ (sV)_p=V_{p-1}, p\in \ZZ$  and its desuspension $s^{-1}V$ is   $ (s^{-1}V)_p=V_{p+1}$.
For homogeneous elements $x_1,\dots,x_n \in V$ and $\sigma\in S_n$, the Koszul sign $\varepsilon(\sigma;  x_1,\dots, x_n)$ is defined by
\begin{equation} \label{Eq: epsilon sign} x_1  x_2  \dots  x_n=\varepsilon(\sigma;  x_1,\dots,x_n)x_{\sigma(1)}  x_{\sigma(2)} \dots  x_{\sigma(n)}\end{equation}
in the graded symmetric algebras generated by $V$,
and introduce also \begin{equation} \label{Eq: chi sign} \chi(\sigma;  x_1,\dots,x_n)=\sgn(\sigma)\varepsilon(\sigma;  x_1,\dots,x_n),\end{equation}
where $\sgn(\sigma)$ is the sign of $\sigma$.

\begin{defn}\label{Def:L-infty}
	Let $L=\bigoplus\limits_{i\in\mathbb{Z}}L_i$ be a graded space over $\bfk$. Assume that $L$ is endowed with a family of graded linear operators $l_n:L^{\ot n}\rightarrow L, n\geqslant 1$ with  $|l_n|=n-2$, subject to  the following conditions:
	for arbitrary  $n\geqslant 1$,  $ \sigma\in S_n$ and $x_1,\dots, x_n\in L$,
	\begin{enumerate}
		\item[(i)](generalised anti-symmetry) $$l_n(x_{\sigma(1)}\ot \cdots \ot x_{\sigma(n)})=\chi(\sigma; x_1,\dots,   x_n)\ l_n(x_1 \ot \cdots  \ot x_n);$$

		\item[(ii)](generalised Jacobi identity)
		$$\sum\limits_{i=1}^n\sum\limits_{\sigma\in \Sh(i,n-i)}\chi(\sigma; x_1, \dots, x_n)(-1)^{i(n-i)}l_{n-i+1}(l_i(x_{\sigma(1)}\ot \cdots \ot x_{\sigma(i)})\ot x_{\sigma(i+1)}\ot \cdots \ot x_{\sigma(n)})=0,$$
		where    $\Sh(i,n-i)$ is the set of $(i,n-i)$
		shuffles, that is, permutations $\sigma\in S_n$ such that
$$\sigma(1)<\cdots< \sigma(i)  \ \mathrm{and}\ \sigma(i+1)<\cdots< \sigma(n).$$
	\end{enumerate}
	Then $(L,\{l_n\}_{n\geqslant1})$ is called an $L_\infty$-algebra.
\end{defn}

\begin{remark} \label{Rem: L-infinity for small n}   Let us consider the generalised Jacobi identity for   $n\leqslant 3$ with the assumption of generalised anti-symmetry.
	
	\begin{enumerate}
		\item[(i)]  $n=1$,  $l_1\circ l_1=0$, that is,  $l_1$ is a differential,

		\item[(ii)]  $n=2$, $l_1\circ l_2=l_2\circ (l_1\ot\Id+\Id\ot l_1)$, that is   $l_1$ is a derivation for  $l_2$,

		\item[(iii)] $n=3$, for homogeneous elements $x_1, x_2, x_3\in L$
		$$\begin{array}{ll} &l_2(l_2(x_1\ot x_2)\ot x_3)+(-1)^{|x_1|(|x_2|+|x_3|)} l_2(l_2(x_2\ot x_3)\ot x_1)+
			(-1)^{|x_3|(|x_1|+|x_2|)} l_2(l_2(x_3\ot x_1)\ot x_2)
			\\
			=&-\Big(l_1(l_3(x_1\ot x_2\ot x_3))+ l_3(l_1 (x_1)\ot x_2\ot x_3 )+(-1)^{|x_1|} l_3(x_1\ot l_1 (x_2)\ot x_3 )+\\
			&(-1)^{|x_1|+|x_2|} l_3(x_1\ot x_2\ot l_1 (x_3) )\Big),\end{array}$$
		that is, $l_2$ satisfies the   Jacobi identity up to homotopy.
	\end{enumerate}
	
	In particular,    if all   $l_n=0$ with  $n\geqslant 3$, then $(L,l_1,l_2)$ is just a dg Lie algebra.
	
\end{remark}

\begin{defn}\label{Def: weakly complete L-infinity algebra}\cite{ LST21}A weakly filtered $L_\infty$-algebra is a pair $(L, \mathcal{F}_\bullet L)$, where $L$ is an $L_\infty$-algebra and $\mathcal{F}_\bullet L$ is a descending filtration  $L=\mathcal{F}_1L\supset \cdots \supset \mathcal{F}_nL\supset \cdots$  of graded subspaces in $L$ satisfying
	\begin{itemize}
		\item[(1)]there exists $n\geqslant 1$ such that $l_k(L^{\ot k})\subset \mathcal{F}_kL$ holds for all $k\geqslant n$.
		\item[(2)] the graded space $L$ is complete with respect to this filtration, i.e., there is an isomorphism of graded spaces $L\cong \varprojlim\limits_n ( L/\mathcal{F}_n L)$.
		\end{itemize}
	\end{defn}

One can also define Maurer-Cartan elements in  weakly complete  $L_\infty$-algebras.
\begin{defn}\cite{ LST21}\label{Def: MC element in L infinity algebra}
	Let $(L,\{l_n\}_{n\geqslant1},\mathcal{F}_\bullet L)$ be a   weakly filtered  $L_\infty$-algebra. An element $\alpha\in L_{-1}$ is called a Maurer-Cartan element if it satisfies the Maurer-Cartan equation:
	\begin{eqnarray}\label{Eq: mc-equation}\sum_{n=1}^\infty\frac{1}{n!}(-1)^{\frac{n(n-1)}{2}} l_n(\alpha^{\ot n})=0.\end{eqnarray}

\end{defn}

\begin{prop}[Twisting procedure]\cite{ LST21}\label{Prop: deformed-L-infty}
	Given a Maurer-Cartan element $\alpha$ in the  weakly filtered  $L_\infty$-algebra $L$, one can introduce  a new $L_\infty$-structure $\{l_n^\alpha\}_{n\geqslant 1}$ on graded space $L$, where $l_n^{\alpha}: L^{\ot n}\rightarrow L$ is defined as :
	\begin{eqnarray}\label{Eq: twisted L infinity algebra} l^\alpha_n(x_1\ot \cdots\ot x_n)=\sum_{i=0}^\infty\frac{1}{i!}(-1)^{in+\frac{i(i-1)}{2}}l_{n+i}(\alpha^{\ot i}\ot x_1\ot \cdots\ot x_n),\ \forall x_1, \dots, x_n\in L,\end{eqnarray}
	The new $L_\infty$-algebra $(L, \{l_n^\alpha\}_{n\geqslant 1})$ is called the twisted $L_\infty$-algebra (by the Maurer-Cartan element $\alpha$).
\end{prop}

\begin{remark}{\label{Rmk: remakrs on L-infinity algebras}}

\begin{itemize}
	
	\item[(i)]The signs in Definition~\ref{Def: MC element in L infinity algebra} and Proposition~\ref{Prop: deformed-L-infty} are different from those appearing in \cite{LST21}, as the conventions in \cite{LST21} are essentially about $L_\infty[1]$-algebras \cite{Vor05a, Vit15}.   We refer the reader to \cite{Vit15} for the translation  between $L_\infty$-structures and $L_\infty[1]$-structures.
	\item[(ii)]  The condition of being weakly filtered ensures the convergence of the infinite sums in Definition~\ref{Def: MC element in L infinity algebra} and Proposition~\ref{Prop: deformed-L-infty}.
	
	
\end{itemize}

\end{remark}

\bigskip

\subsection{Homotopy  (co)operads}\label{Subsection: Homotopy  (co)operads}\

In this subsection, we collect  some basics  on nonsymmetric homotopy  (co)operads, as they are scattered in several references \cite{Mar96, MV09a, MV09b, DP16}; in particular, we also explain how to obtain $L_\infty$-structures from homotopy operads, in particular, from convolution homotopy operads.

Since we will work with nonsymmetric homotopy (co)operads, we omit the adjective ``nonsymmetric'' everywhere.

  Recall that a graded collection $\calo=\{\calo(n)\}_{n\geqslant 1}$ is a  family of  graded space indexed by positive integers, i.e.,  each $\calo(n)$ itself being a graded space. For any $v\in \calo(n)$, the index $n$ is called the arity of $v$. The suspension of $\calo$, denoted by $s\calo$ is defined to be the graded collection $\{s\calo(n)\}_{n\geqslant 1}$. In the same way, one has the desuspension $s^{-1}\calo$ of the graded collection $\calo$.

We need some preliminaries about    trees.
For any   tree $T$,  denote the weight (=number of vertices)  and arity (=number of leaves)  of $T$ to be $\omega(T)$ and $\alpha(T)$  respectively. A planar tree is said to be reduced if each of its vertices has at least one leaf.  As we  only consider planar reduced  trees,  we also remove the adjectives ``planar reduced'' everywhere.  Write $\mathfrak{T}$ to be the set of all   trees with weight $\geqslant 1$ and for arbitrary $n\geqslant 1$, denote by  $\mathfrak{T}(n)$   the set of trees of weight $n$.
Since  trees are planar, each vertex in a tree has a total order on its inputs which will be  drawn clockwisely.
 By the existence of the root, there is a natural induced total order on the set of all    vertices of a given tree $T\in \mathfrak{T}$, which is given by counting the   vertices  starting from the root clockwisely along the tree.   We call this order the planar order.

 Let $T'$ be a divisor of $T$. Define $T/T'$ to be the tree obtained from $T$ by replacing $T'$ by a corolla (say,  a tree with only one vertex)  of arity $\alpha(T')$. There is a natural permutation $\sigma=\sigma(T, T')\in S_{\omega(T)}$ associated to the pair  $(T, T')$ defined as follows.  Assume the ordered set $\{v_1<\dots<v_n\}$ to be the sequence of all   vertices of $T$ in the planar order and $\omega(T')=j$. Let $v'$ be the vertex in $T/T'$ corresponding to the divisor $T'$ in $T$ and the serial number of $v'$ in $T/T'$ is $i$ in the planar order (so there are $i-1$ vertices ``before'' $T'$). Then define $\sigma=\sigma(T, T')\in S_{n}$ to be the unique permutation which does not permute the vertices $v_1, \dots, v_{i-1}$, and   such that   the  ordered set $\{v_{\sigma(i)}<\dots<v_{\sigma(i+j-1)}\}$ is exactly the planar ordered set of all   vertices  of $T'$ and the ordered set $\{v_{1}<\dots<v_{i-1}<v'<v_{\sigma(i+j)}<\dots< v_{\sigma(n)}\}$ is exactly the  planar ordered   set of all   vertices  in the tree $T/T'$. 

Let $\calp=\{\calp(n)\}_{n\geqslant 1}$ be a graded collection. Let  $T\in \mathfrak{T}$ and $\{v_1<\dots<v_n\}$ be the set of all   vertices  of $T$ in the planar order. Define $\calp^{\otimes T}$ to be   $\calp(\alpha(v_1))\ot \cdots \ot \calp(\alpha(v_n))$; {  morally an element in $\calp^{\ot T}$ is a linear combination of decorated trees whose underlying tree is the tree $T$ and each vertex  $v_i$ is decorated by an element  of  $\calp(\alpha(v_i))$.}
\begin{defn} \label{Def: homotopy operad}
A homotopy operad structure on a graded collection $\calp=\{\calp(n)\}_{n\geqslant 1}$ consists of a family of operations $$\{m_T: \calp^{\ot T}\rightarrow \calp(\alpha(T))\}_{T\in \mathfrak{T}}$$ with $|m_T|=\omega(T)-2$ such that the equation	
\[\sum_{T'\subset{T}}(-1)^{i-1+jk}\sgn(\sigma(T,T'))\ m_{T/T'}\circ(\id^{\ot {i-1}}\ot m_{T'}\ot \id^{\ot k})\circ r_{\sigma(T,T')}=0\]
holds for any $T\in \mathfrak{T}$, where $T'$ runs through the set of all subtrees of $T$,  $i$ is the serial number of the vertex  $v'$ in $T/T'$, $j=\omega(T')$, $k=\omega(T)-i-j$, and
	  where $r_{\sigma(T,T')}$ denoted  the right action by $\sigma=\sigma(T,T')$, that is,  $$r_\sigma(x_1\ot \cdots\ot x_n)=\varepsilon(\sigma; x_1,\dots, x_{n})x_{\sigma(1)}\ot \cdots  \ot x_{\sigma(n)}.$$	
\end{defn}

 Given two homotopy operads, a strict morphism between them is a morphism of graded collections compatible with all operations $m_T, T\in \mathfrak T$.

 Let $\cali$ be the collection with $\cali(1)=\bfk$ and $\cali(n)=0$ for any $n\ne1$. The collection $\cali$ is endowed with a homotopy operad structure in the natural way, that is, $m_T: \cali(1)\ot \cali(1)\to \cali(1)$ is given by the identity,   when $T$ is the tree with two vertices and one unique leaf, and $m_T$ vanishes otherwise.

  A homotopy operad $\calp$ is called strictly unital if there exists a strict morphism of homotopy operads $\eta: \cali\rightarrow \calp$ such that for each $n\geqslant 1$,  the compositions  $$\calp(n)\cong \calp(n)\ot \cali(1)\xrightarrow{\id\ot \eta}\calp(n)\ot \calp(1)\xrightarrow{m_{T_{1, i}}}\calp(n)$$ and
  $$\calp(n)\cong \cali(1)\ot \calp(n)\xrightarrow{\eta\ot \id}\calp(1)\ot \calp(n)\xrightarrow{m_{T_2}} \calp(n)$$ are identity maps on $\calp(n)$, where $T_{1,i}$  with $1\leqslant i\leqslant n$ is the tree of weight $2$, arity $n$ with its second vertex having arity $1$ and connecting to the first vertex on its $i$-th leaf, and $T_2$ is the tree of weight $2$, arity $n$ whose first vertex has arity $1$. {  Furthermore, for any tree $T$ {with $\omega(T)\ne 2$}, $m_T(\Id^{\ot i-1}\ot \eta\ot \Id^{\ot \omega(T)-i})$ is required to be $0$ for all $1 \leqslant i \leqslant \omega(T)$.}
A strictly unital homotopy operad $\calp$ is called augmented if there exists a strict morphism of homotopy operads $\varepsilon: \calp\rightarrow \cali$ such that $\varepsilon \circ \eta=\id_{\cali}$.

{   If a homotopy operad $\calp$ satisfies $m_T=0$ for all $T\in \frakt$ with $\omega(T)\geqslant 3$, then $\calp$ is just a nonunital dg operad in the sense of Markl \cite{Mar08}.  Let's recall the construction of free operads here. Let $``\ |\ "$ represent the trivial tree with weight $0$. For a graded collection $\calp$, define $\calp^{\ot |\ } = \cali$.  Denote $\mathfrak{T}^+=\{|\}\cup \mathfrak{T}$. Then the unital free graded operad generated by $\calp$ is exactly the graded collection $\oplus_{T\in \mathfrak{T}^+}\calp^{\ot T}$ equipped with the inherent composition operations induced by tree grafting, where $\calp^{\ot |}$ serves as the operadic unit.}

   There is a natural  $L_\infty$-algebra  associated with a homotopy operad $\calp=\{\calp(n)\}_{n\geqslant 1}$.
 Denote $\calp^{\prod}:= \prod\limits_{n=1}^\infty\calp(n)$. For each $n\geqslant 1$, define operations $m_n=\sum\limits_{T\in \mathfrak{T}(n)}m_T: (\calp^{\prod})^{\ot n}\to \calp^{\prod} $
  and $l_n$ is the anti-symmetrization of $m_n$, i.e., $$l_n(x_1\ot \cdots \ot x_n)=\sum\limits_{\sigma\in S_n}\chi(\sigma; x_1,\dots, x_n)m_n(x_{\sigma(1)}\ot \cdots \ot x_{\sigma(n)}).$$

{
\begin{prop}\cite{VdL02, MV09a}
	Let $\calp$ be a homotopy operad. Then $(\calp^{\prod}, \{l_n\}_{n\geqslant 1})$ is an $L_\infty$-algebra. In particular, if $\calp$ is a dg operad, $\calp^{\prod}$ is just a dg Lie algebra.
\end{prop}

}
We shall need a fact about dg operads.
\begin{defn}Let $\calp$ be a (nonunital) dg operad in the sense of Markl \cite{Mar08}.
 For any $f\in \calp(m)$ and $g_1\in\calp(l_1),\dots,g_n\in \calp(l_n) $ with $1\leqslant n\leqslant m$, define
\[f\{g_1,\dots,g_n\}=\sum_{i_j\geqslant l_{j-1}+i_{j-1},n\geqslant j\geqslant 2, i_1\geqslant 1}\Big(\big((f\circ_{i_1} g_1)\circ_{i_2}g_2\big)\dots\Big)\circ_{i_n}g_n.\]
It is called the brace operation on $\calp^{\prod}$. For $f\in \calp(m), g\in \calp(m)$, define
\[[f,g]_{G}=f\{g\}-(-1)^{|f||g|}g\{f\}\in \calp(m+n-1),\]
called the Gerstenhaber bracket of $f$ and $g$.
\end{defn}
 The operation $l_2$ in the dg Lie algebra $\calp^{\prod}$ is exactly the Gerstenhaber bracket defined above.

The brace operation in a dg operad $\calp$ satisfies the following pre-Jacobi identity:
	\begin{prop}[\cite{Ger63, Get93, GV95}]
		For any homogeneous elements $f, g_1,\dots, g_m,  h_1,\dots,h_n$ in $\calp^{\prod}$, we have
		\begin{eqnarray}
			\label{Eq: pre-jacobi}	&&\Big(f \{g_1,\dots,g_m\}\Big)\{h_1,\dots,h_n\}=\\
			\notag &&\quad  \sum\limits_{0\leqslant i_1\leqslant j_1\leqslant i_2\leqslant j_2\leqslant \dots \leqslant i_m\leqslant  j_m\leqslant n}(-1)^{\sum\limits_{k=1}^m(|g_k|)(\sum\limits_{j=1}^{i_k}(|h_j|))}
			f\{h_{1, i_1},  g_1\{h_{i_1+1, j_1}\},\dots,   g_m\{h_{i_m+1, j_m} \}, h_{j_m+1,  n}\}.
		\end{eqnarray}
	
\end{prop}
In particular, we have
	\begin{eqnarray}
			\label{Eq: pre-jacobi1}
(f \{g\})\{h\}=f\{g\{h\}\}+f\{g, h\}+(-1)^{|g||h|}f\{h, g\}
\end{eqnarray}

Given a dg operad $\calp$, for each tree $T$, one can define  the composition $m_\calp^T:\calp^{\ot T}\to \calp(\alpha(T))$ in $\calp$ along  $T$ as follows:  for $\omega(T)=1$, $m_\calp^T=\Id$; for  $\omega(T)=2$, $m_\calp^T=m_T$;
 for $\omega(T)\geqslant 3$, write $T$ as the grafting  of a subtree $T'$,   whose vertex set is that of $T$ except the last one, with the corolla whose unique vertex is exactly the last vertex of $T$ in the planar order, then define
 $m_\calp^T=m_{T/T'}\circ (m_\calp^{T'}\ot \Id) $, where $m_\calp^{T'}$ is obtained by induction.

\bigskip

Dualizing the definition of homotopy operads, one has the notion of homotopy cooperads.

\begin{defn}Let $\calc=\{\calc(n)\}_{n\geqslant 1}$ be a graded collection. A homotopy cooperad structure on $\calc$ consists of a family of operations $\{\Delta_T: \calc(n)\rightarrow \calc^{\ot T }\}_{T\in\mathfrak{T}}$ with $|\Delta_T|=\omega(T)-2$ such that for any $c\in \calc$, $\Delta_T(c)=0$ for almost all but finitely many $T\in\frakt$ , and  the family of operations $\{\Delta_T\}_{T\in\frakt}$ satisfies the following identity:
	$$\sum_{T'\subset T}\sgn(\sigma(T,T')^{-1})(-1)^{i-1+jk} r_{\sigma(T,T')^{-1}}\circ (\id^{\ot i-1}\ot \Delta_{T'}\ot \id^{\ot k})\circ  \Delta_{T/T'}=0$$
 for any $T\in\frakt$, where  $T', i, j, k$ have the same meanings as for  homotopy operads.
\end{defn}

The graded collection   $\cali$ has a natural  homotopy cooperad structure, that is, $\Delta_T: \cali(1)\to  \cali(1)\ot \cali(1)$ is given by the identity,   when $T$ is the tree with two vertices and one unique leaf, and $\Delta_T$ vanishes otherwise.
A homotopy cooperad $\calc$ is called strictly counital if there exists a strict morphism of homotopy cooperad $\varepsilon: \calc\rightarrow \cali$ such that $$\calc(n)\xrightarrow{\Delta_{T_{1, i}}}\calc(n)\ot \calc(1)\xrightarrow{\id\ot \varepsilon}\calc(n)\ot \cali(1)\cong \calc(n)$$ and $$\calc(n)\xrightarrow{T_2}\calc(1)\ot \calc(n)\xrightarrow{\varepsilon\ot \id}\cali(1)\ot \calc(n)\cong \calc(n)$$ are identity maps on $\calc(n)$, where $T_{1,i}$  with $1\leqslant i\leqslant n$  and $T_2$ are the notations used before. {  Additionally,  for any planar tree $T$ with $\alpha(T)=n,\omega(T)=m\ne 2, 1\leqslant j\leqslant m$, the  composition \\{\small
$\calc(n)\xrightarrow{\Delta_T} \calc(\alpha(v_1))\ot \dots\ot  \calc(\alpha(v_m))\xrightarrow{\id^{\ot j-1}\ot \varepsilon\ot \id^{\ot m-j}}$ $$\quad  \quad\quad\quad\quad\quad\quad\quad \quad\quad\quad\quad \calc(\alpha(v_1))\ot \cdots \ot \calc(\alpha(v_{j-1}))\ot \cali(1)\ot \calc(\alpha(v_{j+1}))\ot \cdots  \ot \calc(\alpha(v_m))$$  }   is required to be $0$, where $v_1,\cdots, v_m$ are vertices of the tree $T$.}

  A homotopy cooperad $\calc$ is called coaugmented if there exists  a strict morphism of homotopy cooperads $\eta:\cali\rightarrow \calc$ such that $\varepsilon\circ \eta=\id_{\cali}$. For a coaugmented homotopy cooperad $\calc$, the graded collection $\overline{\calc}=\Ker(\varepsilon)$ endowed with operations $\{\overline{\Delta}_T\}_{T\in\frakt}$ is naturally a homotopy cooperad, where $\overline{\Delta}_T$ is the  the restriction of operation $\Delta_T$ on $\overline{\calc}$.

A homotopy cooperad  $\cale=\{\cale(n)\}_{n\geqslant 1}$  such that   $\{\Delta_T\}$ vanish for all $\omega(T)\geqslant 3$ is  exactly a noncounital dg cooperad in the sense of Markl \cite{Mar08}.

For  a  (noncounital) dg cooperad $\cale$, on can define  the cocomposition  $\Delta^{T}_\cale:\cale(\alpha(T))\to \cale^{\ot T}$  along a tree $T$ in the dual way as the composition $m^T_\calp$ along $T$ for a dg operad $\calp$.


\begin{prop-def}\cite{MV09a} Let $\calc$ be a homotopy cooperad and $\cale$ be a dg cooperad. Then the graded collection $\calc\ot \cale$ with $(\calc\ot\cale)(n):=\calc(n)\ot \cale(n), \forall n\geqslant 1$ has a natural structure of homotopy cooperad as follows:
	\begin{itemize}
		\item[(i)] For any tree $T\in \frakt$ of weight $1$ and arity $n$, $$\Delta_T^H(c\ot e)=\Delta_T^\calc(c)\ot e+(-1)^{|c|}c\ot d_{\cale}e$$ for arbitrary homogeneous elements $c\in \calc(n), e\in \cale(n)$;
		\item[(ii)]for any tree $T$ of weight $n\geqslant 2$, define $$\Delta_T^H(c\ot e)=(-1)^{\sum\limits_{k=1}^{n-1}\sum\limits_{j=k+1}^n|e_k||c_j|}(c_1\ot e_1)\ot \cdots \ot (c_n\ot e_n)\in (\calc\ot \cale )^{\ot T},$$ with $c_1\ot \cdots\ot c_n=\Delta_T^\calc(c)\in \calc^{\ot T}$  and $e_1\ot \cdots \ot e_n=\Delta^T_\cale(e)\in \cale^{\ot T}$, where $\Delta^T_\cale$ is the cocomposition  in $\cale$  along   $T$.
	\end{itemize}
The new homotopy cooperad  is called the Hadamard product of $\calc$ and $\cale$,  and denoted by $\calc\ot_{\rH}\cale$.
	\end{prop-def}

Define $\cals=\mathrm{End}_{\bfk s}^c$ to be the graded cooperad whose underlying graded collection is given as $\cals(n)=\Hom((\bfk s)^{\ot n},\bfk s),n\geqslant 1$. Denote $\delta_n$ to be the map in $\cals(n)$ which takes $s^{\ot n}$ to $s$. The cooperad structure is given as $$\Delta_T(\delta_n)=(-1)^{(j-1)(i-1)}\delta_{n-i+1}\ot \delta_i\in (\cals)^{\ot T} $$ for any tree $T$ of weight $2$ whose second vertex is connected with the $j$-the leaf of its first vertex. We also define $\cals^{-1}$ to be the graded cooperad whose underlying graded collection is $\cals^{-1}(n)=\Hom((\bfk s^{-1})^{\ot n},s^{-1})$ for all $n\geqslant 1$ and the cooperad structure is given as $$\Delta_T(\varepsilon_n)=(-1)^{(i-1)(n-i+1-j)}\varepsilon_{n-i+1}\ot \varepsilon_i\in (\cals^{-1})^{\ot T},$$ where $\varepsilon_n\in \cals^{-1}(n)$ is the map which takes $(s^{-1})^{\ot n}$ to $s^{-1}$ and $T$ is the same as before.
 It is easy to see $\cals\ot_{\mathrm{H}}\cals^{-1}\cong \cals^{-1}\ot_{\mathrm{H}}\cals=:\mathbf{As}^\vee$. Notice that for any homotopy cooperad $\calc$, we have $ \calc\ot_\rH \mathbf{As}^\vee\cong\calc\cong \mathbf{As}^\vee\ot_\rH \calc.$

\begin{defn}Let $\calc$ be a homotopy cooperad. Define the operadic suspension (resp. desuspension)  of $\calc$ to be the homotopy cooperad $\calc\ot _{\mathrm{H}} \cals$ (resp. $\calc\ot_{\mathrm{H}}\cals^{-1}$),  denoted as $\mathscr{S}\calc$ (resp. $\mathscr{S}^{-1}\calc$).
	\end{defn}

\begin{defn}
	Let $\calc=\{\calc(n)\}_{n\geqslant 1}$ be a coaugmented homotopy cooperad. The cobar construction of $\calc$, denoted by $\Omega\calc$,  is the free graded operad generated by the graded collection $s^{-1}\overline{\calc}$, endowed with the differential $\partial$ which is lifted from $\partial: s^{-1}\overline{\calc}\to \Omega\calc$ given by   for any $f\in \overline{\calc}(n)$,
	$$\partial(s^{-1}f)=-\sum_{T\in \mathfrak{T}(n)} (s^{-1})^n \circ \Delta_T(f),$$
 \end{defn}

This provides an alternative definition for   homotopy operads. In fact, a   graded collection  $\overline{\calc}=\{\overline{\calc}(n)\}_{n\geqslant 1}$ is a     homotopy cooperad if and only if   the free graded operad generated by $s^{-1}\overline{\calc}$ {    (also called the cobar construction of $\calc=\overline{C}\oplus  \cali$)} is endowed with a differential such that it becomes a dg operad.

\begin{prop}\cite{VdL02,MV09a}
    	Let $\calc$ be a homotopy cooperad and $\calp$ be a dg operad. Then the graded collection $\mathbf{Hom}(\calc,\calp)=\{\Hom(\calc(n), \calp(n))\}_{n\geqslant 1}$ has a natural homotopy operad structure, which is defined in the following way:
	 \begin{itemize}
	 	\item[(i)] For any $T\in \mathfrak{T}$ with $\omega(T)=1$, $$m_T(f)=d_{\calp}\circ f-(-1)^{|f|}f\circ \Delta_T$$ for  $f\in\mathbf{Hom}(\calc,\calp)(n)$;
	 	\item[(ii)] for any $T\in \mathfrak{T}$ with $\omega(T)\geqslant 2$, $$m_T(f_1\ot \cdots \ot f_n)=(-1)^{\frac{n(n-1)}{2}+1+n\sum_{i=1}^n |f_i|}m_{\calp}^T\circ(f_1\ot \cdots\ot f_n)\circ \Delta_T,$$  where $m_\calp^T$ is the composition in $\calp$ along   $T$.
	 \end{itemize}
\end{prop}

\begin{prop}\cite{VdL02,MV09a}\label{Prop: Linfinity give MC}
	Let $\calc$ be a coaugmented homotopy cooperad and $\calp$ be a unital dg operad. Then there is a natural bijection:
	\[\Hom_{udgOp}(\Omega\calc, \calp)\cong \calm\calc\Big(\mathbf{Hom}(\overline{\calc},\calp)^{\prod}\Big),\]
	where the {  left-hand} side is the set of morphisms of unital dg operads from $\Omega C$ to $\calp$ and the {  right-hand} side is the set of Maurer-Cartan elements in the $L_\infty$-algebra $\mathbf{Hom}(\overline{\calc},\calp)^{\prod}$.
\end{prop}

\subsection{Rota-Baxter   algebras and their bimodules}\
\label{Rota-Baxter   algebras and their bimodules}

\begin{defn}\label{Def: Rota-Baxter   algebra}
Let $(A, \mu=\cdot)$ be an associative algebra over field $\bfk$ and
$\lambda\in \bfk$. A linear
operator $T: A\rightarrow A$ is said to be a Rota-Baxter operator of
weight $\lambda$ if it satisfies
\begin{eqnarray}\label{Eq: Rota-Baxter relation}
T(a)\cdot T(b)=T\big(a\cdot T(b)+T(a)\cdot b+\lambda\  a\cdot
b\big)\end{eqnarray}	
for any $a,b \in A$, or in terms of maps
\begin{eqnarray}\label{Eq: Rota-Baxter relation in terms of maps}
\mu\circ (T \ot T)=T\circ (\Id\ot T +T \ot \Id)+\lambda\ T\circ \mu.\end{eqnarray} In this case,  $(A,\mu,T)$ is called a Rota-Baxter
algebra of weight $\lambda$. Denote by $\RBA$ the category of Rota-Baxter   algebras of weight $\lambda$ with obvious morphisms.
\end{defn}

\begin{remark}\label{Remark: Koszul duality not applied to RB algebras}

The  operad for Rota-Baxter   algebras of weight $\lambda$, denoted by $\RB$,  is generated by a unary operator $T$ and a binary operator $\mu$ with   operadic relations generated by $$ \mu\circ_1\mu-\mu\circ_2\mu\quad \mathrm{and}\quad (\mu\circ_1T)\circ_2T-(T\circ_1\mu)\circ_1T-(T\circ_1\mu)\circ_2T-\lambda T\circ_1\mu.$$
It is obvious that  the operad $\RB$ is not a Koszul operad, not even  a quadratic operad.

Recall that  for a Koszul operad $\calp$, the minimal model of $\calp$ is exactly $\Omega \calp^{\ac}$, the cobar construction of its Koszul dual cooperad $\calp^\ac$. However, since  the operad $\RB$ is not Koszul, its Koszul dual ${\RB^\ac}$ should be  a homotopy cooperad, rather than a genuine cooperad; we will exhibit this homotopy cooperad in  Section~\ref{Section: Koszul dual cooperad}.

\end{remark}


\begin{defn}\label{Def: Rota-Baxter bimodules}
Let $(A,\mu,T)$ be a Rota-Baxter   algebra and $M$ be a bimodule
over associative algebra $(A,\mu)$. We say that $M$ is a bimodule
over Rota-Baxter   algebra $(A,\mu, T)$  or a Rota-Baxter bimodule if  $M$ is endowed with a
linear operator $T_M: M\rightarrow M$ such that the following
equations
\begin{eqnarray}T(a) T_M(m)&=&T_M\big(aT_M(m)+T(a)m+\lambda am\big),\\
T_M(m)T(a)&=&T_M\big(mT(a)+T_M(m)a+\lambda ma\big).
\end{eqnarray}
hold for any $a\in A$ and $m\in M$.
\end{defn}
Of course, $(A,\mu, T)$ itself is a bimodule over the Rota-Baxter   algebra $(A,\mu,T)$, called the regular Rota-Baxter bimodule.



Rota-Baxter   algebras and Rota-Baxter bimodules have some descendent properties.
\begin{prop}[{\cite[Theorem 1.1.17]{Guo12}}]\label{Prop: new RB algebra}
	Let $(A,\mu,T)$ be a Rota-Baxter   algebra. Define a new binary operation as:
\begin{eqnarray}a\star  b:=a\cdot T(b)+T(a)\cdot b+\lambda a\cdot b\end{eqnarray}
for any $a,b\in A$. Then
\begin{itemize}

\item[(i)]   the operation $\star $ is associative and  $(A,\star )$ is a new  associative algebra;

    \item[(ii)] the triple  $(A,\star ,T)$ also forms a Rota-Baxter   algebra of weight $\lambda$  and denote it by $A_\star $;

\item[(iii)] the map $T:(A,\star , T)\rightarrow (A,\mu, T)$ is a  morphism of Rota-Baxter  algebras.
    \end{itemize}
	\end{prop}
One can also construct new Rota-Baxter bimodules from old ones.
\begin{prop}\label{Prop: new-bimodule}
Let $(A,\mu,T)$ be a Rota-Baxter   algebra of weight $\lambda$ and $(M,T_M)$ be a Rota-Baxter bimodule over it. We define a left action $``\rhd "$ and a right action $``\lhd"$ of $A$ on $M$ as follows: for any $a\in A,m\in M$,
\begin{eqnarray}
a\rhd m:&=& T(a)m-T_M(am),\\
m\lhd a:&=& mT(a)-T_M(ma).
\end{eqnarray}
Then these actions make $M$ into a Rota-Baxter bimodule over $A_\star $ and denote this new  bimodule by $_\rhd M_\lhd$.
\end{prop}
The easy proof of the above result is left to the reader.

\bigskip

\section{The Koszul dual homotopy cooperad}\label{Section: Koszul dual cooperad}

In this section,   we will   construct a homotopy cooperad which can be considered as the Koszul dual of the operad for Rota-Baxter   algebras, and we will prove that the cobar construction of this homotopy cooperad is exactly the minimal model for the operad of Rota-Baxter   algebras in Section~\ref{Section: minimal model}.




 Define a graded collection $\mathscr{S}(\RB^\ac)$ by   $$\mathscr{S}(\RB^\ac)(n)=\bfk u_n\oplus \bfk v_n, \ \mathrm{with}\   |u_n|=0, |v_n|=1$$ for $n\geqslant 1$. Now, we put a coaugmented homotopy cooperad structure on   $\mathscr{S}(\RB^\ac)$. Firstly, consider  trees of arity $n\geqslant 1$ in the following list:
\begin{itemize}[keyvals]
	
	\item[(I)] Trees of weight two:     for each $1\leqslant j\leqslant n$ and   $1\leqslant i\leqslant n-j+1$, there exists such a tree, which can be visualized as
	
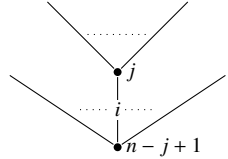
\begin{eqnarray*}
		\begin{tikzpicture}[scale=0.75,descr/.style={fill=white}]
			\tikzstyle{every node}=[thick,minimum size=3pt, inner sep=1pt]
			\node(r) at (0,-0.5)[minimum size=0pt,circle]{};
			\node(v0) at (0,0)[fill=black, circle,label=right:{\tiny $ n-j+1$}]{};
			\node(v1-1) at (-1.5,1){};
			\node(v1-2) at(0,1)[fill=black,circle,label=right:{\tiny $\tiny j$}]{};
			\node(v1-3) at(1.5,1){};
			\node(v2-1)at (-1,2){};
			\node(v2-2) at(1,2){};
			\draw(v0)--(v1-1);
			\draw(v0)--(v1-3);
			\draw(v1-2)--(v2-1);
			\draw(v1-2)--(v2-2);
			\draw[dotted](-0.4,1.5)--(0.4,1.5);
			\draw[dotted](-0.5,0.5)--(-0.1,0.5);
			\draw[dotted](0.1,0.5)--(0.5,0.5);
			\path[-,font=\scriptsize]
			(v0) edge node[descr]{{\tiny$i$}} (v1-2);
		\end{tikzpicture}
	\end{eqnarray*}
	
	\item[(II)] Trees of weight $k+1\geqslant 3$ and of ``height'' two: in these trees,  the first vertex in the planar order has arity $k$, the vertex connected to the $t$-th leaf of the first vertex has arity $r_t$ for each $1\leqslant t\leqslant k$ (so $2\leqslant k\leqslant n$ and $r_1+\cdots +r_k=n$); these trees have the following picture:
	
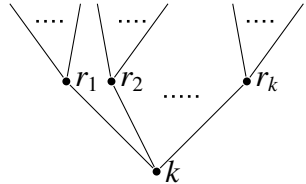
\begin{eqnarray*}
		\begin{tikzpicture}[scale=0.8,descr/.style={fill=white}]
			\tikzstyle{every node}=[thick,minimum size=3pt, inner sep=1pt]
			\node(v0) at (0,-1.5)[circle, fill=black,label=right:$k$]{};
			\node(v1) at(-1.2,-0.3)[circle, fill=black,label=right:$r_1$]{};
			\node(v1-1) at(-2,0.8){};
			\node(v1-2) at (-1,0.8){};
			\node(v3) at (1.2,-0.3)[circle, fill=black,label=right:$r_k$]{};
			\node(v3-1) at (1,0.8){};
			\node(v3-2) at (2,0.8){};
			\node(v2-1) at (-0.5,-0.3) [circle,fill=black,label=right:$r_2$]{};
			\node(v2-1-1) at (-0.8,0.8){};
			\node(v2-1-2) at (0.2, 0.8){};
			\draw [dotted,line width=0.5pt] (0,-0.6)--(0.6,-0.6);
			\draw [dotted,line width=0.5pt] (-0.5,0.5)--(-0.1,0.5);
			\draw [dotted,line width=0.5pt] (-1.6,0.5)--(-1.2, 0.5);
			\draw [dotted, line width=0.5pt] (1.2,0.5)--(1.6,0.5);
			\draw        (v0)--(v1);
			\draw         (v0)--(v3);
			\draw         (v1)--(v1-1);
			\draw          (v1)--(v1-2);
			\draw        (v2-1)--(v2-1-1);
			\draw        (v2-1)--(v2-1-2);
			\draw        (v3)--(v3-1);
			\draw        (v3)--(v3-2);
            \draw        (v0)--(v2-1);
		\end{tikzpicture}
	\end{eqnarray*}
	 where $k\geqslant 2$, $r_t\geqslant 1$ for all $k\geqslant t\geqslant 1$.
	\item[(III)] Trees of weight $q+1\geqslant 3$ and of ``height'' three and there exists a unique vertex in the first two levels: in these trees, there are numbers $2\leqslant q\leqslant p\leqslant n$, $1\leqslant k_1<\dots<k_{q-1}\leqslant p$, $1\leqslant i\leqslant r_1$ and $r_j\geqslant 1$ for all $1\leqslant j\leqslant q$ such that the first vertex has arity $r_1$, the second vertex has arity $p$ and is connected to the $i$-th leaf of the first vertex and the other vertices connected to the $k_1$-th (resp. $k_2$-th, $\dots, k_{q-1}$-th) leaf of the second vertex and with arity $r_2$ (resp. $r_3, \dots, r_q$), so $r_1+\cdots+r_q+p-q=n$; these trees are drawn in this way:
	
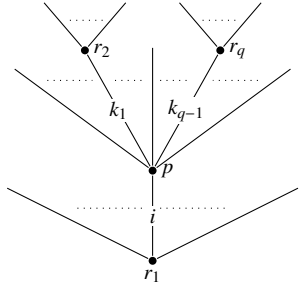
\begin{eqnarray*}
		\begin{tikzpicture}[scale=0.8,descr/.style={fill=white}]
			\tikzstyle{every node}=[thick,minimum size=3pt, inner sep=1pt]
			\node(v0) at (0,0)[circle,fill=black,label=below:{\tiny $r_1$}]{};
			\node(v1-1) at (-2,1){};
			\node(v1-2) at(0,1.2)[circle,fill=black,label=right:{\tiny $p$}]{};
			\node(v1-3) at(2,1){};
			\node(v2-1) at(-1.9,2.45){};
			\node(v2-2) at (-0.9, 2.8)[circle,fill=black,label=right:{\tiny $r_2$}]{};
			\node(v2-3) at (0,2.9){};
			\node(v2-4) at(0.9,2.8)[circle,fill=black,label=right:{\tiny $r_q$}]{};
			\node(v2-5) at(1.9,2.45){};
			\node(v3-1) at (-1.5,3.5){};
			\node(v3-2) at (-0.3,3.5){};
			\node(v3-3) at (0.3,3.5){};
			\node(v3-4) at(1.5,3.5){};
			\draw(v0)--(v1-1);
			\draw(v0)--(v1-3);
			\path[-,font=\scriptsize]
			(v0) edge node[descr]{{\tiny$i$}} (v1-2);
			\draw(v1-2)--(v2-1);
			\draw(v1-2)--(v2-3);
			\draw(v1-2)--(v2-5);
			\path[-,font=\scriptsize]
			(v1-2) edge node[descr]{{\tiny$k_1$}} (v2-2)
			edge node[descr]{{\tiny$k_{q-1}$}} (v2-4);
			\draw(v2-2)--(v3-1);
			\draw(v2-2)--(v3-2);
			\draw(v2-4)--(v3-3);
			\draw(v2-4)--(v3-4);
			\draw[dotted](-1,0.7)--(-0.1,0.7);
			\draw[dotted](0.1,0.7)--(1,0.7);
			\draw[dotted](-0.5,2.4)--(-0.1,2.4);
			\draw[dotted](0.1,2.4)--(0.5,2.4);
			\draw[dotted](-1.4,2.4)--(-0.8,2.4);
			\draw[dotted](1.4,2.4)--(0.8,2.4);
			\draw[dotted](-1.1,3.2)--(-0.6,3.2);
			\draw[dotted](1.1,3.2)--(0.6,3.2);
		\end{tikzpicture}
	\end{eqnarray*}
\end{itemize}

Now, we define a family of operations $\{\Delta_T: \mathscr{S}(\RB^\ac)\rightarrow \mathscr{S}({\RB^\ac})^{\ot T}\}_{T\in \frakt}$ as follows:
\begin{itemize}
	\item[(1)] For a tree $T$ of type $\mathrm{(I)}$ with $1\leqslant j\leqslant n, 1\leqslant i\leqslant n-j+1$,   define $\Delta_T(u_n)=u_{n-j+1}\ot u_j$, which can be drawn as
	\begin{eqnarray*}
		\begin{tikzpicture}[scale=0.8,descr/.style={fill=white}]
			\tikzstyle{every node}=[thick,minimum size=5pt, inner sep=1pt]
			\node(r) at (0,-0.5)[minimum size=0pt,rectangle]{};
			\node(v-2) at(-2,0.5)[minimum size=0pt, label=left:{$\Delta_T(u_n)=$}]{};
			\node(v0) at (-0.5,0)[draw,rectangle]{{\small $u_{n-j+1}$}};
			\node(v1-1) at (-2,1){};
			\node(v1-2) at(-0.5,1)[draw,rectangle]{\small$u_j$};
			\node(v1-3) at(1,1){};
			\node(v2-1)at (-1.5,2){};
			\node(v2-2) at(0.5,2){};
			\draw(v0)--(v1-1);
			\draw(v0)--(v1-3);
			\draw(v1-2)--(v2-1);
			\draw(v1-2)--(v2-2);
			\draw[dotted](-0.9,1.5)--(-0.1,1.5);
			\draw[dotted](-1,0.5)--(-0.5,0.5);
			\draw[dotted](-0.5,0.5)--(0,0.5);
			\path[-,font=\scriptsize]
			(v0) edge node[descr]{{\tiny$i$}} (v1-2);
		\end{tikzpicture}
	\end{eqnarray*}
and
$$\Delta_T(v_n)=\left\{\begin{array}{ll} v_n\ot u_1, & j=1,\\
\lambda^{j-1}v_{n-j+1}\ot u_j, & 2\leqslant j\leqslant n-1,\\
\lambda^{n-1}v_{1}\ot u_n+u_1\ot v_n, & j=n, \end{array}\right.$$
which can be pictured as
	\begin{eqnarray*}
		\begin{tikzpicture}[scale=0.9,descr/.style={fill=white}]
			\tikzstyle{every node}=[thick,minimum size=5pt, inner sep=1pt]
			\node(r) at (0,-0.5)[minimum size=0pt,rectangle]{};
			\node(v-1) at(-2,0.5)[minimum size=0pt, label=left:{when $j=1$, $\Delta_T(v_n)=$}]{};
			\node(v0) at (-0.5,0)[draw,rectangle]{\small$v_n$};
			\node(v1-1) at (-1.8,1){};
			\node(v1-2) at(-0.5,1)[draw,rectangle]{\small$u_1$};
			\node(v1-3) at(0.8,1){};
			\node(v2-1)at (-0.5,1.8){};
			\draw(v0)--(v1-1);
			\draw(v0)--(v1-3);
			\draw(v1-2)--(v2-1);
			\draw[dotted](-1,0.5)--(-0.6,0.5);
			\draw[dotted](-0.4,0.5)--(0,0.5);
			\path[-,font=\scriptsize]
			(v0) edge node[descr]{{\tiny$i$}} (v1-2);
		\end{tikzpicture}
	\end{eqnarray*}
	\begin{eqnarray*}
		\begin{tikzpicture}[scale=0.8,descr/.style={fill=white}]
			\tikzstyle{every node}=[thick,minimum size=5pt, inner sep=1pt]
			\node(r) at (0,-0.5)[minimum size=0pt,rectangle]{};
			\node(v-2) at(-3,0.5)[minimum size=0pt, label=left:{when $2\leqslant j\leqslant n-1$,\  $\Delta_T(v_n)=$}]{};
			\node(v-1) at(-1.75,0.5)[minimum size=0pt,label=left:{$\lambda^{j-1}$}]{};
			\node(v0) at (-0.5,0)[draw, rectangle]{$v_{n-j+1}$};
			\node(v1-1) at (-2,1){};
			\node(v1-2) at(-0.5,1.25)[draw,rectangle]{\small $u_j$};
			\node(v1-3) at(1,1){};
			\node(v2-1)at (-1.5,2.25){};
			\node(v2-2) at(0.5,2.25){};
			\draw(v0)--(v1-1);
			\draw(v0)--(v1-3);
			\draw(v1-2)--(v2-1);
			\draw(v1-2)--(v2-2);
			\draw[dotted](-0.9,1.75)--(-0.1,1.75);
			\draw[dotted](-1,0.5)--(-0.6,0.5);
			\draw[dotted](-0.4,0.5)--(0,0.5);
			\path[-,font=\scriptsize]
			(v0) edge node[descr]{{\tiny$i$}} (v1-2);
		\end{tikzpicture}
	\end{eqnarray*}
	\begin{eqnarray*}
		\begin{tikzpicture}[scale=0.8,descr/.style={fill=white}]
			\tikzstyle{every node}=[thick,minimum size=5pt, inner sep=1pt]
			\node(r) at (0,-0.5)[minimum size=0pt,rectangle]{};
			\node(va) at(-3,0.5)[minimum size=0pt, label=left:{when $j=n$, \ $\Delta_T(v_n)=$}]{};
			\node(vb) at(-2,0.5)[minimum size=0pt,label=left:{$\lambda^{n-1}$}]{};
			\node(vc) at (-1.5,0)[draw, rectangle]{\small $v_1$};
			\node(v1) at(-1.5,1)[draw,rectangle]{\small $u_n$};
			\node(v2-1)at (-2.5,2){};
			\node(v2-2) at(-0.5,2){};
			\node(vd) at(0,0.5)[minimum size=0, label=right:$+$]{};
			\node(ve) at (2,0)[draw, rectangle]{\small $u_1$};
			\node(ve1) at (2,1)[draw,rectangle]{\small $v_n$};
			\node(ve2-1) at(1,2){};
			\node(ve2-2) at(3,2){};
			\draw(v1)--(v2-1);
			\draw(v1)--(v2-2);
			\draw[dotted](-1.9,1.5)--(-1.1,1.5);
			\draw(vc)--(v1);
			\draw(ve)--(ve1);
			\draw(ve1)--(ve2-1);
			\draw(ve1)--(ve2-2);
			\draw[dotted](1.6,1.5)--(2.4,1.5);
		\end{tikzpicture}
	\end{eqnarray*}
	
	\item[(2)] For a tree $T$ of type $\mathrm{(II)}$ with  $2\leqslant k\leqslant n,   r_1+\cdots +r_k=n, r_1, \dots, r_k\geqslant 1$,  define
	\begin{eqnarray*}
		\begin{tikzpicture}[scale=0.8,descr/.style={fill=white}]
			\tikzstyle{every node}=[thick,minimum size=5pt, inner sep=1pt]
			\node(v-2) at (-5,-0.5)[minimum size=0pt, label=left:{$\Delta_T(v_n)=$}]{};
			\node(v-1) at(-5,-0.4)[minimum size=0pt,label=right:{$(-1)^{\frac{k(k-1)}{2}}$}]{};
			\node(v0) at (-1,-1.5)[rectangle, draw]{\small $u_k$};
			\node(v1) at(-2.2,-0.3)[rectangle,draw]{\small $v_{r_1}$};
			\node(v1-1) at(-3,0.8){};
			\node(v1-2) at (-2,0.8){};
			\node(v3) at (0.2,-0.3)[rectangle, draw]{\small $v_{r_k}$};
			\node(v3-1) at (0,0.8){};
			\node(v3-2) at (1,0.8){};
			\node(v2-1) at (-1.4,-0.3) [rectangle,draw]{\small $v_{r_2}$};
			\node(v2-1-1) at (-1.8,0.8){};
			\node(v2-1-2) at (-0.8, 0.8){};
			\draw [dotted,line width=0.5pt] (-1,-0.6)--(-0.4,-0.6);
			\draw [dotted,line width=0.5pt] (-1.5,0.5)--(-1.1,0.5);
			\draw [dotted,line width=0.5pt] (-2.6,0.5)--(-2.2, 0.5);
			\draw [dotted, line width=0.5pt] (0.2,0.5)--(0.6,0.5);
			\draw        (v0)--(v1);
			\draw         (v0)--(v3);
            \draw         (v0)--(v2-1);
			\draw         (v1)--(v1-1);
			\draw          (v1)--(v1-2);
			\draw        (v2-1)--(v2-1-1);
			\draw        (v2-1)--(v2-1-2);
			\draw        (v3)--(v3-1);
			\draw        (v3)--(v3-2);
		\end{tikzpicture}
	\end{eqnarray*}
	\item[(3)] For a tree $T$ of type $\mathrm{(III)}$ with $2\leqslant q\leqslant p\leqslant n, 1\leqslant k_1<\dots<k_{q-1}\leqslant p, r_1+\cdots +r_q+p-q=n, 1\leqslant i\leqslant r_1, r_1, \dots, r_q\geqslant 1$, define
	\begin{eqnarray*}
		\begin{tikzpicture}[scale=0.8,descr/.style={fill=white}]
			\tikzstyle{every node}=[minimum size=4pt, inner sep=1pt]
			\node(v-2) at (-5,1.2)[minimum size=0pt, label=left:{$\Delta_T(v_n)=$}]{};
			\node(v-1) at(-5,1.3)[minimum size=0pt,label=right:{$(-1)^\frac{q(q-1)}{2}\lambda^{p-q}$}]{};
			\node(v0) at (0,0)[rectangle,draw]{\small $v_{r_1}$};
			\node(v1-1) at (-2,1){};
			\node(v1-2) at(0,1.2)[rectangle,draw]{\small $u_p$};
			\node(v1-3) at(2,1){};
			\node(v2-1) at(-1.9,2.4){};
			\node(v2-2) at (-0.9, 2.8)[rectangle,draw]{\small$v_{r_2}$};
			\node(v2-3) at (0,2.9){};
			\node(v2-4) at(0.9,2.8)[rectangle,draw]{\small $v_{r_q}$};
			\node(v2-5) at(1.9,2.4){};
			\node(v3-1) at (-1.5,3.5){};
			\node(v3-2) at (-0.3,3.5){};
			\node(v3-3) at (0.3,3.5){};
			\node(v3-4) at(1.5,3.5){};
			\draw(v0)--(v1-1);
			\draw(v0)--(v1-3);
			\path[-,font=\scriptsize]
			(v0) edge node[descr]{{\tiny$i$}} (v1-2);
			\draw(v1-2)--(v2-1);
			\draw(v1-2)--(v2-3);
			\draw(v1-2)--(v2-5);
			\path[-,font=\scriptsize]
			(v1-2) edge node[descr]{{\tiny$k_1$}} (v2-2)
			edge node[descr]{{\tiny$k_{q-1}$}} (v2-4);
			\draw(v2-2)--(v3-1);
			\draw(v2-2)--(v3-2);
			\draw(v2-4)--(v3-3);
			\draw(v2-4)--(v3-4);
			\draw[dotted](-1,0.7)--(-0.1,0.7);
			\draw[dotted](0.1,0.7)--(1,0.7);
			\draw[dotted](-0.5,2.3)--(-0.1,2.3);
			\draw[dotted](0.1,2.3)--(0.5,2.3);
			\draw[dotted](-1.4,2.3)--(-0.8,2.3);
			\draw[dotted](1.4,2.3)--(0.8,2.3);
			\draw[dotted](-1.1,3.2)--(-0.6,3.2);
			\draw[dotted](1.1,3.2)--(0.6,3.2);
		\end{tikzpicture}
	\end{eqnarray*}
	\item[(4)]All other components of $\Delta_T, T\in \frakt$ vanish.
\end{itemize}

\begin{prop}{\label{Prop: Homotopy cooperad }}
The graded collection $\mathscr{S}(\RB^\ac)$ endowed with the  operations $\{\Delta_T\}_{T\in \frakt}$ introduced above forms a coaugmented homotopy cooperad, whose  strict counit is the natural projection $\varepsilon:\mathscr{S}(\RB^\ac)\twoheadrightarrow \bfk u_1\cong \cali$ and the coaugmentation is just the natural embedding $\eta:\cali\cong \bfk u_1\hookrightarrow \mathscr{S}(\RB^\ac)$. 	
\end{prop}
\begin{proof} One needs to show that the induced derivation $\partial$ on the cobar construction of $\mathscr{S}(\RB^\ac)$, i.e., the free operad generated by $s^{-1}\overline{\mathscr{S}(\RB^\ac)}$ is a differential, that is, $\partial^2=0$.

 Denote $s^{-1}u_n, n\geqslant 2$ (resp. $s^{-1}v_n, n\geqslant 1$) by $x_n$  (resp. $y_n$) which are  the generators of $\Omega\big(\mathscr{S}(\RB^\ac)\big)$. Notice that $|x_n|=-1$ and $|y_n|=0$. By the definition of cobar construction of coaugmented homotopy cooperads,   the action of differential $\partial$ on generators $x_n, y_n$ is given by the following formulas:
	for $n\geqslant 2$, \begin{eqnarray}\label{Eq: partial xn}
		\partial(x_n)&=&-\sum\limits_{j=2}^{n-1}x_{n-j+1}\{x_j\};\end{eqnarray}
and for $n\geqslant 1$,
		\begin{eqnarray}\label{Eq: partial yn} &&\\
		\nonumber	\ \ \partial(y_n)&=&-\sum_{k=2}^{n}\sum_{r_1+\dots+r_k=n,\atop r_1,\dots,r_k\geqslant 1}x_k\{y_{r_1},\dots,y_{r_k}\}+\sum_{2\leqslant p\leqslant n,\atop 1\leqslant q\leqslant p}\sum_{r_1+\dots+r_q+p-q=n,\atop r_1,\dots,r_q\geqslant1}\lambda^{p-q}y_{r_1}\big\{x_p\{y_{r_2},\dots,y_{r_{q}}\}\big\}.
	\end{eqnarray}
Note that \begin{eqnarray}\label{Eq: partial y1} \partial(y_1)=0.\end{eqnarray}

We just need to see that $\partial^2=0$ holds on generators $x_n,n\geqslant 2$ and $y_n,n\geqslant 1$, which can be checked by direct computations.
{   Firstly, note that the suboperad in $\Omega \mathscr{S}(\RB^\ac)$ generated by the collection $\{x_n\}_{n\geqslant 2}$ is exactly the  operadic suspension of the dg operad that governs $A_\infty$-algebras \cite{LV12}. {\rm Equation~\eqref{Eq: partial xn}} precisely represents the formulas for the differential on generators in this dg operad. Therefore, the equation $\partial^2(x_n) = 0$ holds true.}

 We also  have
   \begin{align*}
&\partial^2(y_n)\\
\stackrel{\eqref{Eq: partial yn}}{=}&\partial\Big(-\sum_{k=2}^{n}\sum_{r_1+\dots+r_k=n,\atop r_1,\dots,r_k\geqslant 1}x_k\{y_{r_1},\dots,y_{r_k}\}+\sum_{2\leqslant p\leqslant n,\atop 1\leqslant q\leqslant p}\sum_{r_1+\dots+r_q+p-q=n,\atop r_1,\dots,r_q\geqslant1}\lambda^{p-q}y_{r_1}\big\{x_p\{y_{r_2},\dots,y_{r_{q}}\}\big\}\Big).\end{align*}
As $\partial$ is a derivation and by \eqref{Eq: partial y1},
\begin{align*}
&\partial^2(y_n)\\
=&-\sum_{k=2}^{n}\sum_{r_1+\dots+r_k=n,\atop r_1,\dots,r_k\geqslant 1}\partial(x_k)\{y_{r_1},\dots,y_{r_k}\}
+\sum_{k=2}^{n-1}\sum_{r_1+\dots+r_k=n,\atop r_1,\dots,r_k\geqslant 1}\sum_{j=1}^k x_k\{y_{r_1},\dots,\partial(y_{r_j}),\dots,y_{r_k}\}\\
&+\sum_{2\leqslant p\leqslant n-1,\atop 1\leqslant q\leqslant p}\sum_{r_1+\dots+r_q+p-q=n,\atop r_1,\dots,r_q\geqslant1}\lambda^{p-q}\partial(y_{r_1})\big\{x_p\{y_{r_2},\dots,y_{r_{q}}\}\big\}\\
&+\sum_{2\leqslant p\leqslant n,\atop 1\leqslant q\leqslant p}\sum_{r_1+\dots+r_q+p-q=n,\atop r_1,\dots,r_q\geqslant1}\lambda^{p-q}y_{r_1}\big\{\partial(x_p)\{y_{r_2},\dots,y_{r_{q}}\}\big\}\\
&-\sum_{2\leqslant p\leqslant n-1,\atop 1\leqslant q\leqslant p}\sum_{r_1+\dots+r_q+p-q=n,\atop r_1,\dots,r_q\geqslant1}\sum_{j=2}^q\lambda^{p-q}y_{r_1}\big\{x_p\{y_{r_2},\dots,\partial(y_{r_j}),\dots,y_{r_{q}}\}\big\}\\
\stackrel{\eqref{Eq: partial xn}\eqref{Eq: partial yn}}{=}
&\underbrace{\sum_{k=2}^{n}\sum_{r_1+\dots+r_k=n \atop r_1,\dots,r_k\geqslant 1}\sum_{i=2}^{k-1}\big(x_{k-i+1}\{x_i\}\big)\{y_{r_1},\dots,y_{r_k}\}}_{\mathrm{(I)}}\\
&- \underbrace{\sum_{k=2}^{n-1}\sum_{r_1+\dots+r_k=n \atop r_1,\dots,r_k\geqslant 1} \sum_{j=1}^k \sum_{u=2}^{r_j} \sum_{l_1+\dots+l_u=r_j,\atop l_1,\dots,l_u\geqslant 1}x_k\{y_{r_1},\dots,   x_u\{y_{l_1}, \dots, y_{l_u}\},\dots,y_{r_k}\}}_{\mathrm{(II)}}\\
&+\underbrace{\sum_{k=2}^{n-1}\sum_{r_1+\dots+r_k=n \atop r_1,\dots,r_k\geqslant 1} \sum_{j=1}^k\sum_{2\leqslant p\leqslant r_j,\atop 1\leqslant q\leqslant p}\sum_{l_1+\dots+l_q+p-q=r_j,\atop l_1,\dots,l_q\geqslant 1}\lambda^{p-q}       x_k\{y_{r_1},\dots,   y_{l_1}\big\{x_p\{y_{l_2},\dots,y_{l_{q}}\}\big\},\dots,y_{r_k}\}}_{\mathrm{(III)}}\\
&-\underbrace{\sum_{2\leqslant p\leqslant n-1,\atop 1\leqslant q\leqslant p}\sum_{r_1+\dots+r_q+p-q=n,\atop r_1,\dots,r_q\geqslant1} \sum_{k=2}^{r_1}\sum_{l_1+\dots+l_k=r_1 \atop  l_1, \dots, l_k\geqslant 1 } \lambda^{p-q}\Big(x_k\big\{y_{l_1},\dots,y_{l_k}\big\}\Big)\Big\{x_p\{y_{r_2},\dots,y_{r_{q}}\}\Big\}}_{\mathrm{(IV)}}\\
&+\tiny{\underbrace{\sum_{2\leqslant p\leqslant n-1,\atop 1\leqslant q\leqslant p}\!\sum_{r_1+\dots+r_q+p-q=n,\atop r_1,\dots,r_q\geqslant1}\! \sum_{r_1+\dots+r_j+k_1-j_1=r_1,\atop k_1\geqslant j_1\geqslant 1,k_1\geqslant 2}\! \sum_{2\leqslant k\leqslant r_1,\atop 1\leqslant j\leqslant k}\sum_{l_1+\dots+l_j+k-j=r_1,\atop l_1,\dots,l_j\geqslant 1}\!\!\!
\lambda^{p-q+k-j}\Big(y_{l_1}\big\{x_{k}\{y_{l_2},\dots, y_{l_{j}}\}\big\}\Big)\Big\{x_p\{y_{r_2},\dots,y_{r_{q}}\}\Big\}}_{\mathrm{(V)}}}\\
&-\underbrace{\sum_{2\leqslant p\leqslant n,\atop 1\leqslant q\leqslant p}\sum_{r_1+\dots+r_q+p-q=n,\atop r_1,\dots,r_q\geqslant1}\sum_{2\leqslant i\leqslant p-1}\lambda^{p-q}y_{r_1}\Big\{\big(x_{p-i+1}\big\{x_i\big\}\big)\big\{y_{r_2},\dots,y_{r_{q}}\big\}\Big\}}_{\mathrm{(VI)}} \\
&+\underbrace{\sum_{2\leqslant p\leqslant n-1\atop 1\leqslant q\leqslant p} \sum_{r_1+\dots+r_q+p-q=n,\atop r_1,\dots,r_q\geqslant1}\sum_{j=2}^q\sum_{k=2}^{r_j} \sum_{l_1+\dots+l_k=r_j\atop l_1,\dots,l_k\geqslant1}\lambda^{p-q}y_{r_1}\big\{x_p\{y_{r_2},\dots,x_{k}\{y_{l_1},\dots,y_{l_{k}}\},\dots,y_{r_{q}}\}\big\}}_{\mathrm{(VII)}}\\
&-\tiny{\underbrace{\sum_{2\leqslant p\leqslant n-1 \atop 1\leqslant q\leqslant p} \sum_{r_1+\dots+r_q+p-q=n \atop r_1,\dots,r_q\geqslant 1}\sum_{j=2}^q \sum_{2\leqslant k\leqslant r_j,\atop 1\leqslant s\leqslant k} \sum_{l_1+\dots+l_s+k-s=r_j\atop l_1,\dots,l_s\geqslant1}\lambda^{p-q+k-s} y_{r_1}\big\{x_p\{y_{r_2},\dots,y_{l_1}\{x_k\{y_{l_2},\dots,y_{l_s}\}\},\dots, y_{r_{q}}\}\big\}}_{\mathrm{(VIII)}}  }\\
 \end{align*}

By the pre-Jacobi identity \eqref{Eq: pre-jacobi}, we have
$$
	\mathrm{(I)}=\mathrm{(II)},
	\mathrm{(III)}=\mathrm{(IV)}\  \mathrm{and}\
	\mathrm{(V)+(VII)}=\mathrm{(VI)+(VIII)},
$$
 so we obtain   $\partial^2(y_n)=0$.

	\end{proof}

We will justify the following  definition by showing its cobar construction is exactly the minimal model of  $\RB$ in Section~\ref{Section: minimal model}, hence the name ``Koszul dual homotopy cooperad''.
\begin{defn}
	The homotopy cooperad $\mathscr{S}(\RB^\ac)\ot_{\mathrm{H}} \cals^{-1}$ is called the Koszul dual homotopy cooperad of $\RB$,  denoted by $\RB^\ac$.
\end{defn}

 Precisely, the underlying graded collection of ${\RB^\ac}$ is $${\RB^\ac}(n)=\bfk e_n\oplus \bfk o_n, n\geqslant 1$$ with $e_n=u_n\ot \varepsilon_n$ and $o_n=v_n\ot \varepsilon_n$, thus $|e_n|=n-1$ and $|o_n|=n$.
 The   defining operations $\{\Delta_T\}_{T\in \frakt}$  are given by the following formulae:
\begin{itemize}
	\item[(1)] For a tree $T$ of type $\mathrm{(I)}$ with $1\leqslant j\leqslant n, 1\leqslant i\leqslant n-j+1$,
	\begin{eqnarray*}
		\begin{tikzpicture}[scale=0.75,descr/.style={fill=white}]
			\tikzstyle{every node}=[thick,minimum size=5pt, inner sep=1pt]
			\node(r) at (0,-0.5)[minimum size=0pt,rectangle]{};
			\node(v-2) at(-1.8,0.5)[minimum size=0pt, label=left:{$\Delta_T(e_n)=(-1)^{(j-1)(n-i+1)}$}]{};
			\node(v0) at (-0.5,0)[draw,rectangle]{{\small $e_{n-j+1}$}};
			\node(v1-1) at (-2,1){};
			\node(v1-2) at(-0.5,1.3)[draw,rectangle]{\small$e_j$};
			\node(v1-3) at(1,1){};
			\node(v2-1)at (-1.5,2.3){};
			\node(v2-2) at(0.5,2.3){};
			\draw(v0)--(v1-1);
			\draw(v0)--(v1-3);
			\draw(v1-2)--(v2-1);
			\draw(v1-2)--(v2-2);
			\draw[dotted](-0.9,1.8)--(-0.1,1.8);
			\draw[dotted](-1,0.5)--(-0.5,0.5);
			\draw[dotted](-0.5,0.5)--(0,0.5);
			\path[-,font=\scriptsize]
			(v0) edge node[descr]{{\tiny$i$}} (v1-2);
		\end{tikzpicture}
	\end{eqnarray*}
	  We distinguish three cases when introducing $\Delta_T(o_n)$. For  $j=1$,
	\begin{eqnarray*}
		\begin{tikzpicture}[scale=0.75,descr/.style={fill=white}]
			\tikzstyle{every node}=[thick,minimum size=5pt, inner sep=1pt]
			\node(r) at (0,-0.5)[minimum size=0pt,rectangle]{};
			\node(v-1) at(-2,0.5)[minimum size=0pt, label=left:{$\Delta_T(o_n)=$}]{};
			\node(v0) at (-0.5,0)[draw,rectangle]{\small$o_n$};
			\node(v1-1) at (-1.8,1){};
			\node(v1-2) at(-0.5,1.3)[draw,rectangle]{\small$e_1$};
			\node(v1-3) at(0.8,1){};
			\node(v2-1)at (-0.5,2.1){};
			\draw(v0)--(v1-1);
			\draw(v0)--(v1-3);
			\draw(v1-2)--(v2-1);
			\draw[dotted](-1,0.5)--(-0.5,0.5);
			\draw[dotted](-0.5,0.5)--(0,0.5);
			\path[-,font=\scriptsize]
			(v0) edge node[descr]{{\tiny$i$}} (v1-2);
		\end{tikzpicture}
	\end{eqnarray*}
		for  $2\leqslant j\leqslant n-1$,
	\begin{eqnarray*}
		\begin{tikzpicture}[scale=0.75,descr/.style={fill=white}]
			\tikzstyle{every node}=[thick,minimum size=5pt, inner sep=1pt]
			\node(r) at (0,-0.5)[minimum size=0pt,rectangle]{};
			\node(v-2) at(-7.5,0.5)[minimum size=0pt, label=left:{$\Delta_T(o_n)=$}]{};
			\node(v-1) at(-3,0.5)[minimum size=0pt,label=left:{$(-1)^{(j-1)(n-i+1)}\lambda^{j-1}$}]{};
			\node(v0) at (-1.5,0)[draw, rectangle]{$o_{n-j+1}$};
			\node(v1-1) at (-3,1){};
			\node(v1-2) at(-1.5,1.2)[draw,rectangle]{\small $e_j$};
			\node(v1-3) at(0,1){};
			\node(v2-1)at (-2.5,2.2){};
			\node(v2-2) at(-0.5,2.2){};
			\draw(v0)--(v1-1);
			\draw(v0)--(v1-3);
			\draw(v1-2)--(v2-1);
			\draw(v1-2)--(v2-2);
			\draw[dotted](-1.9,1.7)--(-1.1,1.7);
			\draw[dotted](-2.1,0.5)--(-1.6,0.5);
			\draw[dotted](-1.4,0.5)--(-0.9,0.5);
			\path[-,font=\scriptsize]
			(v0) edge node[descr]{{\tiny$i$}} (v1-2);
		\end{tikzpicture}
	\end{eqnarray*}
	and for  $j=n$,
	\begin{eqnarray*}
		\begin{tikzpicture}[scale=0.75,descr/.style={fill=white}]
			\tikzstyle{every node}=[thick,minimum size=5pt, inner sep=1pt]
			\node(r) at (0,-0.5)[minimum size=0pt,rectangle]{};
			\node(va) at(-3.5,0.5)[minimum size=0pt, label=left:{$\Delta_T(o_n)=\ $}]{};
			\node(vb) at(-2.5,0.5)[minimum size=0pt,label=left:{$\lambda^{n-1}$}]{};
			\node(vc) at (-2,0)[draw, rectangle]{\small $o_1$};
			\node(v1) at(-2,1)[draw,rectangle]{\small $e_n$};
			\node(v2-1)at (-3,2){};
			\node(v2-2) at(-1,2){};
			\node(vd) at(-0.5,0.5)[minimum size=0, label=right:$+$]{};
			\node(ve) at (1.5,0)[draw, rectangle]{\small $e_1$};
			\node(ve1) at (1.5,1)[draw,rectangle]{\small $o_n$};
			\node(ve2-1) at(0.5,2){};
			\node(ve2-2) at(2.5,2){};
			\draw(v1)--(v2-1);
			\draw(v1)--(v2-2);
			\draw[dotted](-2.4,1.5)--(-1.6,1.5);
			\draw(vc)--(v1);
			\draw(ve)--(ve1);
			\draw(ve1)--(ve2-1);
			\draw(ve1)--(ve2-2);
			\draw[dotted](1.1,1.5)--(1.9,1.5);
		\end{tikzpicture}
	\end{eqnarray*}
	
	\item[(2)]  For a tree $T$ of type $\mathrm{(II)}$ with  $2\leqslant k\leqslant n,   r_1+\cdots +r_k=n, r_1, \dots, r_k\geqslant 1$,
		\begin{eqnarray*}
		\begin{tikzpicture}[scale=0.8,descr/.style={fill=white}]
			\tikzstyle{every node}=[thick,minimum size=5pt, inner sep=1pt]
			\node(v-2) at (-5,-0.5)[minimum size=0pt, label=left:{$\Delta_T(o_n)=$}]{};
			\node(v-1) at(-5,-0.4)[minimum size=0pt,label=right:{$(-1)^{\frac{k(k-1)}{2}}$}]{};
			\node(v0) at (-1,-1.5)[rectangle, draw]{\small $e_k$};
			\node(v1) at(-2.2,-0.3)[rectangle,draw]{\small $o_{r_1}$};
			\node(v1-1) at(-3,0.8){};
			\node(v1-2) at (-2,0.8){};
			\node(v3) at (0.2,-0.3)[rectangle, draw]{\small $o_{r_k}$};
			\node(v3-1) at (0,0.8){};
			\node(v3-2) at (1,0.8){};
			\node(v2-1) at (-1.3,-0.3) [rectangle,draw]{\small $o_{r_2}$};
			\node(v2-1-1) at (-1.8,0.8){};
			\node(v2-1-2) at (-0.8, 0.8){};
			\draw [dotted,line width=0.5pt] (-0.9,-0.6)--(-0.3,-0.6);
			\draw [dotted,line width=0.5pt] (-1.5,0.5)--(-1.1,0.5);
			\draw [dotted,line width=0.5pt] (-2.6,0.5)--(-2.2, 0.5);
			\draw [dotted, line width=0.5pt] (0.2,0.5)--(0.6,0.5);
			\draw        (v0)--(v1);
			\draw         (v0)--(v3);
            \draw         (v0)--(v2-1);
			\draw         (v1)--(v1-1);
			\draw          (v1)--(v1-2);
			\draw        (v2-1)--(v2-1-1);
			\draw        (v2-1)--(v2-1-2);
			\draw        (v3)--(v3-1);
			\draw        (v3)--(v3-2);
		\end{tikzpicture}
	\end{eqnarray*}
	\item[(3)]  For a tree $T$ of type $\mathrm{(III)}$ with $2\leqslant q\leqslant p\leqslant n, 1\leqslant k_1<\dots<k_{q-1}\leqslant p, r_1+\cdots +r_q+p-q=n, 1\leqslant i\leqslant r_1, r_1, \dots, r_q\geqslant 1$,
	\begin{eqnarray*}
		\begin{tikzpicture}[scale=0.8,descr/.style={fill=white}]
			\tikzstyle{every node}=[thick,minimum size=5pt, inner sep=1pt]
			\node(v-2) at (-5,1.2)[minimum size=0pt, label=left:{$\Delta_T(o_n)=$}]{};
			\node(v-1) at(-5,1.2)[minimum size=0pt,label=right:{$(-1)^\gamma\lambda^{p-q}$}]{};
			\node(v0) at (-0.5,0)[rectangle,draw]{\small $o_{r_1}$};
			\node(v1-1) at (-2.5,1){};
			\node(v1-2) at(-0.5,1.2)[rectangle,draw]{\small $e_p$};
			\node(v1-3) at(1.5,1){};
			\node(v2-1) at(-2.4,2.45){};
			\node(v2-2) at (-1.4, 2.8)[rectangle,draw]{\small$o_{r_2}$};
			\node(v2-3) at (-0.5,2.9){};
			\node(v2-4) at(0.4,2.8)[rectangle,draw]{\small $o_{r_q}$};
			\node(v2-5) at(1.4,2.45){};
			\node(v3-1) at (-2,3.5){};
			\node(v3-2) at (-0.8,3.5){};
			\node(v3-3) at (-0.2,3.5){};
			\node(v3-4) at(1,3.5){};
			\draw(v0)--(v1-1);
			\draw(v0)--(v1-3);
			\path[-,font=\scriptsize]
			(v0) edge node[descr]{{\tiny$i$}} (v1-2);
			\draw(v1-2)--(v2-1);
			\draw(v1-2)--(v2-3);
			\draw(v1-2)--(v2-5);
			\path[-,font=\scriptsize]
			(v1-2) edge node[descr]{{\tiny$k_1$}} (v2-2)
			edge node[descr]{{\tiny$k_{q-1}$}} (v2-4);
			\draw(v2-2)--(v3-1);
			\draw(v2-2)--(v3-2);
			\draw(v2-4)--(v3-3);
			\draw(v2-4)--(v3-4);
			\draw[dotted](-1.5,0.7)--(-0.6,0.7);
			\draw[dotted](-0.5,0.7)--(0.5,0.7);
			\draw[dotted](-1,2.3)--(-0.5,2.3);
			\draw[dotted](0.7,2.3)--(1.1,2.3);
			\draw[dotted](-1.9,2.3)--(-1.3,2.3);
			\draw[dotted](-0.5,2.3)--(0.5,2.3);
			\draw[dotted](-1.6,3.2)--(-1.1,3.2);
			\draw[dotted](0.6,3.2)--(0.1,3.2);
		\end{tikzpicture}
	\end{eqnarray*}
	where \begin{eqnarray*}
		\gamma=\sum_{j=1}^{q-1}(q-j)r_j+(p-1)(q-1)+(\sum_{j=2}^qr_j+p-q)(r_1-i)+\sum_{j=2}^q(r_j-1)(p-k_{j-1})	\end{eqnarray*}
	\item[(4)]All other components of $\Delta_T, T\in \frakt$ vanish.
\end{itemize}

\section{The minimal model}\label{Section: minimal model}

In  the previous section, we   constructed a coaugmented homotopy cooperad ${\RB^\ac}$. Now, we will prove that  its cobar construction   $\Omega({\RB^\ac})$ is exactly the minimal model for the operad $\RB$.

\medskip

 Firstly, let's recall the notion of minimal model of operads.
 \begin{defn}[\cite{DCV13}] \label{Def: Minimal model of operads} A minimal model for an operad $\mathcal{P}$  is a quasi-free dg operad $ (\mathcal{F}(\calm),d)$ together with a surjective quasi-isomorphism of operads $(\mathcal{F}(\calm), \partial)\overset{\sim}{\twoheadrightarrow}\calp$, where the dg operad $(\mathcal{F}(\calm),  \partial)$  satisfies the following conditions:
 	\begin{itemize}
 		\item[(i)] the differential $\partial$ is decomposable, i.e. ,  $\partial$ takes $\calm$ to $\mathcal{F}(\calm)^{\geqslant 2}$, the subspace of $\mathcal{F}(\calm)$ consisting of elements with weight $\geqslant 2$;
 		\item[(ii)] the generating collection $\calm$ admits a decomposition $\calm=\bigoplus\limits_{i\geqslant 1}\calm_{(i)}$  such that $\partial(\calm_{(k+1)})\subset \mathcal{F}\Big(\bigoplus\limits_{i=1}^k\calm_{(i)}\Big)$ for any $k\geqslant 1$.  \end{itemize}
 \end{defn}
 \begin{thm}[\cite{DCV13}] When an operad $\calp$ admits a minimal model, it is unique up to isomorphisms.
 \end{thm}

{
\begin{defn}  The differential graded operad $\Omega (\RB^\ac)$, denoted by $\RBinfty$, is called the operad of homotopy Rota-Baxter algebras of weight $\lambda$.
\end{defn}
}

Let us make this dg operad explicit.
The  dg operad $\RB_\infty$ is the free operad generated by the graded collection $s^{-1}\overline{{\RB^\ac}}$ endowed with the differential induced from the homotopy cooperad structure.
More precisely, $$s^{-1}\overline{{\RB^\ac}}(1)=\bfk s^{-1}o_1 \ \mathrm{and}\ s^{-1}\overline{{\RB^\ac}}(n)=\bfk s^{-1}e_n\oplus \bfk s^{-1}o_n, n\geqslant2.$$
Denote $m_n=s^{-1}e_n, n\geqslant 2$ and $T_n=s^{-1}o_n, n\geqslant 1$  respectively, so  $|m_n|=n-2, |T_n|=n-1$. The action of the  differential on these generators in $\RB_\infty$ is given by the following formulae:
\begin{eqnarray}\label{Eq: defining HRB 1} \forall n\geqslant 2,\quad
	\partial{m_n} =  \sum_{j=2}^{n-1}\sum_{i=1}^{n-j+1}(-1)^{i+j(n-i)}m_{n-j+1}\circ_i m_j
\end{eqnarray} and  for any $ n\geqslant 1$,
{\small
$$\begin{array}{ll} \label{Eq: defining HRB 2}   &  \\
	\notag  &     \partial T_n\\
\notag=&\sum\limits_{k=2}^n\sum_{l_1+\cdots+l_k=n\atop l_1, \dots, l_k\geqslant 1}(-1)^{\alpha}\Big(\cdots\big((m_k\circ_1 T_{l_1})\circ_{l_1+1}T_{l_2}\big)\cdots\Big)\circ_{l_1+\cdots+l_{k-1}+1}T_{l_k}+ \sum\limits_{{\small\substack{2\leqslant p\leqslant n \\ 1\leqslant q\leqslant p
	}}}\sum\limits_{\small\substack{ r_1+\dots+r_q+p-q=n\\r_1, \dots, r_q\geqslant 1\\1\leqslant i\leqslant r_1\\1\leqslant k_1<\dots< k_{q-1}\leqslant p }}\\
	\notag     & (-1)^{\beta}\lambda^{p-q}\Big(T_{r_1}\circ_i \Big(\big((\cdots(( m_p\circ_{k_1}T_{r_2})\circ_{k_2+r_2-1} T_{r_3}))\cdots \big)\circ_{k_{q-1}+r_2+\dots+r_{q-1}-q+2}T_{r_q}\Big)\Big),
\end{array}$$
}
where the signs $(-1)^{\alpha}$ and $(-1)^{\beta}$ are   respectively
\begin{eqnarray*}\label{Eq: sign   alpha'}
	\alpha&=&1+\frac{k(k-1)}{2}+\sum_{j=1}^k(k-j)l_j=1+\sum_{j=1}^k(k-j)(l_j-1),\\
	\label{Eq: sign   beta'}\beta &=& 1+i+\big(p+\sum\limits_{j=2}^q(r_j-1)\big)\big(r_1-i\big)+\sum\limits_{j=2}^q(r_j-1)(p-k_{j-1}).	\end{eqnarray*}

Let us  display  the elements in the dg operad $\RBinfty$ using labelled planar rooted trees.
The corolla with $n$ leaves and a black vertex   represents the  generator  $m_n, n\geqslant 2$, while the generators $T_n, n\geqslant 1$  by   the corolla with $n$ leaves and a white vertex:
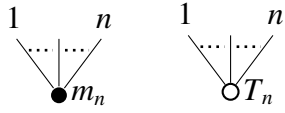
\begin{eqnarray*}
	\begin{tikzpicture}[scale=0.4]
		\tikzstyle{every node}=[thick,minimum size=4pt, inner sep=1pt]
		\node[circle, fill=black, label=right:$m_n$] (b0) at (-2,-0.5)  {};
		\node (b1) at (-3.5,1.5)  [minimum size=0pt,label=above:$1$]{};
		 \node (b2) at (-2,1.5)  [minimum size=0pt  ]{};
		\node (b3) at (-0.5,1.5)  [minimum size=0pt,label=above:$n$]{};
		\draw        (b0)--(b1);
		\draw        (b0)--(b2);
		\draw        (b0)--(b3);
		\draw [dotted,line width=0.5pt] (-3,1)--(-2.2,1);
		\draw [dotted,line width=0.5pt] (-1.8,1)--(-1,1);
	\end{tikzpicture}
	\hspace{8mm}
	\begin{tikzpicture}[scale=0.4]
		\tikzstyle{every node}=[thick,minimum size=4pt, inner sep=1pt]
		\node[circle, draw, label=right:$T_n$] (b0) at (-2,-0.5)  {};
		\node (b1) at (-3.5,1.5)  [minimum size=0pt,label=above:$1$]{};
		 \node (b2) at (-2,1.5)  [minimum size=0pt,label=above: ]{};
		\node (b3) at (-0.5,1.5)  [minimum size=0pt,label=above:$n$]{};
		\draw        (b0)--(b1);
		\draw        (b0)--(b2);
		\draw        (b0)--(b3);
		\draw [dotted,line width=0.5pt] (-3,1)--(-2.2,1);
		\draw [dotted,line width=0.5pt] (-1.8,1)--(-1,1);
	\end{tikzpicture}
\end{eqnarray*}
In this means, the action of the  differential operator $\partial$ on generators can be expressed by trees as follows:
\begin{eqnarray*}
	\begin{tikzpicture}[scale=0.6]
		\tikzstyle{every node}=[thick,minimum size=4pt, inner sep=1pt]
		\node(a) at (-4,0.25){\large{$\partial$}};
		\node[circle, fill=black, label=right:$m_n$] (b0) at (-2,-0.5)  {};
		\node (b1) at (-3.5,1.5)  [minimum size=0pt,label=above:$1$]{};
		 \node (b2) at (-2,1.5)  [minimum size=0pt  ]{};
		\node (b3) at (-0.5,1.5)  [minimum size=0pt,label=above:$n$]{};
		\draw        (b0)--(b1);
		\draw        (b0)--(b2);
		\draw        (b0)--(b3);
		\draw [dotted,line width=0.5pt] (-3,1)--(-2.2,1);
		\draw [dotted,line width=0.5pt] (-1.8,1)--(-1,1);
	\end{tikzpicture}
	&
	\begin{tikzpicture}
		\node(0){{$=\sum\limits_{j=2}^{n-1} \sum\limits_{i=1}^{n-j+1}(-1)^{i+j(n-i)}$}};
	\end{tikzpicture}
	&
	\begin{tikzpicture}[scale=0.8]
		\tikzstyle{every node}=[thick,minimum size=4pt, inner sep=1pt]
		\node(e0) at (0,-1.5)[circle, fill=black,label=right:$\ m_{n-j+1}$]{};
		\node(e1) at(-1.5,0){{\tiny$1$}};
		\node(e2-0) at (0,-0.5){{\tiny$i$}};
		\node(e3) at (1.5,0){{\tiny{$n-j+1$}}};
		\node(e2-1) at (0,0.5) [circle,fill=black,label=right: $\ m_j$]{};
		\node(e2-1-1) at (-1,1.5){{\tiny$1$}};
		\node(e2-1-2) at (1, 1.5){{\tiny $j$}};
		\draw [dotted,line width=0.5pt] (-0.7,-0.5)--(-0.2,-0.5);
		\draw [dotted,line width=0.5pt] (0.3,-0.5)--(0.8,-0.5);
		\draw [dotted,line width=0.5pt] (-0.4,1)--(0.4,1);
		\draw        (e0)--(e1);
		\draw         (e0)--(e3);
		\draw         (e0)--(e2-0);
		\draw         (e2-0)--(e2-1);
		\draw        (e2-1)--(e2-1-1);
		\draw        (e2-1)--(e2-1-2);
	\end{tikzpicture}	
\end{eqnarray*}

\begin{eqnarray*}
	\begin{tikzpicture}[scale=0.6]
		\tikzstyle{every node}=[thick,minimum size=4pt, inner sep=1pt]
		\node(a) at (-4,0.25){\large{$\partial$}};
		\node[circle, draw, label=right:$T_n$] (b0) at (-2,-0.5)  {};
		\node (b1) at (-3.5,1.5)  [minimum size=0pt,label=above:$1$]{};
		 \node (b2) at (-2,1.5)  [minimum size=0pt,label=above: ]{};
		\node (b3) at (-0.5,1.5)  [minimum size=0pt,label=above:$n$]{};
		\draw        (b0)--(b1);
		\draw        (b0)--(b2);
		\draw        (b0)--(b3);
		\draw [dotted,line width=0.5pt] (-3,1)--(-2.2,1);
		\draw [dotted,line width=0.5pt] (-1.8,1)--(-1,1);
	\end{tikzpicture}
	&
	\begin{tikzpicture}
		\node(0){{$= \sum\limits_{k=2}^n\sum\limits_{l_1+\cdots+l_k=n \atop l_1, \dots, l_k\geqslant 1}(-1)^{\alpha}$}};
	\end{tikzpicture}
	&
	\begin{tikzpicture}[scale=0.8]
		\tikzstyle{every node}=[thick,minimum size=4pt, inner sep=1pt]
		\node(e0) at (0,-1.5)[circle, fill=black,label=right:$m_{k}$]{};
		\node(e1) at(-1.2,-0.3)[circle, draw, label=left:$T_{l_1}$]{};
		\node(e1-1) at(-2,0.8){};
		\node(e1-2) at (-1,0.8){};
		\node(e3) at (1.2,-0.3)[draw, circle, label=right: $T_{l_k}$]{};
		\node(e3-1) at (1,0.8){};
		\node(e3-2) at (2,0.8){};
		\node(e2-1) at (-0.3,-0.3) [draw,circle,label=right: $T_{l_2}$]{};
		\node(e2-1-1) at (-0.5,1){};
		\node(e2-1-2) at (0.5, 1){};
		\draw [dotted,line width=0.5pt] (0.1,-0.6)--(0.8,-0.6);
		\draw [dotted,line width=0.5pt] (-0.3,0.5)--(0.1,0.5);
		\draw [dotted,line width=0.5pt] (-1.6,0.5)--(-1.2, 0.5);
		\draw [dotted, line width=0.5pt] (1.2,0.5)--(1.6,0.5);
		\draw        (e0)--(e1);
		\draw         (e0)--(e3);
		\draw         (e1)--(e1-1);
		\draw          (e1)--(e1-2);
		\draw         (e0)--(e2-1);
		\draw        (e2-1)--(e2-1-1);
		\draw        (e2-1)--(e2-1-2);
		\draw        (e3)--(e3-1);
		\draw        (e3)--(e3-2);
		
	\end{tikzpicture}	\\
	&\begin{tikzpicture}
		\node(0){{$+\sum\limits_{{\small\substack{2\leqslant p\leqslant n \\ 1\leqslant q\leqslant p
				}}}\sum\limits_{\small\substack{ r_1+\dots+r_q+p-q=n\\r_1, \dots, r_q\geqslant 1\\1\leqslant i\leqslant r_1\\1\leqslant k_1<\dots< k_{q-1}\leqslant p }}(-1)^{\beta}\lambda^{p-q}$}};
	\end{tikzpicture}&
	\begin{tikzpicture}[scale=0.9]
		\tikzstyle{every node}=[thick,minimum size=4pt, inner sep=1pt]
		\node(e0) at (0,-1.5)[circle, draw,label=right:{\footnotesize $\ T_{r_1}$}]{};
		\node(e1) at(-1.8,-0.3){};
		\node(e1-1) at(-2,0.8){};
		\node(e1-2) at (-1,0.8){};
		\node(e2-0) at (0,-0.9){{\tiny$i$}};
		\node(e3) at (1.8,-0.3){};
		\node(e3-1) at (1,0.8){};
		\node(e3-2) at (2,0.8){};
		\node(e2-1) at (0,-0.3) [fill=black,draw,circle,label=right: {\scriptsize $\ \ m_p$}]{};
		\node(e2-1-1) at (-1.9,1){};
		\node(e2-1-2) at(-0.6,0.6){{\tiny$k_1$}};
		\node(e2-1-2-0) at (-1,1.2)[draw, circle,label=left:{\tiny $T_{r_2}$}]{};
		\node(e2-1-2-1) at (-1.4,1.9){};
		\node(e2-1-2-2) at (-0.6,1.9){};
		\node(e2-1-3) at(-0.4,1.5){};
		\node(e2-1-4) at (0.4,1.5){};
		\node(e2-1-5) at (0.6,0.6){{\tiny $k_{q-1}$}};
		\node(e2-1-5-0) at (1,1.2) [circle, draw,label=right: {\tiny $\ T_{r_q}$}]{};
		\node(e2-1-5-1) at (0.6,1.9){};
		\node(e2-1-5-2) at (1.4,1.9){};
		\node(e2-1-6) at (1.9, 1){};
		\draw [dotted,line width=0.5pt] (-0.7,-0.7)--(-0.2,-0.7);
		\draw [dotted,line width=0.5pt] (0.3,-0.7)--(0.8,-0.7);
		\draw [dotted, line width=0.5pt] (-1.3,0.9)--(1.3,0.9);
		\draw        (e0)--(e1);
		\draw [dotted, line width =0.5pt](-1.2,1.7)--(-0.7,1.7);
		\draw [dotted, line width=0.5pt](0.7,1.7)--(1.2,1.7);
		\draw         (e0)--(e3);
		\draw         (e0)--(e2-0);
		\draw         (e2-0)--(e2-1);
		\draw        (e2-1)--(e2-1-1);
		\draw      (e2-1-2)--(e2-1-2-0);
		\draw        (e2-1)--(e2-1-2);
		\draw        (e2-1)--(e2-1-3);
		\draw        (e2-1)--(e2-1-4);	
		\draw        (e2-1)--(e2-1-5);
		\draw         (e2-1-5)--(e2-1-5-0);
		\draw        (e2-1)--(e2-1-6);
		\draw (e2-1-2-0)--(e2-1-2-1);
		\draw (e2-1-2-0)--(e2-1-2-2);
		\draw (e2-1-5-0)--(e2-1-5-1);
		\draw (e2-1-5-0)--(e2-1-5-2);
	\end{tikzpicture}
\end{eqnarray*}

\begin{remark}
It should be noted that the collection spanned by the trees  which  only have black vertices is a dg suboperad of   $\RBinfty$, and this suboperad is exactly the $A_\infty$-operad.
\end{remark}

The following result is the main result of this section, whose proof occupies the rest of this section.
\begin{thm}\label{Thm: Minimal model}
	The dg operad $\RBinfty$ is the minimal model of the operad $\RB$.
\end{thm}

\bigskip

In order to prove Theorem~\ref{Thm: Minimal model},  we are going to construct a quasi-isomorphism of dg operads $\RBinfty\xrightarrow{~}\RB$, where $\RB$ is considered as a dg operad concentrated in degree 0.

Recall that the  operad   $\RB$  is generated by a unary operator $T$ and a binary operator $\mu$ with  operadic relations: $$ \mu\circ_1\mu-\mu\circ_2\mu\quad \mathrm{and}\quad (\mu\circ_1T)\circ_2T-(T\circ_1\mu)\circ_1T-(T\circ_1\mu)\circ_2T-\lambda T\circ_1\mu.$$
\begin{lem} There exists a  natural surjective  map $\phi: \RBinfty\twoheadrightarrow \RB$ of dg operads sending $m_2$ (resp. $T_1$, all other generators) to $\mu$ (resp. $T$, zero), which induces an isomorphism
  $\rmH_0(\RBinfty)  \cong \RB$.

\end{lem}
\begin{proof}
  The degree zero part of $\RBinfty$ is the free graded operad  generated by $\{m_2, T_1\}$. The image of $\partial$  in this  degree zero part is the operadic ideal   generated by $\partial T_2$ and $\partial m_3$. By definition, we have
   \begin{eqnarray*}
 	\partial(m_3)&=&m_2\circ_1 m_2-m_2\circ_2m_2,\\
 	\partial (T_2)&=& T_1\circ_1(m_2\circ_1 T_1)+T_1\circ_1(m_2\circ_2 T_1)+\lambda T_1\circ_1 m_2-(m_2\circ_1T_1)\circ_2 T_1.
 	\end{eqnarray*}
Thus  the map  $\phi: \RBinfty\twoheadrightarrow \RB$ induces the isomorphism  $\rmH_0(\RBinfty)  \cong \RB$.
\end{proof}

\medskip

   To show that  $\phi: \RBinfty\twoheadrightarrow \RB$ is a quasi-isomorphism, we just need to prove that $\rmH_i(\RBinfty)=0$ for all $i\geqslant 1$. This will be achieved by constructing explicitly a homotopy. To this end, we  need to establish a graded  path-lexicographic ordering on $\RBinfty$.

  Each tree monomial gives rise to a path sequence; for details, see \cite[Chapter 3]{BD16}.  More precisely,
 to any tree monomial $\mathcal{T}$ with  $n$ leaves (written as $\mbox{arity}(\mathcal{T})=n$), we can associate   with a sequence   $(x_1, \dots, x_n)$  where  $x_i$ is the word formed by   generators of $\RBinfty$ corresponding to the vertices along the unique path from the root of $\mathcal{T}$  to its $i$-th leaf.

  For two graded tree monomials $\mathcal{T},\mathcal{T}'$, we compare $\mathcal{T},\mathcal{T}'$ in the following way:
 \begin{itemize}
 	\item[(i)] If $\mbox{arity}(\mathcal{T})>\mbox{arity}(\mathcal{T}')$, then  $\mathcal{T}>\mathcal{T}'$;
 	\item[(ii)] if $\mbox{arity}(\mathcal{T})=\mbox{arity}(\mathcal{T}')$, and $\deg(\mathcal{T})>\deg(\mathcal{T}')$, then $\mathcal{T}>\mathcal{T}'$, where $\deg(\mathcal{T})$ is the sum of the degrees of all generators of $\RBinfty$ appearing in    $\mathcal{T}$;
 	\item[(iii)] if $\mbox{arity}(\mathcal{T})=\mbox{arity}(\mathcal{T}')(=n), \deg(\mathcal{T})=\deg(\mathcal{T}')$, then $\mathcal{T}>\mathcal{T}'$ if the path sequences   $(x_1,\dots,x_n), (x'_1,\dots,x_n')$ associated to $\mathcal{T}, \mathcal{T}'$ satisfies $(x_1,\dots,x_n)>(x_1',\dots,x_n')$ with respect to the length-lexicographic order of words induced by $$T_1<m_2<T_2<m_3<\cdots<T_n<m_{n+1}<T_{n+1}<\cdots.$$

 \end{itemize}
 It is ready to see that this is a well order.
Under this order, the leading terms in the expansion of $\partial(m_n), \partial(T_n)$ are the following tree monomials respectively:
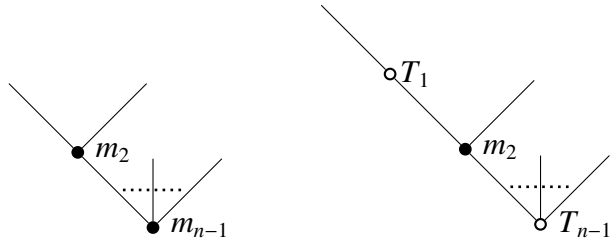
\begin{figure}[h]
	\begin{tikzpicture}[scale=0.8]
		\tikzstyle{every node}=[thick,minimum size=4pt, inner sep=1pt]
		\node(1) at (0,0) [draw, circle, fill=black, label=right:$\ m_{n-1}$]{};
		\node(2-1) at (-1,1) [draw, circle, fill=black, label=right: $\ m_2$]{};
		\node(3-1) at (-2,2){};
		\node(3-2) at (0,2){};
		\node(2-2) at (0,1){};
		\node(2-3) at (1,1){};
		\draw (1)--(2-1);
		\draw  (1)--(2-2);
		\draw  (1)--(2-3);
		\draw (2-1)--(3-1);
		\draw (2-1)--(3-2);
		\draw [dotted,line width=0.5pt](-0.4,0.5)--(0.4,0.5);
	\end{tikzpicture}
\hspace{8mm}
\begin{tikzpicture}[scale=0.75]
	\tikzstyle{every node}=[thick,minimum size=4pt, inner sep=1pt]
	\node(1) at (0,0) [draw, circle, label=right:$\ T_{n-1}$]{};
	\node(2-1) at (-1,1) [draw, circle, fill=black, label=right: $\ m_2$]{};
	\node(3-1) at (-2,2)[draw, circle, label=right: $T_1$]{};
	\node(3-2) at (0,2){};
	\node(2-2) at (0,1){};
	\node(2-3) at (1,1){};
	\node(4) at (-3,3){};
	\draw (1)--(2-1);
	\draw  (1)--(2-2);
	\draw  (1)--(2-3);
	\draw (2-1)--(3-1);
	\draw (2-1)--(3-2);
	\draw (3-1)--(4);
	\draw [dotted,line width=0.5pt](-0.4,0.5)--(0.4,0.5);
\end{tikzpicture}
\end{figure}

	Let $\mathcal{S}$ be a generator of degree $\geqslant 1$ in $ \RBinfty$. Denote the leading monomial of $\partial \mathcal{S}$ by $\widehat{\mathcal{S}}$ and the  coefficient of $\widehat{\mathcal{S}}$ in $\partial$ is written  as $l_\mathcal{S}$.  A tree monomial of the form $\widehat{\mathcal{S}}$ is called typical, so all typical tree {  monomials} are  of the form $$ m_{n-1}\circ_1 m_2  \ \ \mathrm{and}\ \ (T_{n-1}\circ_1 m_2)\circ_1 T_1, $$
which are illustrated above.
 It is  easily seen that the coefficients $l_\mathcal{S}$ are always $\pm 1$.

\medskip

 \begin{defn}\label{Def: effective tree monomials}  A tree monomial $\mathcal{T}$ in $\RBinfty$ is called effective if $\mathcal{T}$ satisfies the following conditions:
 	\begin{itemize}
 		\item[(i)] There exists a typical divisor $\mathcal{T}'=\widehat{\mathcal{S}}$ in $\mathcal{T}$ such that  on the path from the root of $\mathcal{T}'$ to the leftmost leaf $l$ of $\mathcal{T}$ above the root of $\mathcal{T}'$, there are no other typical divisors, and there are no {  vertices} of positive degree on this path except the root of $\mathcal{T}'$ possibly.
 		\item[(ii)] For any leaf $l'$ of $\mathcal{T}$ which lies on the left of $l$, there are no vertices of positive degree and no typical divisors  on the path from the root of $\mathcal{T}$ to $l'$.
 	\end{itemize}
The typical divisor $\mathcal{T}'$ is called the effective divisor of $\mathcal{T}$ and  the leaf $l$ is called the typical leaf of $\mathcal{T}$.
\end{defn}

Morally, the effective divisor of a tree monomial $\mathcal{T}$ is the left-upper-most typical divisor of $\mathcal{T}$.
It can be easily seen that for the effective divisor $\mathcal{T}'$ in $\mathcal{T}$ with effective leaf $l$, any vertex in $\mathcal{T}'$ doesn't belong to the path from root of $\mathcal{T}$ to any leaf $l'$ located on the left of $l$.

\begin{exam}Consider three tree monomials as follows:
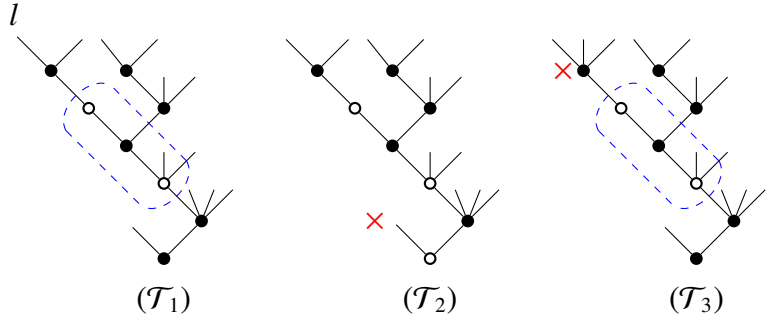
\begin{eqnarray*}
\begin{tikzpicture}[scale=0.9]
	\tikzstyle{every node}=[thick,minimum size=4pt, inner sep=1pt]
	\node(1)at (0,0)[circle, draw, fill=black]{};
	\node(2-1) at (-0.5,0.5){};
	\node(2-2) at (0.5,0.5)[circle,draw, fill=black]{};
	\node(3-1) at (0,1)[circle, draw]{};
	\node(3-2) at (0.3,1){};
	\node(3-3) at (0.7,1){};
	\node(3-4) at (1,1){};
	\node(4-1) at (-0.5,1.5)[circle, draw,fill=black]{};
	\node (4-2) at(0,1.5){};
	\node(4-3) at (0.5,1.5){};
	\node(5-1) at (-1,2)[circle, draw]{};
	\node(5-2) at(0,2)[circle, draw, fill=black] {};
	\node (6-1) at (-1.5, 2.5)[circle,draw, fill=black]{};
	\node(6-2) at (-0.5, 2.5)[circle, draw, fill=black]{};
	\node(6-3) at(0,2.5){};
	\node(6-4) at (0.5,2.5){};
	\node(7-1) at (-2,3)[minimum size=0pt, label=above:$l$]{};
	\node(7-2) at (-1,3){};
	\node(7-3) at(-0.9,3){};
	\node(7-4) at (0,3){};
	\draw (1)--(2-1);
	\draw (1)--(2-2);
	\draw (2-2)--(3-1);
	\draw (2-2)--(3-2);
	\draw (2-2)--(3-3);
	\draw (2-2)--(3-4);
	\draw (3-1)--(4-1);
	\draw (3-1)--(4-2);
	\draw (3-1)--(4-3);
	\draw (4-1)--(5-1);
	\draw (4-1)--(5-2);
	\draw (5-1)--(6-1);
	\draw (6-1)--(7-1);
	\draw (6-1)--(7-2);
	\draw (5-2)--(6-2);
	\draw (5-2)--(6-3);
	\draw (5-2)--(6-4);
	\draw(6-2)--(7-3);
	\draw(6-2)--(7-4);
\draw[dashed,blue](-1.3,1.7) to [in=150, out=120] (-0.7,2.3) ;
\draw[dashed,blue](-1.3,1.7)--(-0.3,0.7);
\draw[dashed,blue] (0.3,1.3)to [in=-30, out=-60] (-0.3,0.7);
	\node[minimum size=0pt,inner sep=0pt,label=below:$(\mathcal{T}_1)$] (name) at (0,-0.3){};
	\draw[dashed, blue](0.3,1.3)--(-0.7,2.3);
\end{tikzpicture}
\hspace{4mm}
\begin{tikzpicture}[scale=0.9]
	\tikzstyle{every node}=[thick,minimum size=4pt, inner sep=1pt]
	\node(1)at (0,0)[circle, draw]{};
	\node(2-1) at (-0.5,0.5)[minimum size=0pt, label=left:$\red \times$]{};
	\node(2-2) at (0.5,0.5)[circle,draw, fill=black]{};
	\node(3-1) at (0,1)[circle, draw]{};
	\node(3-2) at (0.3,1){};
	\node(3-3) at (0.7,1){};
	\node(3-4) at (1,1){};
	\node(4-1) at (-0.5,1.5)[circle, draw,fill=black]{};
	\node (4-2) at(0,1.5){};
	\node(4-3) at (0.5,1.5){};
	\node(5-1) at (-1,2)[circle, draw]{};
	\node(5-2) at(0,2)[circle, draw, fill=black] {};
	\node (6-1) at (-1.5, 2.5)[circle,draw, fill=black]{};
	\node(6-2) at (-0.5, 2.5)[circle, draw, fill=black]{};
	\node(6-3) at(0,2.5){};
	\node(6-4) at (0.5,2.5){};
	\node(7-1) at (-2,3){};
	\node(7-2) at (-1,3){};
	\node(7-3) at (-0.9,3){};
	\node(7-4) at(0,3){};
	\draw(6-2)--(7-3);
	\draw(6-2)--(7-4);
	\draw (1)--(2-1);
	\draw (1)--(2-2);
	\draw (2-2)--(3-1);
	\draw (2-2)--(3-2);
	\draw (2-2)--(3-3);
	\draw (2-2)--(3-4);
	\draw (3-1)--(4-1);
	\draw (3-1)--(4-2);
	\draw (3-1)--(4-3);
	\draw (4-1)--(5-1);
	\draw (4-1)--(5-2);
	\draw (5-1)--(6-1);
	\draw (6-1)--(7-1);
	\draw (6-1)--(7-2);
	\draw (5-2)--(6-2);
	\draw (5-2)--(6-3);
	\draw (5-2)--(6-4);
	\node[minimum size=0pt,inner sep=0pt,label=below:$(\mathcal{T}_2)$] (name) at (0,-0.3){};
\end{tikzpicture}
\hspace{4mm}\begin{tikzpicture}[scale=0.9]
	\tikzstyle{every node}=[thick,minimum size=4pt, inner sep=1pt]
	\node(1)at (0,0)[circle, draw, fill=black]{};
	\node(2-1) at (-0.5,0.5){};
	\node(2-2) at (0.5,0.5)[circle,draw, fill=black]{};
	\node(3-1) at (0,1)[circle, draw]{};
	\node(3-2) at (0.3,1){};
	\node(3-3) at (0.7,1){};
	\node(3-4) at (1,1){};
	\node(4-1) at (-0.5,1.5)[circle, draw,fill=black]{};
	\node (4-2) at(0,1.5){};
	\node(4-3) at (0.5,1.5){};
	\node(5-1) at (-1,2)[circle, draw]{};
	\node(5-2) at(0,2)[circle, draw, fill=black] {};
	\node (6-1) at (-1.5, 2.5)[circle,draw, fill=black,label=left:$\red \times$]{};
	\node(6-2) at (-0.5, 2.5)[circle,draw,fill=black]{};
	\node(6-3) at(0,2.5){};
	\node(6-4) at (0.5,2.5){};
	\node(7-1) at (-2,3){};
	\node(7-2) at (-1.5,3){};
	\node(7-3) at(-1,3){};
	\node(7-4) at(-0.9,3){};
	\node(7-5) at(0,3){};
	\draw(6-2)--(7-4);
	\draw(6-2)--(7-5);
	\draw (1)--(2-1);
	\draw (1)--(2-2);
	\draw (2-2)--(3-1);
	\draw (2-2)--(3-2);
	\draw (2-2)--(3-3);
	\draw (2-2)--(3-4);
	\draw (3-1)--(4-1);
	\draw (3-1)--(4-2);
	\draw (3-1)--(4-3);
	\draw (4-1)--(5-1);
	\draw (4-1)--(5-2);
	\draw (5-1)--(6-1);
	\draw (6-1)--(7-1);
	\draw (6-1)--(7-2);
	\draw (6-1)--(7-3);
	\draw (5-2)--(6-2);
	\draw (5-2)--(6-3);
	\draw (5-2)--(6-4);
	\node[minimum size=0pt,inner sep=0pt,label=below:$(\mathcal{T}_3)$] (name) at (0,-0.3){};

\draw[dashed,blue](-1.3,1.7) to [in=150, out=120] (-0.7,2.3) ;
\draw[dashed,blue](-1.3,1.7)--(-0.3,0.7);
\draw[dashed,blue] (0.3,1.3)to [in=-30, out=-60] (-0.3,0.7);
	
	\draw[dashed, blue](0.3,1.3)--(-0.7,2.3);
\end{tikzpicture}
\end{eqnarray*}

For the three trees displayed above, each has two typical divisors.
\begin{itemize}
	\item $\mathcal{T}_1$ is effective and the divisor in the blue  dashed circle is its effective divisor and $l$ is its effective leaf.
	\item $\mathcal{T}_2$ is not effective, since the first leaf is incident  to a vertex of degree 1, say the root of $\mathcal{T}_2$, which violates Condition (ii) in Definition \ref{Def: effective tree monomials}.
	\item $\mathcal{T}_3$ is not effective since there is a vertex of degree 1 on the path from the root of the typical divisor in the blue  dashed circle to the leftmost leaf above it, which violates Condition (i) in Definition \ref{Def: effective tree monomials}.
\end{itemize}
\end{exam}

Now we are going to construct a homotopy map $\mathfrak{H}: \RBinfty\rightarrow \RBinfty$,  i.e., a graded map of degree $1$,  satisfying  $\partial \mathfrak{H}+\mathfrak{H}\partial =\mathrm{Id}$ in positive degrees.

\begin{defn}
Let $\mathcal{T}$ be an effective tree monomial in $\RBinfty$ and $\mathcal{T}'$ be its effective divisor. Assume that $\mathcal{T}'=\widehat{\mathcal{S}}$, where $\mathcal{S}$ is a generator of positive degree. Then define $$\overline{\mathfrak{H}}(\mathcal{T})=(-1)^\omega \frac{1}{l_\mathcal{S}}m_{\mathcal{T}', \mathcal{S}}(\mathcal{T}),$$ where $m_{\mathcal{T}',\mathcal{S}}(\mathcal{T})$ is the tree monomial obtained from $\mathcal{T}$ by replacing the effective divisor $\mathcal{T}'$ by $\mathcal{S}$, $\omega$ is the sum of degrees of all the vertices on the path from root of $\mathcal{T}'$ to the root of $\mathcal{T}$  (except the root vertex of $\mathcal{T}'$) and on the left of this path .
\end{defn}

Then we define a map $\mathfrak{H}$ of degree 1 on $\RBinfty$ as
\begin{itemize}
	\item[(i)] If $\mathcal{T}$ is not an effective tree monomial, then define $\mathfrak{H}(\mathcal{T})=0$;
	\item[(ii)] If $\mathcal{T}$ is effective, denote by $\overline{\mathcal{T}}$ is obtained from $\mathcal{T}$ by replacing $\mathcal{T}'$ by $\mathcal{T}'-\frac{1}{l_\mathcal{S}}\partial \mathcal{S}$ with $\mathcal{T}'$ being the leading term of $\partial \mathcal{S}$.  Define $\mathfrak{H}(\mathcal{T})=\overline{\mathfrak{H}}(\mathcal{T})+\mathfrak{H}(\overline{\mathcal{T}})$, where, since  each tree monomial in  $\overline{\mathcal{T}}$ is strictly smaller than $\mathcal{T}$, define $\mathfrak{H}(\overline{\mathcal{T}})$ by taking induction on leading terms (this can be done by Lemma~\ref{Lem: homotopy well defined}).
	\end{itemize}
Let's explain more on the definition of $\mathfrak{H}$. Denote $\mathcal{T}$ by $\mathcal{T}_1$. By definition above, $\mathfrak{H}(\mathcal{T})=\overline{\mathfrak{H}}(\mathcal{T}_1)+\mathfrak{H}(\overline{\mathcal{T}_1})$. Since $\mathfrak{H}$ vanishes on non-effective tree monomials, we have $\mathfrak{H}(\overline{\mathcal{T}}_1)=\mathfrak{H}(\sum_{i_2\in I_2} \mathcal{T}_{i_2})$ where $\{\mathcal{T}_{i_2}\}_{i_2\in I_2}$ is the set of effective tree monomials together with their coefficients appearing in the expansion of $\overline{\mathcal{T}_1}$. Then by definition of $\mathfrak{H}$, $\mathfrak{H}(\sum_{i_2\in I_2} \mathcal{T}_{i_2})=\overline{\mathfrak{H}}(\sum_{i_2\in I_2} \mathcal{T}_{i_2})+\mathfrak{H}(\sum_{i_2\in I_2}\overline{\mathcal{T}_{i_2}})$, we have
$$\mathfrak{H}(\mathcal{T})=\overline{\mathfrak{H}}(\mathcal{T}_1)+\overline{\mathfrak{H}}(\sum_{i_2\in I_2} \mathcal{T}_{i_2})+\mathfrak{H}(\sum_{i_2\in I_2}\overline{\mathcal{T}_{i_2}}).$$   Taking induction on leading terms,  $\mathfrak{H}(\mathcal{T})$ is the following series:
\begin{eqnarray}
	\label{Eq:Definition-of-homotopy}
	\mathfrak{H}(\mathcal{T})=\overline{\mathfrak{H}}(\mathcal{T}_1)+\overline{\mathfrak{H}}(\sum_{i_2\in I_2} \mathcal{T}_{i_2})+\overline{\mathfrak{H}}(\sum_{i_3\in I_3} \mathcal{T}_{i_3})+\dots+\overline{\mathfrak{H}}(\sum_{i_n\in I_n} \mathcal{T}_{i_n})+\dots,
	\end{eqnarray}
 where $\{\mathcal{T}_{i_n}\}_{i_n\in I_n}$ is the set of  the effective tree monomials with their nonzero coefficients  appearing in the expansion of $\sum_{i_{n-1}\in I_{n-1}} \overline{\mathcal{T}_{i_{n-1}}}$.
	
	\begin{lem}\label{Lem: homotopy well defined}  For any effective tree monomial $\mathcal{T}$, the expansion of $\mathfrak{H}(\mathcal{T})$ in Equation \eqref{Eq:Definition-of-homotopy} is always a finite sum, i.e, there exists some large integer $n$ such that all tree monomials in $\sum_{i_n\in I_n}\overline{\mathcal{T}_{i_n}}$ are not effective.
		\end{lem}
	\begin{proof}	 It is easy to see that $\max\{\mathcal{T}_{i_k}|i_k\in I_k\}>\max\{\mathcal{T}_{i_{k+1}}|i_{k+1}\in I_{k+1}\}$ for all   $k\geq 1$ (by convention, $i_1\in I_1=\{1\}$), so {  the right-hand side  of} Equation~\eqref{Eq:Definition-of-homotopy} cannot be an infinite sum, as $``>"$ is a well order.

\end{proof}

\begin{lem}{\label{Lem: Induction}}Let $\mathcal{T}$ be an effective tree monomial. Then $\partial \overline{\mathfrak{H}}(\mathcal{T})+\mathfrak{H}\partial(\mathcal{T}-\overline{\mathcal{T}})=\mathcal{T}-\overline{\mathcal{T}}$.
	\end{lem}

\begin{proof}
	We can write $\mathcal{T}$ as a compositions  in the following way:
	$$(\cdots (((((\cdots(X_1\circ_{i_1}X_2)\circ\cdots )\circ_{i_{p-1}}X_p)\circ_{i_p} \widehat{\mathcal{S}})\circ_{j_1}Y_1)\circ_{j_2}Y_2)\cdots)\circ_{j_q}Y_q,$$
	 where $\widehat{\mathcal{S}}$ is the effective divisor of $\mathcal{T}$ and $X_1,\dots, X_p$ are generators of $\RBinfty$ corresponding to the vertices which live on the path from root of $\mathcal{T}$ and root of $\widehat{\mathcal{S}}$ (except the root of $\widehat{\mathcal{S}}$) and on the left of this path in the underlying  tree of  $\mathcal{T}$.
	
	 By definition,
{\small
	 \begin{align*}
	 	\partial \overline{\mathfrak{H}}(\mathcal{T})=&\
\frac{1}{l_\mathcal{S}}(-1)^{\sum\limits_{j=1}^p|X_j|}\partial
\Big((\cdots ((((\cdots ((X_1\circ_{i_1}X_2)\circ_{i_2}\cdots )\circ_{i_{p-1}}X_p)\circ_{i_p} \mathcal{S})\circ_{j_1}Y_1)\circ_{j_2} \cdots)\circ_{j_q}Y_q\Big)\\
=&\ \frac{1}{l_\mathcal{S}}\Big(\sum_{k=1}^p(-1)^{\sum\limits_{j=1}^p|X_j|+\sum\limits_{j=1}^{k-1}|X_j|}\\
	 	&(\cdots ((((\cdots ((\cdots(X_1\circ_{i_1} X_2)\circ_{i_2} \cdots ) \circ_{i_{k-1}}\partial X_k )\circ_{i_k} \cdots)\circ_{i_{p-1}}X_p)\circ_{i_p} \mathcal{S})\circ_{j_1}Y_1)\circ_{j_2} \cdots)\circ_{j_q}Y_q\Big)\\
	 	&+\frac{1}{l_\mathcal{S}}\Big((\cdots ((( (\cdots\cdots(X_1\circ_{i_1}X_2)\circ_{i_2}\cdots )\circ_{i_{p-1}}X_p)\circ_{i_p} \partial \mathcal{S})\circ_{j_1}Y_1)\circ_{j_2} \cdots)\circ_{j_q}Y_q\Big)\\
	 	&+\frac{1}{l_\mathcal{S}}\Big(\sum_{k=1}^q(-1)^{|\mathcal{S}|+\sum\limits_{j=1}^{k-1}|Y_j|}\\
	 	&(\cdots ((\cdots((((\cdots(X_1\circ_{i_1}X_2)\circ\cdots )\circ_{i_{p-1}}X_p)\circ_{i_p} \mathcal{S})\circ_{j_1}Y_1)\circ_{j_2} \cdots ) \circ_{j_k}\partial Y_{k})\circ_{j_{k+1}}\cdots )\circ_{j_q}Y_q\Big)
	 	\end{align*}
	 }
	
 	 and
 	{\small
 	\begin{align*}
 	 	&\mathfrak{H}\partial (\mathcal{T}-\overline{\mathcal{T}})\\
 	 	=\ \ &\frac{1}{l_\mathcal{S}}\mathfrak{H}\partial \Big((\cdots((((\cdots(X_1\circ_{i_1}X_2)\circ_{i_2}\cdots )\circ_{i_{p-1}}X_p)\circ_{i_p} \partial \mathcal{S})\circ_{j_1}Y_1)\circ_{j_2} \cdots)\circ_{j_q}Y_q\Big)\\
 	 	=\ \ &\frac{1}{l_\mathcal{S}}\mathfrak{H}\Big(\sum\limits_{k=1}^p(-1)^{\sum\limits_{j=1}^{k-1}|X_j|}\!\!
 (\cdots((((\cdots  ((\cdots(X_1\circ_{i_1} \cdots )\circ_{i_{k-1}}\partial X_k)\circ_{i_k}\cdots)\circ_{i_{p-1}}X_p)\circ_{i_p} \partial \mathcal{S})\circ_{j_1}Y_1)\circ_{j_2} \cdots)\circ_{j_q}Y_q\Big)\\
 	 	&+\frac{1}{l_\mathcal{S}}\mathfrak{H}\Big(\sum_{k=1}^q(-1)^{\sum\limits_{j=1}^p|X_j|+|\mathcal{S}|-1+\sum_{j=1}^{k-1}|Y_j|}\\
 	 	&(\cdots((\cdots ((((\cdots(X_1\circ_{i_1}X_2)\circ_{i_2}\cdots )\circ_{i_{p-1}}X_p)\circ_{i_p} \partial \mathcal{S})\circ_{j_1}Y_1)\circ_{j_2} \cdots)\circ_{j_k}\partial Y_k)\circ_{j_{k+1}}\cdots\circ_{j_q}Y_q\Big).
 	 	\end{align*}
	}
By the definition of the effective divisor in an effective tree monomial, it can be easily seen that each tree monomial in the expansion of
 \begin{align*} (\cdots((((\cdots  ((\cdots(X_1\circ_{i_1} \cdots )\circ_{i_{k-1}}\partial X_k)\circ_{i_k}\cdots)\circ_{i_{p-1}}X_p)\circ_{i_p} \widehat{\mathcal{S}})\circ_{j_1}Y_1)\circ_{j_2} \cdots)\circ_{j_q}Y_q\end{align*}
 and of
 	\begin{align*}(\cdots((\cdots ((((\cdots(X_1\circ_{i_1}X_2)\circ_{i_2}\cdots )\circ_{i_{p-1}}X_p)\circ_{i_p}  \widehat{\mathcal{S}})\circ_{j_1}Y_1)\circ_{j_2} \cdots)\circ_{j_k}\partial Y_k)\circ_{j_{k+1}}\cdots\circ_{j_q}Y_q
 \end{align*}
 is still effective tree monomial whose  effective divisor is still $\widehat{\mathcal{S}}$.
 Thus we have
 {\small
 \begin{eqnarray*}
 	&&\mathfrak{H}\partial(\mathcal{T}-\overline{\mathcal{T}})\\
 	 	&=&\frac{1}{l_S}\Big(\sum\limits_{k=1}^p(-1)^{\sum\limits_{j=1}^{k-1}|X_j|+\sum\limits_{j=1}^p|X_j|-1}\\
 &&(\cdots((((\cdots  ((\cdots(X_1\circ_{i_1} \cdots )\circ_{i_{k-1}}\partial X_k)\circ_{i_k}\cdots)\circ_{i_{p-1}}X_p)\circ_{i_p} \mathcal{S})\circ_{j_1}Y_1)\circ_{j_2} \cdots)\circ_{j_q}Y_q\Big)\\
 	 	&&+ \frac{1}{l_S}\Big(\sum_{k=1}^q(-1)^{|S|-1+\sum\limits_{j=1}^{k-1}|Y_j|}\\
 	 	&&(\cdots((\cdots ((((\cdots(X_1\circ_{i_1}X_2)\circ_{i_2}\cdots )\circ_{i_{p-1}}X_p)\circ_{i_p}   \mathcal{S})\circ_{j_1}Y_1)\circ_{j_2} \cdots)\circ_{j_k}\partial Y_k)\circ_{j_{k+1}}\cdots\circ_{j_q}Y_q\Big).
 	\end{eqnarray*}}
Take sum of the above expansion, then we get $\partial \overline{\mathfrak{H}}(\mathcal{T})+\mathfrak{H}\partial(\mathcal{T}-\overline{\mathcal{T}})=\mathcal{T}-\overline{\mathcal{T}}$.
	\end{proof}

\begin{prop}The degree $1$ map $\mathfrak{H}$ defined above satisfies $\partial \mathfrak{H}+\mathfrak{H}\partial=\mathrm{Id}$ in all  positive degrees of $\RBinfty$.
	\end{prop}
\begin{proof}
Let $\mathcal{T}$  be an effective tree monomial.  Since the leading term of  $\overline{\mathcal{T}}$ is strictly smaller than $\mathcal{T}$, by induction, we have $$\mathfrak{H}\partial(\overline{\mathcal{T}})+\partial \mathfrak{H}(\overline{\mathcal{T}})=\overline{\mathcal{T}}.$$
	By the definition of $\mathfrak{H}$, $\mathfrak{H}(\mathcal{T})=\overline{\mathfrak{H}}(\mathcal{T})+\mathfrak{H}(\overline{\mathcal{T}})$ and we have $\partial \mathfrak{H}(\mathcal{T})=\partial \overline{\mathfrak{H}}(\mathcal{T})+\partial \mathfrak{H}(\overline{\mathcal{T}})$. Thus, $$\begin{array}{rcl} \partial \mathfrak{H}(\mathcal{T})+\mathfrak{H}\partial (\mathcal{T})&=&\partial \overline{\mathfrak{H}}(\mathcal{T})+\partial \mathfrak{H}(\overline{\mathcal{T}})+\mathfrak{H}\partial(\mathcal{T}-\overline{\mathcal{T}})+\mathfrak{H}\partial(\overline{\mathcal{T}})\\
  &=&\partial \overline{\mathfrak{H}}(\mathcal{T})+\mathfrak{H}\partial(\mathcal{T}-\overline{\mathcal{T}})+\partial \mathfrak{H}(\overline{\mathcal{T}})+\mathfrak{H}\partial(\overline{\mathcal{T}})\\
  &=&\mathcal{T}-\overline{\mathcal{T}}+\overline{\mathcal{T}}\\
  &=&\mathcal{T},\end{array}$$
  where in the third equality we have used the induction hypothesis and $$\partial \overline{\mathfrak{H}}(\mathcal{T})+\mathfrak{H}\partial(\mathcal{T}-\overline{\mathcal{T}})=\mathcal{T}-\overline{\mathcal{T}}$$ by Lemma \ref{Lem: Induction}.

  \medskip

	Next let's prove that for a non-effective tree monomial $\mathcal{T}$, the equation $\partial \mathfrak{H}(\mathcal{T})+\mathfrak{H}\partial(\mathcal{T})=\mathcal{T}$ holds.

By the definition of $\mathfrak{H}$, since $\mathcal{T}$ is not effective, $\mathfrak{H}(\mathcal{T})=0$, thus we just need to check that $\mathfrak{H}\partial(\mathcal{T})=\mathcal{T}$. Since $\mathcal{T}$ has   positive degree, there must exist at least one vertex of positive degree. Let's pick a special vertex $\mathcal{S}$ satisfying  the following conditions:
	\begin{itemize}
		\item[(i)] on the path from $\mathcal{S}$ to the leftmost leaf $l$ of $\mathcal{T}$ above $\mathcal{S}$, there are no other vertices of positive degree;
		\item[(ii)] for any leaf $l'$ of $\mathcal{T}$ located on the left of $l$, the vertices  on the path from the root of $\mathcal{T}$ to $l'$ are all of degree 0.
	\end{itemize}
It is easy to see such a vertex always exists in $\mathcal{T}$. Morally, this vertex is the ``left-upper-most'' vertex of positive degree.
	Then the tree monomial  $\mathcal{T}$ can be written as
	$$(\cdots((((\cdots(X_1\circ_{i_1}X_2)\circ_{i_2}\cdots )\circ_{i_{p-1}}X_p)\circ_{i_p} \mathcal{S})\circ_{j_1}Y_1)\circ_{j_2} \cdots)\circ_{j_q}Y_q,$$
	where $X_1,\dots,X_p$ correspond to the vertices located on the path from the root of $\mathcal{T}$ to $\mathcal{S}$ and on the left of this path in the plane.
	
	By definition,
	{\small	\begin{align*}
		&\mathfrak{H}\partial \mathcal{T}\\
		=&\mathfrak{H}\ {\Huge\{}\\
		&\sum_{k=1}^p(-1)^{\sum\limits_{t=1}^{k-1}|X_t|}
(\cdots(( ((\cdots((\cdots(X_1\circ_{i_1} \cdots)\circ_{i_k}\partial X_k)\circ_{i_{k+1}}\cdots )\circ_{i_{p-1}}X_p)\circ_{i_p} \mathcal{S})\circ_{j_1}Y_1)\circ_{j_2} \cdots)\circ_{j_q}Y_q\\
		&+(-1)^{\sum_{t=1}^p|X_t|}(\cdots((((\cdots(X_1\circ_{i_1}X_2)\circ\cdots )\circ_{i_{p-1}}X_p)\circ_{i_p} \partial \mathcal{S})\circ_{j_1}Y_1)\circ_{j_2} \cdots)\circ_{j_q}Y_q\\
		&+\sum_{k=1}^q(-1)^{\sum\limits_{t=1}^p|X_t|+|\mathcal{S}|+\sum\limits_{t=1}^{k-1}|Y_t|} (\cdots(\cdots((((\cdots(X_1\circ_{i_1}X_2)\circ_{i_2}\cdots )\circ_{i_{p-1}}X_p)\circ_{i_p} \mathcal{S})\circ_{j_1}Y_1)\circ_{j_2} \cdots)\circ_{j_k}\partial Y_{k})\cdots\circ_{j_q}Y_q\\
		&{\Huge \}}
	\end{align*}
}
	
	By the assumption, the divisor consisting of the path from $\mathcal{S}$ to $l$ must be one of the following forms
 
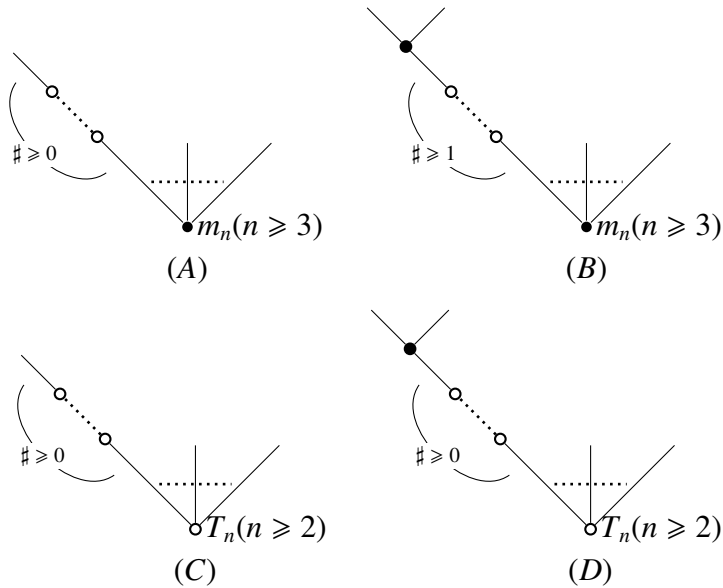
\begin{eqnarray*}
		\begin{tikzpicture}[scale=0.5,descr/.style={fill=white}]
			\tikzstyle{every node}=[thick,minimum size=4pt, inner sep=1pt]
			\node(r) at(0,-0.5)[minimum size=0pt, label=below:$(A)$]{};
			\node(0) at(0,0)[circle, fill=black, label=right:$m_n(n\geqslant 3)$]{};
			\node(1-1) at(-2,2)[circle, draw]{};
			\node(1-2) at(0,2){};
			\node(1-3) at (2,2){};
			\node(2-1) at(-3,3)[circle,draw]{};
			\node(3-1) at(-4,4){};
			\draw(0)--(1-1);
			\draw(0)--(1-2);
			\draw(0)--(1-3);
			\draw[dotted, line width=0.5pt](-0.8,1)--(0.8,1);
			\draw[dotted, line width=0.5pt](1-1)--(2-1);
			\draw(2-1)--(3-1);
			\path[-,font=\scriptsize]
			(-1.8,1.2) edge [bend left=80] node[descr]{{\tiny$\sharp\geqslant 0$}} (-3.8,3.2);
			
		\end{tikzpicture}
		\hspace{4mm}
		\begin{tikzpicture}[scale=0.5,descr/.style={fill=white}]
			\tikzstyle{every node}=[thick,minimum size=4pt, inner sep=1pt]
			\node(r) at(0,-0.5)[minimum size=0pt, label=below:$(B)$]{};
			\node(0) at(0,0)[circle, fill=black, label=right:$m_n(n\geqslant 3)$]{};
			\node(1-1) at(-2,2)[circle, draw]{};
			\node(1-2) at(0,2){};
			\node(1-3) at (2,2){};
			\node(2-1) at(-3,3)[circle,draw]{};
			\node(3-1) at(-4,4)[circle, draw, fill=black]{};
			\node(4-1) at (-5,5){};
			\node(4-2) at(-3,5){};
			\draw(0)--(1-1);
			\draw(0)--(1-2);
			\draw(0)--(1-3);
			\draw[dotted, line width=0.5pt](-0.8,1)--(0.8,1);
			\draw[dotted, line width=0.5pt](1-1)--(2-1);
			\draw(2-1)--(3-1);
			\draw(3-1)--(4-1);
			\draw(3-1)--(4-2);
			\path[-,font=\scriptsize]
			(-1.8,1.2) edge [bend left=80] node[descr]{{\tiny$\sharp\geqslant 1$}} (-3.8,3.2);
		\end{tikzpicture}\\
		\begin{tikzpicture}[scale=0.5,descr/.style={fill=white}]
			\tikzstyle{every node}=[thick,minimum size=4pt, inner sep=1pt]
			\node(r) at(0,-0.5)[minimum size=0pt, label=below:$(C)$]{};
			\node(0) at(0,0)[circle, draw, label=right:$\ T_n(n\geqslant 2)$]{};
			\node(1-1) at(-2,2)[circle, draw]{};
			\node(1-2) at(0,2){};
			\node(1-3) at (2,2){};
			\node(2-1) at(-3,3)[circle,draw]{};
			\node(3-1) at(-4,4){};
			\draw(0)--(1-1);
			\draw(0)--(1-2);
			\draw(0)--(1-3);
			\draw[dotted, line width=0.5pt](-0.8,1)--(0.8,1);
			\draw[dotted, line width=0.5pt](1-1)--(2-1);
			\draw(2-1)--(3-1);
			\path[-,font=\scriptsize]
			(-1.8,1.2) edge [bend left=80] node[descr]{{\tiny$\sharp\geqslant 0$}} (-3.8,3.2);
		\end{tikzpicture}
		\hspace{4mm}
		\begin{tikzpicture}[scale=0.5,descr/.style={fill=white}]
			\tikzstyle{every node}=[thick,minimum size=4pt, inner sep=1pt]
			\node(r) at(0,-0.5)[minimum size=0pt, label=below:$(D)$]{};
			\node(0) at(0,0)[circle, draw, label=right:$\ T_n(n\geqslant 2)$]{};
			\node(1-1) at(-2,2)[circle, draw]{};
			\node(1-2) at(0,2){};
			\node(1-3) at (2,2){};
			\node(2-1) at(-3,3)[circle,draw]{};
			\node(3-1) at(-4,4)[circle, draw, fill=black]{};
			\node(4-1) at (-5,5){};
			\node(4-2) at(-3,5){};
			\draw(0)--(1-1);
			\draw(0)--(1-2);
			\draw(0)--(1-3);
			\draw[dotted, line width=0.5pt](-0.8,1)--(0.8,1);
			\draw[dotted, line width=0.5pt](1-1)--(2-1);
			\draw(2-1)--(3-1);
			\draw(3-1)--(4-1);
			\draw(3-1)--(4-2);
			\path[-,font=\scriptsize]
			(-1.8,1.2) edge [bend left=80] node[descr]{{\tiny$\sharp\geqslant 0$}} (-3.8,3.2);
		\end{tikzpicture}
	\end{eqnarray*}
	By the assumption that $\mathcal{T}$ is not effective and the speciality of the position of $\mathcal{S}$, one can see that the effective tree monomials in $\partial \mathcal{T}$ will only appear in the expansion of
	$$(-1)^{\sum_{t=1}^p |X_t| }(\cdots((((\cdots(X_1\circ_{i_1}X_2)\circ_{i_2}\cdots )\circ_{i_{p-1}}X_p)\circ_{i_p} \partial \mathcal{S})\circ_{j_1}Y_1)\circ_{j_2}  \cdots)\circ_{j_q}Y_q.$$
	
	Consider the tree monomial $$(\cdots((((\cdots(X_1\circ_{i_1}X_2)\circ_{i_2}\cdots )\circ_{i_{p-1}}X_p)\circ_{i_p} \widehat{ \mathcal{S}})\circ_{j_1}Y_1)\circ_{j_2} \cdots)\circ_{j_q}Y_q$$ in $\partial \mathcal{T}$. Then the path connecting root of $\widehat{\mathcal{S}}$ and $l$ must be one of the following forms:
	
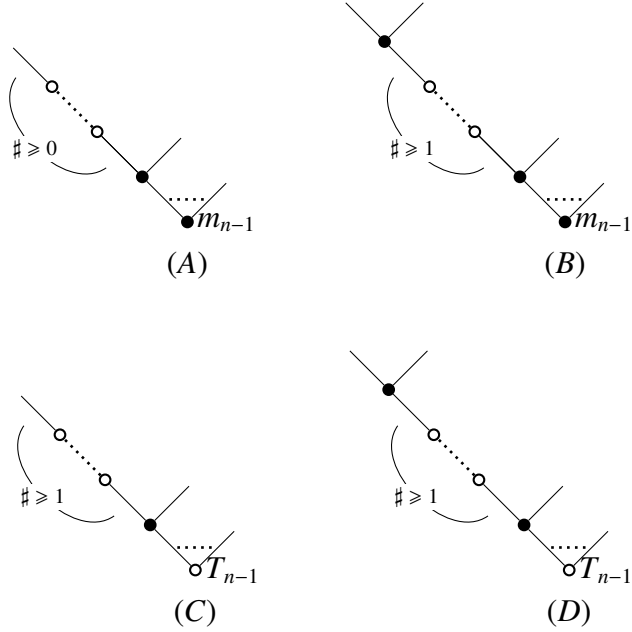
		\begin{eqnarray*}
		\begin{tikzpicture}[scale=0.55,descr/.style={fill=white}]
			\tikzstyle{every node}=[thick,minimum size=4pt, inner sep=1pt]
			\node(r) at(0,-0.5)[minimum size=0pt,label=below:$(A)$]{};
			\node(1) at (0,0)[circle,draw,fill=black,label=right:$\ m_{n-1}$]{};
			\node(2-1) at(-1,1)[circle,draw,fill=black]{};
			\node(2-2) at (1,1) {};
			\node(3-2) at(0,2){};
			\node(3-1) at (-2,2)[circle,draw]{};
			\node(4-1) at (-3,3)[circle,draw]{};
			\node(5-1) at (-4,4){};
			\draw(1)--(2-1);
			\draw(1)--(2-2);
			\draw(2-1)--(3-1);
			\draw(2-1)--(3-2);
			\draw [dotted,line width=1pt](-0.4,0.5)--(0.4,0.5);
			\draw(2-1)--(3-1);
			\draw[dotted,line width=1pt] (3-1)--(4-1);
			\draw(4-1)--(5-1);
			\path[-,font=\scriptsize]
			(-1.8,1.2) edge [bend left=80] node[descr]{{\tiny$\sharp\geqslant 0$}} (-3.8,3.2);
		\end{tikzpicture}
		\hspace{10mm}
		\begin{tikzpicture}[scale=0.55,descr/.style={fill=white}]
			\tikzstyle{every node}=[thick,minimum size=4pt, inner sep=1pt]
			\node(r) at(0,-0.5)[minimum size=0pt,label=below:$(B)$]{};
			\node(1) at (0,0)[circle,draw,fill=black,label=right:$\ m_{n-1}$]{};
			\node(2-1) at(-1,1)[circle,draw,fill=black]{};
			\node(2-2) at (1,1) {};
			\node(3-2) at(0,2){};
			\node(3-1) at (-2,2)[circle,draw]{};
			\node(4-1) at (-3,3)[circle,draw]{};
			\node(5-1) at (-4,4)[circle,draw,fill=black]{};
			\node(6-1) at(-5,5){};
			\node(6-2) at (-3,5){};
			\draw(1)--(2-1);
			\draw(1)--(2-2);
			\draw(2-1)--(3-1);
			\draw(2-1)--(3-2);
			\draw [dotted,line width=1pt](-0.4,0.5)--(0.4,0.5);
			\draw(2-1)--(3-1);
			\draw[dotted,line width=1pt] (3-1)--(4-1);
			\draw(4-1)--(5-1);
			\draw(5-1)--(6-1);
			\draw(5-1)--(6-2);
			\path[-,font=\scriptsize]
			(-1.8,1.2) edge [bend left=80] node[descr]{{\tiny$\sharp\geqslant 1$}} (-3.8,3.2);
		\end{tikzpicture}\\
		\\
		\begin{tikzpicture}[scale=0.55,descr/.style={fill=white}]
			\tikzstyle{every node}=[thick,minimum size=4pt, inner sep=1pt]
			\node(r) at(0,-0.5)[minimum size=0pt,label=below:$(C)$]{};
			\node(1) at (0,0)[circle,draw,label=right:$\ T_{n-1}$]{};
			\node(2-1) at(-1,1)[circle,draw,fill=black]{};
			\node(2-2) at (1,1) {};
			\node(3-1) at(-2,2)[circle,draw]{};
			\node(3-2) at(0,2){};
			\node(4-1) at(-3,3)[circle,draw]{};
			\node(5-1) at (-4,4){};
			\draw(1)--(2-1);
			\draw(1)--(2-2);
			\draw(2-1)--(3-1);
			\draw(2-1)--(3-2);
			\draw[dotted,line width=1pt](3-1)--(4-1);
			\draw(4-1)--(5-1);
			\draw [dotted,line width=1pt](-0.4,0.5)--(0.4,0.5);
			\path[-,font=\scriptsize]
			(-1.8,1.2) edge [bend left=80] node[descr]{{\tiny$\sharp\geqslant 1$}} (-3.8,3.2);
		\end{tikzpicture}
		\hspace{10mm}
		\begin{tikzpicture}[scale=0.55,descr/.style={fill=white}]
			\tikzstyle{every node}=[thick,minimum size=4pt, inner sep=1pt]
			\node(r) at(0,-0.5)[minimum size=0pt,label=below:$(D)$]{};
			\node(1) at (0,0)[circle,draw,label=right:$\ T_{n-1}$]{};
			\node(2-1) at(-1,1)[circle,draw,fill=black]{};
			\node(2-2) at (1,1) {};
			\node(3-1) at(-2,2)[circle,draw]{};
			\node(3-2) at(0,2){};
			\node(4-1) at(-3,3)[circle,draw]{};
			\node(5-1) at (-4,4)[circle,draw,fill=black]{};
			\node(6-1) at(-5,5){};
			\node(6-2) at(-3,5){};
			\draw(1)--(2-1);
			\draw(1)--(2-2);
			\draw(2-1)--(3-1);
			\draw[dotted,line width=1pt](3-1)--(4-1);
			\draw(4-1)--(5-1);
			\draw(5-1)--(6-1);
			\draw(5-1)--(6-2);
			\draw(2-1)--(3-2);
			\draw [dotted,line width=1pt](-0.4,0.5)--(0.4,0.5);
			\path[-,font=\scriptsize]
			(-1.8,1.2) edge [bend left=80] node[descr]{{\tiny$\sharp\geqslant 1$}} (-3.8,3.2);
		\end{tikzpicture}
	\end{eqnarray*}
	So the tree monomial $$(\cdots((((\cdots(X_1\circ_{i_1}X_2)\circ_{i_2}\cdots )\circ_{i_{p-1}}X_p)\circ_{i_p} \widehat{ \mathcal{S}})\circ_{j_1}Y_1)\circ_{j_2} \cdots)\circ_{j_q}Y_q$$ is effective and its effective divisor is exactly $\widehat{\mathcal{S}}$ itself.
	Then we have
	{\small
	\begin{eqnarray*}&&\mathfrak{H}\partial \mathcal{T}\\
		&=&\mathfrak{H}((-1)^{\sum_{t=1}^p |X_t|}(\cdots((((\cdots(X_1\circ_{i_1}X_2)\circ\cdots )\circ_{i_{p-1}}X_p)\circ_{i_p} \partial \mathcal{S})\circ_{j_1}Y_1)\circ_{j_2}Y_2\cdots)\circ_{j_q}Y_q)\\
		&=&l_\mathcal{S}\mathfrak{H}\Big((-1)^{\sum_{t=1}^p |X_t|}(\cdots((((\cdots(X_1\circ_{i_1}X_2)\circ_{i_2}\cdots )\circ_{i_{p-1}}X_p)\circ_{i_p} \widehat{ \mathcal{S}})\circ_{j_1}Y_1)\circ_{j_2} \cdots)\circ_{j_q}Y_q\Big)\\
		&&+\mathfrak{H}\Big((-1)^{\sum_{t=1}^p |X_t|}(\cdots((((\cdots(X_1\circ_{i_1}X_2)\circ_{i_2}\cdots )\circ_{i_{p-1}}X_p)\circ_{i_p} (\partial \mathcal{S}-l_\mathcal{S}\widehat{\mathcal{S}}))\circ_{j_1}Y_1)\circ_{j_2} \cdots)\circ_{j_q}Y_q\Big)\\
		&=& l_\mathcal{S}\overline{\mathfrak{H}}\Big((-1)^{\sum_{t=1}^p |X_t|}(\cdots((((\cdots(X_1\circ_{i_1}X_2)\circ_{i_2}\cdots )\circ_{i_{p-1}}X_p)\circ_{i_p} \widehat{ \mathcal{S}})\circ_{j_1}Y_1)\circ_{j_2} \cdots)\circ_{j_q}Y_q\Big)\\
		&&+l_\mathcal{S}\mathfrak{H}\Big((-1)^{\sum_{t=1}^p |X_t|}(\cdots((((\cdots(X_1\circ_{i_1}X_2)\circ_{i_2}\cdots )\circ_{i_{p-1}}X_p)\circ_{i_p}  (\widehat{\mathcal{S}}-\frac{1}{l_\mathcal{S}}\partial \mathcal{S}))\circ_{j_1}Y_1)\circ_{j_2} \cdots)\circ_{j_q}Y_q\Big)\\
		&&+\mathfrak{H}\Big((-1)^{\sum_{t=1}^p |X_t|}(\cdots((((\cdots(X_1\circ_{i_1}X_2)\circ_{i_2}\cdots )\circ_{i_{p-1}}X_p)\circ_{i_p} (\partial \mathcal{S}-l_\mathcal{S}\widehat{\mathcal{S}}))\circ_{j_1}Y_1)\circ_{j_2} \cdots)\circ_{j_q}Y_q\Big)\\
		&=& l_\mathcal{S}\overline{\mathfrak{H}}\Big((-1)^{\sum_{t=1}^p |X_t|}(\cdots((((\cdots(X_1\circ_{i_1}X_2)\circ_{i_2}\cdots )\circ_{i_{p-1}}X_p)\circ_{i_p} \widehat{ \mathcal{S}})\circ_{j_1}Y_1)\circ_{j_2} \cdots)\circ_{j_q}Y_q\Big)\\
		&=&(\cdots((((\cdots(X_1\circ_{i_1}X_2)\circ_{i_2}\cdots )\circ_{i_{p-1}}X_p)\circ_{i_p} \mathcal{S})\circ_{j_1}Y_1)\circ_{j_2} \cdots)\circ_{j_q}Y_q\\
		&=& \mathcal{T} .
	\end{eqnarray*}
}
	This completes the proof.
  \end{proof}

\textbf{Proof of Theorem~\ref{Thm: Minimal model}:}

We have proved that the natural map $\phi: \RBinfty\twoheadrightarrow \RB$ is a surjective quasi-isomorphism, and  it can be easily seen that the differential $\partial$ on $\RBinfty$ satisfies the conditions $(1)$ and $(2)$ in Definition~\ref{Def: Minimal model of operads}.  We are done!

\medskip

\begin{remark}\label{rem: dotsenko} We are grateful to  Dotsenko who kindly pointed out an alternative proof of Theorem~\ref{Thm: Minimal model}.

Let $\mathfrak{Tr}$ be the free operad generated by a unary operation $T$ and a binary operation $\mu.$ Then $\RB\cong \mathfrak{Tr}/\langle\tilde{G}\rangle$ where $\tilde{G}$ is the defining relations for Rota-Baxter   algebras. In \cite{DK13}, the authors prove that $\tilde{G}$ is a Gr\"obner-Shirshov basis and they constructed the minimal model for the operad $\mathfrak{Tr}/\langle G\rangle$ where $G$ is the set of leading monomials in $\tilde{G}$. Denote this minimal model by ${}_m\RBinfty$.
	
	Now, let's introduce a new ordering $\prec$ on the the set $\mathfrak{Tr}(n)$ as follows: For two tree monomials $\mathcal{T, T'}$ in $\mathfrak{Tr}(n)$,
	\begin{itemize}
		\item[(1)]If $\mbox{weight}(\mathcal{T})<\mbox{weight}( \mathcal{T'})$, then $\mathcal{T}\prec\mathcal{T'}$.
		\item[(2)] If $\mbox{weight}(\mathcal{T})=\mbox{weight}( \mathcal{T'})$, compare $\mathcal{T}$ with $\mathcal{T'}$ via the natural path-lexicographic induced by setting $T\prec \mu.$
	\end{itemize}
With respect to the ordering $\prec$, the set $\mathfrak{Tr}(n)$ becomes a totally ordered set: $$\mathfrak{Tr}(n)=\{x_1\prec x_2\prec x_3\prec\dots\}\cong \mathbb{N}^+.$$ Now, we define a map $\omega:\RBinfty\rightarrow \mathfrak{Tr}$ by replacing vertices $m_n (n\geqslant 2), T_n(n\geqslant 1)$ by $\underbrace{\mu\circ_1\mu\circ_1\mu\dots\circ_1\mu}_{n-1}$ and $ \underbrace{T\circ_1\mu\circ_1T\circ_1\mu\circ_1\dots\circ_1T\circ_1\mu\circ_1T}_{2n-1}$ in $\mathfrak{Tr}$ respectively. Notice that different tree monomials in $\RBinfty$ may have same image in $\mathfrak{Tr}$ under the action of $\omega$.
 Define $\mathcal{F}^n_i $  to be the subspace of $\RBinfty(n)$ spanned by the tree monomials $\mathcal{R}$ with $\omega(\mathcal{R})$ smaller than or equal to $x_i$ with respect to  the ordering $\prec$ on $\mathfrak{Tr}(n)$.
Then we get a bounded below and exhaustive filtration for $\RBinfty(n)$ \[0=\mathcal{F}^n_0\subset \mathcal{F}^n_1\subset\mathcal{F}^n_2\subset \dots .\]
It  can be easily seen that the filtration is compatible with the differential $\partial$ on $\RBinfty(n)$. Moreover, one can prove that there is an isomorphism of complexes
$$\bigoplus_{i\geq0}\mathcal{F}^n_{i+1}/\mathcal{F}^n_i\cong {}_m\RBinfty(n).$$
Since all positive homologies of ${}_m\RBinfty(n)$   vanish, by classical spectral sequence argument, we have that all positive homologies of $\RBinfty(n)$ are trivial.

This provides another proof for   Theorem \ref{Thm: Minimal model}.

  According to the results in \cite{DK13}, the minimal model of the operad $\RB$ can be derived by employing a homotopical perturbation process on the minimal model $_m\RBinfty$ of the operad $_m\RB$. This process yields recursive formulas for both the differential and a homotopy within the minimal model of $\RB$. Our earlier constructions of the differential $\partial$ and homotopy $\mathfrak{H}$ on the differential graded operad $\RBinfty$ offer  explicit and concise formulas for these recursive expressions.

\end{remark}

By the general theory of minimal models of operads,{  the space spanned by} the generators of the minimal model for an operad is exactly the desuspension of its Quillen homology, i.e., the homology of the bar construction of the operad. So for the operad $\RB$,  denote by $\mathrm{B}(\RB)$  the bar construction of $\RB$ and   we have the following result.

\begin{cor} There exists   a quasi-isomorphism of homotopy cooperads between ${\RB^\ac}$ and $ \mathrm{B}(\RB),$  which induces  an isomorphism of
graded collections  $${\mathrm{H}}_\bullet(\mathrm{B}(\RB))\cong {\RB^\ac}.$$

\end{cor}

\section{Homotopy Rota-Baxter   algebras}

Since we have found the operad of ``homotopy Rota-Baxter   algebras  of weight $\lambda$'',  we could now give the definition of homotopy Rota-Baxter   algebras.

\begin{defn}
	Let $(V,d_V)$ be a complex. Then a homotopy Rota-Baxter   algebra of weight $\lambda$ on $V$ is defined to be a morphism of dg operads from $\RBinfty$ to the endomorphism operad $\End_V$.
\end{defn}

Let $(V, d_V)$ be an algebra over the operad $\RBinfty$.
 Still denote by
$  m_n: V^{\ot n}\rightarrow V , n\geqslant 2$ (resp.$ T_n: V^{\ot n}\rightarrow V, n\geqslant 1$)  the image of $m_n\in \RBinfty$ (resp. $T_n\in \RBinfty$).
We also rewrite $m_1=d_V$. Then Equations~\eqref{Eq: defining HRB 1} and \eqref{Eq: defining HRB 2} give
\begin{eqnarray}\label{Eq: stasheff-id}
	\sum_{    i+j+k= n,\atop
i, k\geqslant 0, j\geqslant 1 } (-1)^{i+jk}m_{i+1+k}\circ\Big(\id^{\ot i}\ot m_j\ot \id^{\ot k}\Big)=0
\end{eqnarray} and
\begin{eqnarray} \label{Eq: homotopy RB-operator-version-2}
	   \sum\limits_{ l_1+\dots+l_k=n,\atop
 l_1, \dots, l_k\geqslant 1 } (-1)^{\delta}m_k\circ\Big(T_{l_1}\ot \cdots \ot T_{l_k}\Big)=\sum\limits_{1\leqslant q\leqslant p}\sum\limits_{  r_1+\dots+r_q+p-q=n,\atop
  r_1, \dots, r_q\geqslant 1 } \sum\limits_{ i+1+k=r_1,\atop
   i, k\geqslant 0 }\sum\limits_{  j_1+\dots+j_q+q-1=p,\atop
j_1, \dots, j_q\geqslant 0  }
\end{eqnarray}
	$$ \quad  \quad\quad\quad\quad\quad  (-1)^\eta\lambda^{p-q} T_{r_1}\circ\Big(\id^{\ot i}\ot m_p\circ(\id^{\ot j_1}\ot T_{r_2}\ot \id^{\ot j_2}\ot \cdots \ot T_{r_q}\ot \id^{\ot j_q})\ot \id^{\ot k}\Big),
$$
where
\begin{align*}\delta&=\frac{k(k-1)}{2}+\frac{n(n-1)}{2}+\sum_{j=1}^k(k-j)l_j,\\
	\eta&=\frac{p(p-1)}{2} +\sum_{j=1}^q\frac{r_j(r_j-1)}{2}+k+\sum_{l=2}^q\big(r_l-1\big)\big(i+\sum_{r=1}^{l-1}j_{r}+\sum_{t=2}^{l-1}r_t\big)+pi\\
	&=\frac{n(n-1)}{2}+i+(p+\sum\limits_{j=2}^q(r_j-1))k+\sum\limits_{l=2}^q(r_l-1)(\sum\limits_{r=l}^qj_r+q-l)
\end{align*}

We obtain thus an equivalent definition of  homotopy Rota-Baxter   algebras.
\begin{defn}
	Let $V$ be a graded space. A homotopy Rota-Baxter   algebra    of weight $\lambda$ on $V$  consists of  two families of graded maps
$  m_n: V^{\ot n}\rightarrow V , n\geqslant 1$  and $ T_n: V^{\ot n}\rightarrow V, n\geqslant 1$  with $|m_n|=n-2, |T_n|=n-1$,  subject to
Equations~\eqref{Eq: stasheff-id} and \eqref{Eq: homotopy RB-operator-version-2}.
\end{defn}

Equation~(\ref{Eq: stasheff-id}) is exactly the Stasheff identity   defining  $A_\infty$-algebras \cite{Sta63}. In particular,  the operator $m_1$ is a differential on $V$  and the operator $m_2$ induces an associative algebra structure on the homology   $\rmH_\bullet(V, m_1)$.

\begin{exam}
Expanding  Equation~\eqref{Eq: homotopy RB-operator-version-2}  for small  $n$'s gives the following:
\begin{itemize}
	\item[(i)]
	When $n=1$, $|T_1|=0$ and$$  m_1\circ T_1=T_1\circ m_1,$$
which implies that $T_1: (V, m_1)\to (V, m_1)$ is a chain map;

	\item[(ii)]
	when $n=2$,  $|T_2|=1$ and
 $$\begin{array}{ll}  &m_2\circ(T_1\ot T_1)-T_1\circ m_2\circ (\id\ot T_1)-T_1\circ m_2\circ (T_1\ot \id)-\lambda T_1\circ m_2  \\
	      =&-\partial(T_2)= -\big(m_1\circ T_2+T_2\circ (\id\ot m_1)+T_2\circ(m_1\ot \id)\big),
\end{array}$$
which 	shows that $T_1$ is a Rota-Baxter operator of weight $\lambda$ with respect to $m_2$, but only up to homotopy given by the operator $T_2$.

\end{itemize}
 Observe that for a homotopy Rota-Baxter   algebra $(V;  \{m_n\}_{n\geqslant 1}, \{T_n\}_{n\geqslant 1})$,  its homology $\rmH_\bullet(V, m_1)$ endowed with the operators induced by $m_2$ and $ T_1$ is
{    a } usual  Rota-Baxter   algebra.
\end{exam}

\bigskip

\section{From   the minimal model to the deformation complex and its   $L_\infty$-algebra structure}\label{Section: Linfinty algebras}


In this section, we will use the minimal model $\RBinfty$, or more precisely the Koszul dual  homotopy cooperad ${\RB^{\ac}}$,  to determine the deformation complex as well as the $L_\infty$-algebra structure  on it for Rota-Baxter   algebras of arbitrary weight.

\begin{defn}Let $V$ be a graded space. Introduce an   $L_\infty$-algebra $\frakC_{\RBA}(V)$ associated to $V$ as $\frakC_{\RBA}(V):=\mathbf{Hom}({\RB^\ac}, \End_V)^{\prod}$.
	\end{defn}

Now, let's determine the $L_\infty$-algebra $\frakC_{\RBA}(V)$ explicitly. The sign rules in the homotopy cooperad ${\RB^\ac}$ are complicated, so we need some transformations. Notice that there is a natural isomorphism of operads $$\mathbf{Hom}(\cals, \End_{sV})\cong \End_V.$$
Explicitly, any $f\in \End_V(n)$ corresponds to an element $\tilde{f}\in\mathbf{Hom}(\cals, \End_{sV})(n)$ which is defined as $\big(\tilde{f}(\delta_n)\big)(sv_1\ot\cdots\ot sv_n)=(-1)^{\sum_{k=1}^{n-1}\sum_{j=1}^k|v_j|}(-1)^{(n-1)|f|}sf(v_1\ot \cdots \ot v_n)$ for any $v_1,\dots, v_n\in V$.

 Thus we have the following  isomorphisms of homotopy operads:
\begin{eqnarray*}\mathbf{Hom}\big({\RB^\ac}, \End_V\big)&\cong& \mathbf{Hom}\big({\RB^\ac}, \mathbf{Hom}(\cals, \End_{sV})\big)\\
	&\cong&\mathbf{Hom}\big({\RB^\ac}\ot_{\mathrm{H}}\cals,\End_{sV}\big)\\
	&=&\mathbf{Hom}\big({\mathscr{S}({\RB^\ac})}, \End_{sV}\big)
	\end{eqnarray*}
We obtain $$\frakC_{\RBA}(V)\cong \mathbf{Hom}\big({\mathscr{S}({\RB^\ac})}, \End_{sV}\big)^{\prod}.$$ Recall that ${\mathscr{S}({\RB^\ac})}(n)=\bfk u_n\oplus \bfk v_n$ with $|u_n|=0$ and $|v_n|=1$. By definition $$\mathbf{Hom}\big({\mathscr{S}({\RB^\ac})}, \mathrm{End}_{sV}\big)(n)=\Hom\big(\bfk u_n\oplus \bfk v_n, \Hom((sV)^{\ot n},sV)\big).$$ Each $f\in \Hom((sV)^{\ot n},sV)$ determines bijectively a map $\tilde{f}$ in $\Hom\big(\bfk u_n, \Hom((sV)^{\ot n},sV)\big)$ by imposing  $\tilde{f}(u_n)=f$, and each $g\in \Hom((sV)^{\ot n},V)$ is in bijection with a map $\hat{g}$ in $\Hom\big(\bfk v_n, \Hom((sV)^{\ot n},sV)\big)$ as $\hat{g}(v_n)=(-1)^{|g|}sg$.
Denote  $$\frakC_{\Alg}(V)=\prod\limits_{n\geqslant 1}\Hom((sV)^{\ot n},sV)\quad \mathrm{and}\quad  \frakC_{\RBO}(V)=\prod\limits_{n\geqslant 1}\Hom((sV)^{\ot n},V).$$
In this way, we identify $\frakC_{\RBA}(V)$ with $\frakC_{\Alg}(V)\oplus \frakC_{\RBO}(V)$.
 By the general theory recalled in Subsection~\ref{Subsection: Homotopy  (co)operads}, a direct computation gives the $L_\infty$-algebra structure on $\frakC_{\RBA}(V)$:
\begin{itemize}
	
	\item[(I)] For homogeneous elements $sf, sh\in \mathfrak{C}_{\Alg}(V)$, define $$l_2(sf\ot sh):= [sf, sh]_G\in\mathfrak{C}_{\Alg}(V).$$

	\item[(II)]
	\begin{itemize}	
		\item[(i)] Let $n\geqslant 1$.  For homogeneous elements $sh\in \Hom((sV)^{\ot n},sV)\subset \mathfrak{C}_{\Alg}(V)$ and $g_1,\dots, g_n\in \mathfrak{C}_{\RBO}(V)$,	define $$l_{n+1}(sh\ot g_1\ot \cdots \ot g_n)\in \mathfrak{C}_{\RBO}(V)$$  as :
		\begin{align*}&l_{n+1}(sh\ot g_1\ot \cdots \ot g_n)=\\
			&  \sum_{\sigma\in S_n}(-1)^{\eta}\Big(h\circ(sg_{\sigma(1)}\ot \cdots \ot sg_{\sigma(n)})-(-1)^{(|g_{\sigma(1)}|+1)(|h|+1)}s^{-1}(sg_{\sigma(1)})\big\{sh\big\{sg_{\sigma(2)},\dots,sg_{\sigma(n)}\big\}\big\}\Big),
		\end{align*}
		where $(-1)^{\eta}=\chi(\sigma; g_1,\dots,g_n)(-1)^{n(|h|+1)+\sum\limits_{k=1}^{n-1}\sum\limits_{j=1}^k|g_{\sigma(j)}|}$.

		\item[(ii)]  Let $n\geqslant 2$.  For homogeneous elements $sh\in \Hom((sV)^{\ot n},sV)\subset \mathfrak{C}_{\Alg}(V)$ and $g_1,\dots ,g_m\in \mathfrak{C}_{\RBO}(V)$ with $1\leqslant m\leqslant n-1$, define
		$$l_{m+1}(sh\ot g_1\ot \cdots\ot g_m)\in \mathfrak{C}_{\RBO}(V)$$ to be:
		\[l_{m+1}(sh\ot g_1\ot \cdots\ot g_m)=\sum_{\sigma\in S_m}(-1)^\xi\lambda^{n-m}  s^{-1} (sg_{\sigma(1)})\big\{sh\big\{sg_{\sigma(2)},\dots,sg_{\sigma(m)}\big\}\big\},\]
		where $(-1)^\xi=\chi(\sigma; g_1,\dots,g_m)(-1)^{1+m(|h|+1)+\sum\limits_{k=1}^{m-1}\sum\limits_{j=1}^k|g_{\sigma(j)}|+(|h|+1)(|g_{\sigma(1)}|+1)}$.
	\end{itemize}
	
	\smallskip
	
	\item[(III)]  Let $m\geqslant 1$.  For homogeneous elements $sh\in \Hom((sV)^{\ot n},sV)\subset \mathfrak{C}_{\Alg}(V), g_1,\dots,g_m\in \Hom(T^c(sV),V)\subset \mathfrak{C}_{\RBO}(V)$ with $1\leqslant m\leqslant n$, for $1\leqslant k\leqslant m$,  define $$l_{m+1}(g_1\ot \cdots\ot g_k\ot sh \ot g_{k+1}\ot \cdots\ot g_m)\in \mathfrak{C}_{\RBO}(V)$$ to be
	$$l_{m+1}(g_1\ot \cdots\ot g_k\ot sh \ot g_{k+1}\ot \cdots\ot g_m)=(-1)^{(|h|+1)(\sum\limits_{j=1}^k|g_j|)+k}l_{m+1}(sh\ot g_1\ot \cdots \ot g_m),$$
	where the RHS has been introduced in (III) (i) and (ii).

	\item[(IV)] All other  components of operators $\{l_n\}_{n\geqslant 1}$ vanish.
\end{itemize}

This is exactly the $L_\infty$-structure found in \cite{WZ21} by direct inspections.   We define a filtration $\mathcal{F}_1\supset \mathcal{F}_2\supset \cdots \supset \mathcal{F}_n\supset \cdots$ on $\mathfrak{C}_{\RBA}(V)$ by setting $$ \mathcal{F}_n=\mathfrak{C}_\Alg^{\geqslant n}(V)\oplus\mathfrak{C}_{\RBO}^{\geqslant n}(V), \forall n\geqslant 1,$$ where $$\mathfrak{C}_\Alg^{\geqslant n}(V)=\prod_{k\geqslant n}\Hom((sV)^{\ot k}, sV), \mathfrak{C}_{\RBO}^{\geqslant n}(V)=\prod_{k\geqslant n}\Hom((sV)^{\ot k}, V).$$ It is not difficult to see that the $L_\infty$-algebra $\mathfrak{C}_\RBA(V)$ is weakly filtered with respect to the filtration $\mathcal{F}_\bullet$.

Proposition~\ref{Prop: Linfinity give MC} gives immediately an alternative  definition of homotopy Rota-Baxter   algebras.
\begin{prop}
	A homotopy Rota-Baxter   algebra structure of weight $\lambda$ on a graded space is equivalent to a Maurer-Cartan element in the   weakly filtered  $L_\infty$-algebra $\frakC_{\RBA}(V)$. In particular, when $V$ is concentrated in degree $0$, a Maurer-Cartan element in $\frakC_{\RBA}(V)$ gives a Rota-Baxter   algebra structure of weight $\lambda$ on $V$.
\end{prop}

 \bigskip
		
	\section{Cohomology theory of Rota-Baxter   algebras}\
\label{Cohomology theory of Rota-Baxter   algebras}

Now we introduce the cochain complex of a   Rota-Baxter   algebra   with coefficients in a Rota-Baxter bimodule. We will see that this is exactly the underlying  complex of Rota-Baxter   algebras in Section~\ref{Section: Linfinty algebras}. An explicit example is also included.

\subsection{Cohomology theory}\

Let $(A, \mu)$ be an associative  algebra and $M$ be a bimodule over it. Recall that the Hochschild cochain complex of $A$ with coefficients in $M$ is $$\C^\bullet_{\mathrm{Alg}}(A,M):=\bigoplus\limits_{n=0}^\infty \C^n_{\mathrm{Alg}}(A,M),$$ where $\C^n_{\mathrm{Alg}}(A,M)=\Hom(A^{\ot n},M)$ and the differential $\delta^n: \C^n_{\mathrm{Alg}}(A,M)\rightarrow \C^{n+1}_{\mathrm{Alg}}(A,M)$ is defined as:
$$\delta^n(f)(a_{1, n+1}  )=  (-1)^{n+1} a_1f(a_{2, n+1})+\sum\limits_{i=1}^n(-1)^{n-i+1}f(a_{1, i-1}\ot a_i\cdot a_{i+1}\ot  a_{i+2, n+1})\\
	 + f(a_{1, n})a_{n+1}
$$
for all $f\in \C^n_{\mathrm{Alg}}(A,M), a_1,\dots,a_{n+1}\in A$.
 The cohomology of the Hochschild cochain complex $\C^\bullet_{\mathrm{Alg}}(A,M)$ is called the Hochschild cohomology of $A$ with coefficients in $M$,  denoted by $\mathrm{HH}^\bullet(A,M)$.
When the bimodule $M$ is the regular bimodule $A$ itself, we just denote $\C^\bullet_{\mathrm{Alg}}(A,A)$ by $\C^\bullet_{\mathrm{Alg}}(A)$ and call it the Hochschild cochain complex of associative algebra $(A, \mu)$.  Denote the cohomology $\mathrm{HH}^\bullet(A, A)$ by $\mathrm{HH}^\bullet(A)$, called the Hochschild cohomology of associative algebra $(A,\mu)$.

Let $(A, \mu, T)$ be a
Rota-Baxter   algebra and $(M,T_M)$ be a Rota-Baxter bimodule over it. Recall that
Proposition~\ref{Prop: new RB algebra} and Proposition~\ref{Prop: new-bimodule}  give a new
associative algebra  $A_\star $ and
  a new   Rota-Baxter bimodule  $_\rhd M_\lhd$ over $A_\star $.
 Consider the Hochschild cochain complex of $A_\star $ with
 coefficients in $_\rhd M_\lhd$:
 $$\C^\bullet_{\mathrm{Alg}}(A_\star , {_\rhd
 	M_\lhd})=\bigoplus\limits_{n=0}^\infty \C^n_{\mathrm{Alg}}(A_\star , {_\rhd
 	M_\lhd}).$$
  More precisely,  for $n\geqslant 0$,  $ \C^n_{\mathrm{Alg}}(A_\star , {_\rhd M_\lhd})=\Hom  (A^{\ot n},M)$ and its differential $$\partial^n:
 \C^n_{\mathrm{Alg}}(A_\star ,\  _\rhd M_\lhd)\rightarrow  \C^{n+1}_{\mathrm{Alg}}(A_\star , {_\rhd M_\lhd}) $$ is defined as:
 \begin{align*}&\partial^n(f)(a_{1, n+1}) \\
 =&(-1)^{n+1} a_1\rhd f(a_{2, n+1})+\sum_{i=1}^n(-1)^{n-i+1}f(a_{1, i-1}\ot a_{i}\star  a_{i+1} \ot   a_{i+2, n+1})
  +f(a_{1, n})\lhd a_{n+1}\\
 =&(-1)^{n+1}\Big(T(a_1)f(a_{2, n+1})-T_M\big(a_1f(a_{2, n+1})\big)\Big)\\
&+\sum_{i=1}^n(-1)^{n-i+1}\Big(f(a_{1, i-1}\ot a_iT(a_{i+1})\ot   a_{i+2, n+1})+f(a_{1, i-1} \ot T(a_i)a_{i+1}\ot   a_{i+2, n+1})
\\ &\quad  +\lambda f(a_{1, i-1} \ot a_ia_{i+1}\ot   a_{i+2, n+1})\Big)\\
&+ \Big(f(a_{1, n})T(a_{n+1})-T_M\big(f(a_{1,n})a_{n+1}\big)\Big)
 \end{align*}
 for any $f\in  \C^n_{\Alg}(A_\star ,\  _\rhd M_\lhd)$ and $a_1,\dots,a_{n+1}\in A$.

 \smallskip

 \begin{defn}\label{Def: Cohomology theory of Rota-Baxter operators}
 	Let $A=(A,\mu,T)$ be a Rota-Baxter   algebra of weight $\lambda$ and $M=(M,T_M)$ be a Rota-Baxter bimodule over it. Then the cochain complex $(\C^\bullet_\Alg(A_\star, {_\rhd M_\lhd}),\partial)$ is called the cochain complex of Rota-Baxter operator $T$ with coefficients in $(M, T_M)$,  denoted by $C_{\RBO}^\bullet(A, M)$. The cohomology of $C_{\RBO}^\bullet(A,M)$, denoted by $\mathrm{H}_{\RBO}^\bullet(A,M)$, {  is} called the cohomology of Rota-Baxter operator $T$ with coefficients in $(M, T_M)$.
 	
 	 When $(M,T_M)$ is the regular Rota-Baxter bimodule $ (A,T)$, we denote $\C^\bullet_{\RBO}(A,A)$ by $\C^\bullet_{\RBO}(A)$ and call it the cochain complex of  Rota-Baxter operator $T$, and denote $\rmH^\bullet_{\RBO}(A,A)$ by $\rmH^\bullet_{\RBO}(A)$ and call it the cohomology of Rota-Baxter operator $T$.
 \end{defn}

Let $M=(M,T_M)$ be a  Rota-Baxter bimodule over a Rota-Baxter   algebra of weight $\lambda$ $A=(A,\mu,T)$. Now, let's construct a chain map   $$\Phi^\bullet:\C^\bullet_{\Alg}(A,M) \rightarrow C_{\RBO}^\bullet(A,M),$$ i.e., the following commutative diagram:
\[\xymatrix{
		\C^0_{\Alg}(A,M)\ar[r]^-{\delta^0}\ar[d]^-{\Phi^0}& \C^1_{\Alg}(A,M)\ar@{.}[r]\ar[d]^-{\Phi^1}&\C^n_{\Alg}(A,M)\ar[r]^-{\delta^n}\ar[d]^-{\Phi^n}&\C^{n+1}_{\Alg}(A,M)\ar[d]^{\Phi^{n+1}}\ar@{.}[r]&\\
		\C^0_{\RBO}(A,M)\ar[r]^-{\partial^0}&\C^1_{\RBO}(A,M)\ar@{.}[r]& \C^n_{\RBO}(A,M)\ar[r]^-{\partial^n}&\C^{n+1}_{\RBO}(A,M)\ar@{.}[r]&
.}\]

Define $\Phi^0=\Id_{\Hom(k,M)}=\Id_M$, and for  $n\geqslant 1$ and $ f\in \C^n_{\Alg}(A,M)$,  define $\Phi^n(f)\in \C^n_{\RBO}(A,M)$ as:
\begin{align*}
 &\Phi^n(f)(a_1\ot\cdots \ot a_n) \\
=&f(T(a_1)\ot \cdots \ot T(a_n))\\
&-\sum_{k=0}^{n-1}\lambda^{n-k-1}\sum_{1\leqslant i_1<i_2<\dots<i_k\leqslant n}T_M\circ f(a_{1, i_1-1} \ot T(a_{i_1})\ot a_{i_1+1, i_2-1}\ot T(a_{i_2})\ot \cdots \ot T(a_{i_k})\ot   a_{i_k+1, n}).
\end{align*}

\smallskip

\begin{prop}\label{Prop: Chain map Phi}
	The map $\Phi^\bullet: \C^\bullet_\Alg(A,M)\rightarrow \C^\bullet_{\RBO}(A,M)$ is a chain map.
\end{prop}
This result is equivalent to the fact that the  cochain complex $(\C^\bullet_{\RBA}(A,M), d^\bullet)$  of Rota-Baxter   algebra $(A,\mu,T)$ with coefficients in $(M,T_M)$ in the following definition is a cochain complex, so it follows from Proposition~\ref{Prop: cohomlogy complex as underlying complex of L infinity algebra}.

\begin{defn}\label{Def: Cohomology of RB algebras}
 Let $M=(M,T_M)$ be a  Rota-Baxter bimodule over a Rota-Baxter   algebra of weight $\lambda$ $A=(A,\mu,T)$.  We define the  cochain complex $(\C^\bullet_{\RBA}(A,M), d^\bullet)$  of Rota-Baxter   algebra $(A,\mu,T)$ with coefficients in $(M,T_M)$ to {  be} the negative shift of the mapping cone of $\Phi^\bullet$, that is,   let
\[\C^0_{\RBA}(A,M)=\C^0_\Alg(A,M)  \quad  \mathrm{and}\quad   \C^n_{\RBA}(A,M)=\C^n_\Alg(A,M)\oplus \C^{n-1}_{\RBO}(A,M), \forall n\geqslant 1,\]
 and the differential $d^n: \C^n_{\RBA}(A,M)\rightarrow \C^{n+1}_{\RBA}(A,M)$ is given by \[d^n(f,g)= (\delta^n(f), -\partial^{n-1}(g)  -\Phi^n(f))\]
 for any $f\in \C^n_\Alg(A,M)$ and $g\in \C^{n-1}_{\RBO}(A,M)$.
The  cohomology of $(\C^\bullet_{\RBA}(A,M), d^\bullet)$, denoted by $\rmH_{\RBA}^\bullet(A,M)$,  is called the cohomology of the Rota-Baxter   algebra $(A,\mu,T)$ with coefficients in $(M,T_M)$.
When $(M,T_M)=(A,T)$, we just denote $\C^\bullet_{\RBA}(A,A), \rmH^\bullet_{\RBA}(A,A)$   by $\C^\bullet_{\RBA}(A),  \rmH_{\RBA}^\bullet(A)$ respectively, and call  them the cochain complex, the cohomology of Rota-Baxter   algebra $(A,\mu,T)$ respectively.
\end{defn}
{
It is easy to see that the following proposition holds:
\begin{prop}\label{Prop: Exact sequence}
There is a short exact sequence of complexes:
\begin{eqnarray}\label{Seq of complexes} 0\to s^{-1}\C^\bullet_{\RBO}(A,M)\xrightarrow{i} \C^\bullet_{\RBA}(A,M)\xrightarrow{p} \C^\bullet_{\Alg}(A,M)\to 0\end{eqnarray}
where $i,p$ are the natural inclusion and projection respectively. Therefore we have a long exact sequence of cohomology groups
$$0\to \rmH^{0}_{\RBA}(A, M)\to\mathrm{HH}^0(A, M)\to\rmH^0_{\RBO}(A, M) \to \rmH^{1}_{\RBA}(A, M)\to\mathrm{HH}^1(A, M)\to\cdots$$
$$\cdots\to \mathrm{HH}^p(A, M)\to \rmH^p_{\RBO}(A, M)\to \rmH^{p+1}_{\RBA}(A, M)\to \mathrm{HH}^{p+1}(A, M)\to \cdots.$$
\end{prop}
}


\subsection{Maurer-Cartan characterisation of   Rota-Baxter   algebras}\

Now we verify easily by direct computation that the cohomology theory introduced above is exactly the deformation cohomology of Rota-Baxter   algebras of arbitrary weight.
 \begin{prop}\label{Prop: cohomlogy complex as underlying complex of L infinity algebra}
	Let $(A,\mu,T)$ be a Rota-Baxter   algebra of weight $\lambda$. Twist the $L_\infty$-algebra $\mathfrak{C}_{\RBA}(A)$ by the Maurer-Cartan element corresponding to the Rota-Baxter   algebra structure $(A,\mu,T)$, then its  underlying complex is exactly $s\C^\bullet_{\RBA}(A)$, the shift of the cochain complex of Rota-Baxter   algebra $(A, \mu, T)$, introduced in Definition~\ref{Def: Cohomology of RB algebras}.
\end{prop}

Although $\mathfrak{C}_{\RBA}(A)$ is an $L_\infty$-algebra, the next result shows that once the associative algebra structure $\mu$ over $A$ is fixed, the graded space  $\mathfrak{C}_{\RBO}(A)$, which, after twisting procedure, controls deformations of Rota-Baxter operators, is a genuine differential graded Lie algebra.

\begin{prop}\label{Prop: Cochain of operators is DGLA}
	Let $(A,\mu)$ be an associative algebra. Then the graded space  $\mathfrak{C}_{\RBO}(A)$ can be endowed with a   dg Lie algebra structure, and the set of its Maurer-Cartan elements is in bijection with the set of Rota-Baxter operators of weight $\lambda$ on $(A,\mu)$.  Given a Rota-Baxter operator $T$ on associative algebra $(A, \mu)$, the underlying complex of the twisted dg Lie algebra $\mathfrak{C}_{\RBO}(A)$ by the corresponding Maurer-Cartan element   is exactly   the cochain complex of Rota-Baxter operator $C_{\RBO}^\bullet(A)$.
\end{prop}
\begin{proof}Consider $A$ as graded space concentrated in degree 0. Define $m=- s\circ \mu\circ (s^{-1}\ot s^{-1}): (sA)^{\ot 2}\rightarrow sA$. Then  $\alpha=(m,0)$ is naturally a Maurer-Cartan element in $L_\infty$-algebra $\mathfrak{C}_{\RBA}(A)$.
	By the construction of   $l_n$ on $\mathfrak{C}_{\RBA}(A)$,    the graded subspace $\mathfrak{C}_{\RBO}(A)$ is closed under the action of operators $\{l_n^\alpha\}_{n\geqslant 1}$. Since the arity of $m$ is 2,  the restriction of $l_n^\alpha$ on $\mathfrak{C}_{\RBO}(A)$ is $0$ for $n\geqslant 3$. Thus $(\mathfrak{C}_{\RBO},\{l_n^\alpha\}_{n=1,2})$ forms a dg Lie algebra.
	More explicitly, for $f\in \Hom((sA)^{\ot n},A),g\in \Hom((sA)^{\ot k},A)$,
	\begin{align*}
		l_1^\alpha(f)=&-l_2(m\ot f)=-(-1)^{|f|+1}\lambda f\big\{ m\big\} =(-1)^n\lambda f\big\{ m\big\}\\
		l_2^\alpha(f\ot g)=&l_3(m\ot f\ot g)\\
		=&(-1)^{|f|}\Big(s^{-1}m\circ (sf\ot sg)-(-1)^{|f|+1}f \big\{m \big\{sg\big\} \big\}\Big)\\
		&+(-1)^{|f||g|+1+|g|}\Big(s^{-1}m\circ(sg\ot sf)-(-1)^{|g|+1}g\big\{m\big\{ sf\big\} \big\}\Big)\\
		=&(-1)^n s^{-1}m\circ (sf\ot sg)+ f \big\{m\big\{ sg\big\} \big\} \\
		&+(-1)^{nk+1+k} s^{-1}m\circ(sg\ot sf)-(-1)^{nk}g\big\{m\big\{ sf\big\}\big\} .
	\end{align*}
	
	Since $A$ is concentrated in degree 0, we have $\mathfrak{C}_{\RBO}(A)_{-1}=\Hom(sA,A)$. Take an element  $\tau\in \Hom(sA,A)_{-1}$.  Then $\tau$
	satisfies the  Maurer-Cartan equation:
	$$l_1^\alpha(\tau)-\frac{1}{2}l_2^{\alpha}(\tau\ot \tau)=0,$$
	if and only if
	$$-\lambda \tau\circ m+s^{-1}m\circ (s\tau\ot s\tau)-\tau\circ(m\big\{s\tau\big\})=0.$$
	Define $T=\tau\circ s:A\rightarrow A$. The above equation is exactly the statement  that $T$ is a Rota-Baxter operator of weight $\lambda$ on associative algebra $(A,\mu)$.

Now let $T$ be a Rota-Baxter operator on    associative algebra $(A, \mu)$. By the first statement, it corresponds to a Maurer-Cartan element $\beta$ in the dg Lie algebra $(\mathfrak{C}_{\RBA}(A), l_1^\alpha, l_2^\alpha)$. More precisely, $\beta\in \mathfrak{C}_{\RBO}(A)_{-1}=\Hom(sA,A)$ is defined to be
$\beta=T\circ s^{-1}$.
 For $f\in   \Hom((sA)^{\ot n}, A)$, we compute $(l_1^\alpha)^\beta(f)$.  In fact,
 $$\begin{array}{rcl} (l_1^\alpha)^\beta(f)&=& l_1^\alpha(f)-l_2^\alpha(\beta\ot f)\\
 &=& (-1)^n \lambda f\{m\}+s^{-1} m \circ (sf\ot s\beta)-\beta\{m\{sf\}\}\\
 &&+s^{-1}m \circ (sf\ot s\beta)+(-1)^n f\{m\{s\beta\}\}, \end{array}$$
which  corresponds to  $ \partial^{n}(\hat{f})$ as defined in Definition~\ref{Def: Cohomology theory of Rota-Baxter operators}. So
 the underlying complex of the twisted dg Lie algebra $\mathfrak{C}_{\RBO}(A)$ by the corresponding Maurer-Cartan element $\beta$   is exactly   the cochain complex of Rota-Baxter operator $C_{\RBO}^\bullet(A)$.
	
\end{proof}

\begin{remark}      The defining  equation   \eqref{Eq: Rota-Baxter relation in terms of maps} of a Rota-Baxter operator    is quadratic-linear  in the Rota-Baxter operator if we consider the associative product as part of the underlying structure. This fact seems to suggest that Rota-Baxter operators are ``relatively Koszul'' with respect to the associative product, which  may explain why the deformation complex of a Rota-Baxter operator  in  Proposition~\ref{Prop: Cochain of operators is DGLA} is a   differential graded Lie algebra instead of an $L_\infty$-algebra. However, to the best of our knowledge, there does  not exist a theory  of relative Koszul duality for operads in the literature and we are working on the project to develop such theory.  While it  is still out of reach, we give a computational proof of Proposition~\ref{Prop: Cochain of operators is DGLA} instead of a conceptual one.

\end{remark}

\medskip

\subsection{A curious example}\

We conclude this section and also this paper by an   example of Rota-Baxter algebra whose cohomology is explicitly computed.

\begin{exam}Let $A$ be the polynomial ring $\bfk[x]$ in one variable and let   $T$ be the indefinite integral operator on  $A$, i.e., $T(x^{n})=\frac{1}{n+1}x^{n+1}$ for any $n\geqslant 0$. Then the algebra $A$ endowed with the operator $T$ is a Rota-Baxter algebra of weight $0$. We will show that
\[\rH^n_{\mathrm{RBA}}(A)=0, \forall n\geqslant 0.\]

By Proposition \ref{Prop: new RB algebra}, the Rota-Baxter operator $T$ on $A$ induces a new multiplication $\star$:
	\[x^m\star x^n=\big(\frac{1}{m+1}+\frac{1}{n+1}\big)x^{m+n+1}, \forall m, n\geqslant 0.\]
	Actually, we have an isomorphism of non-unital associative algebras $$A_\star\cong t\bfk[t], x^n\mapsto \frac{1}{n+1}t^{n+1}$$
where $t\bfk[t]$ is the set of  all polynomials in a variable $t$ without    constant terms, considered as a non-unital associative algebra via the usual multiplication. 
 Thus   $\bfk\oplus A_\star\cong A$ as unital associative algebras. In this way, the cochain complex $\C^\bullet_{\RBO}(A)$ of the Rota-Baxter operator is exactly the normalized Hochschild cochain complex of $A$ with coefficients in the  bimodule ${}_\rhd
	A_\lhd$.  Since $A$ is of global dimension $1$,   we have
	$$\rH_{\mathrm{RBO}}^n(A)=\mathrm{H}^n_\Alg(A_\star,{_\rhd A_\lhd})\cong\left\{\begin{array}{ll} A & n=0,1\\
		0 & n\geqslant 2, \end{array}\right.$$
We need to realise these isomorphisms on the complex $\C^\bullet_{\RBO}(A)$:
$$0\to  \Hom(\bfk, {}_\rhd
	A_\lhd)\stackrel{\partial^0}{\to} \Hom(A, {}_\rhd
	A_\lhd)\stackrel{\partial^1}{\to} \Hom(A^{\otimes 2}, {}_\rhd
	A_\lhd)\to \cdots.$$
It is easy to see that  $\partial^0$ vanishes, so $\rH^0_{\mathrm{RBO}}(A)=\Hom(\bfk, {}_\rhd
	A_\lhd)\cong A$ and $\rH^1_{\mathrm{RBO}}(A)=\Ker(\partial^1)$. Let's make the isomorphism $\rH^1_{\mathrm{RBO}}(A)\cong A$ explicit.
In fact,
for any $f\in \C_{\mathrm{RBO}}^1(A)$, $\partial^1(f)=0$ if and only if $\partial(f)(x^m\ot x^n)=0$ for any $m,n\geqslant 0$. Then   $f$ must fulfill the equation
\[T(x^m)f(x^n)-T(x^mf(x^n))-f(x^{m}\star x^n)+f(x^m)T(x^n)-T(f(x^m)x^n)=0, \forall m,n\geqslant 0.\]
Taking $m=0$, we have
\[f(x^{n+1})=\frac{n+1}{n+2}\Big(xf(x^n)-T(f(x^n))+\frac{1}{n+1}f(1)x^{n+1}-T(f(1)x^n)\Big)\]
So  the map $f$ is uniquely determined by $f(1)\in A$. Conversely, starting from any element $a\in A$ and using the above inductive formula, we can define a map $f_a:A\rightarrow A$ belonging to $\Ker(\partial_1)$ via
\[   f_a(1)=  a, f_a(x^n)=x^na-nT\big(x^{n-1}a\big), \forall n\geqslant 1.\]

Now we can compute the cohomology of the Rota-Baxter algebra $(A,\mu,T)$. Recall  that $\Phi^0=\Id$, so   $\rH_{\mathrm{RBA}}^0(A)=0$.
  According to {\rm Proposition~\ref{Prop: Exact sequence}}, we have the long exact sequence:
{\small \[0=\rH^0_{\mathrm{RBA}}(A)\rightarrow \mathrm{H}_\Alg^0(A)\stackrel{\Id}{\rightarrow} \rH^0_{\mathrm{RBO}}(A)\rightarrow \rH^1_{\mathrm{RBA}}(A)\rightarrow \mathrm{HH}^1(A)\stackrel{\overline{\Phi^1}}{\rightarrow} \rH^1_{\mathrm{RBO}}(A)\rightarrow \rH^2_{\mathrm{RBA}}(A)\rightarrow \mathrm{H}_\Alg^2(A)=0 \]}
where $\overline{\Phi^1}:\mathrm{HH}^1(A)\rightarrow \rH^1_{\mathrm{RBO}}(A)$ is the induced map by $\Phi^1$,
 and $\mathrm{H}_\Alg^i(A)=\rH_{\mathrm{RBO}}^i(A)=0$ for all $i\geqslant 2$, we have
$$\rH_{\mathrm{RBA}}^n(A)=\left\{\begin{array}{ll} 0& n=0\ \mbox{or}\  n\geqslant 3,\\
	\Ker\overline{\Phi^1},& n=1, \\
\mathrm{Coker}(\overline{\Phi^1}),&n=2
\end{array}\right.$$
  Let's describe $\Ker(\overline{\Phi^1})$ and $\cok(\overline{\Phi^1})$ more precisely. Recall that there is an isomorphism  \[\mathrm{H}_\Alg^1(A)=\mathrm{Der}(A)\cong A,\]
where $\mathrm{Der}(A)$ is the space of derivations on $A$. Given a derivation $d$ on $A$, it is uniquely determined by $d(x)\in A$, providing the above isomorphism $\mathrm{Der}(A)\cong A$. Given a derivation $d$ on $A$,
\[\Phi^1(d)=d\circ T-T\circ d\in \Ker(\partial^1),\]
the corresponding element in $A$ for this cocycle is
$$(d\circ T-T\circ d)(1)=d\circ T(1),$$ as $d(1)=0$. If $d\in \Ker(\overline{\Phi^1})$, we must have $d\circ T(1)=d(x)=0$, which implies that $d=0$. So $\overline{\Phi^1}$ is injective and $\rH^1_\RBA(A)=\Ker\overline{\Phi^1}=0$. And $d(x)$ can be any element of  $A$, so $\overline{\Phi^1}$ is also full and $\rH_{\mathrm{RBA}}^2(A)=\mathrm{Coker}(\overline{\Phi^1})=0$.

Therefore,  we have proven that
\[\rH^n_{\mathrm{RBA}}(A)=0,\forall n\geqslant 0.\]

Our computation shows that this Rota-Baxter algebra has no nontrivial deformations when deforming simultaneously both the associative algebra structure and the Rota-Baxter operator, although there does exist nontrivial deformations when deforming only the Rota-Baxter operator.
\end{exam}

\bigskip

 \textbf{Acknowledgements:}  This work was   supported  by the National Natural Science Foundation of China (No.  12071137), by  Key Laboratory of Ministry of Education,  by  Shanghai Key Laboratory of PMMP  (No.  22DZ2229014),   and by Fundamental Research Funds for the Central Universities.

   The authors are grateful to Jun Chen, Xiaojun Chen, Vladimir Dotsenko,  Li Guo, Yunnan Li,  Zihao Qi, Martin Markl,  Yunhe Sheng, Rong Tang etc       for many useful comments.  Apurba Das draw our attention to his papers \cite{Das21, DM20}.
   Vladimir   Dotsenko kindly pointed out  an alternative proof of Theorem~\ref{Thm: Minimal model} which is reproduced in Remark~\ref{rem: dotsenko}.  Li Guo  read carefully  part of the previous version \cite{WZ21} and gave  very detailed suggestions and  kind encouragements. Yunhe Sheng kindly provided some comments about the organisation of this paper.
  We are very grateful to these researchers for their interests and comments.

  We would like to express our sincere gratitude to an anonymous referee for her/his comments which   led  to a substantial revision of   the paper.
  
 Monsieur Le Professeur Alexander Zimmermann a eu 60 ans au d\'ebut de l'ann\'ee 2024. Le deuxi\`eme auteur est tr\`es heureux d'avoir cette opportunit\'e de lui exprimer sa plus sinc\`ere gratitude pour  sauver sa carri\`ere math\'ematique il y a 19 ans. On soulahite bonne chant\'e et bonne continuation dans les ann\'ees \`a venir pour Monsieur Zimmermann.






\bigskip

\end{document}